\documentclass{amsart}

\usepackage{amsmath,amsfonts,amssymb}
\usepackage{xcolor,mhsetup}
\usepackage[extra,safe]{tipa}
\usepackage{mathtools}
\usepackage[left=2cm,right=2cm,top=2cm,bottom=2cm]{geometry}
\usepackage[applemac]{inputenc}
\usepackage{mathrsfs}

\theoremstyle{definition}
\newtheorem{theorem}{Theorem}[section]
\newtheorem{corollary}{Corollary}[section]

\newtheorem{prop}{Proposition}[section]
\newtheorem{definition}[theorem]{Definition}
\newtheorem{example}[theorem]{Example}
\newtheorem{lem}{Lemma}[section]

\theoremstyle{remark}
\newtheorem{remark}[theorem]{Remark}

\numberwithin{equation}{section}

\def \bP{\mathbf{P}}
\def \bQ{\mathbf{Q}}
\def \bA{\boldsymbol {A}}

\def \bd{\mathbf{d}}

\def \bmu{\boldsymbol{\mu}}

\def\my_c{c_\infty}

\newcommand{\mynewtheorem}[2]{
  \newaliascnt{#1}{dummy}
  \newtheorem{#1}[#1]{#2}
  \aliascntresetthe{#1}
  \expandafter\def\csname #1autorefname\endcsname{#2}
}

\newcommand{\be}{\begin{equation}}
\newcommand{\ee}{\end{equation}}
\newcommand{\bde}{\begin{displaymath}}
\newcommand{\ede}{\end{displaymath}}
\newcommand{\beq}{\begin{eqnarray*}}
\newcommand{\eeq}{\end{eqnarray*}}
\newcommand{\beqa}{\begin{eqnarray}}
\newcommand{\eeqa}{\end{eqnarray}}
\newcommand{\bel }{\left\{\begin{array}{ll}}
\newcommand{\eel}{\cr \end{array} \right.}

\newcommand{\seq}[1]{{\lbrace #1 \rbrace}}

\newcommand{\dcb}{\begin{array}{lll}}
\newcommand{\dce}{\end{array}}
\newcommand{\ebe}{\begin{enumerate}\setlength{\baselineskip}{13pt}\setlength{\parskip}{0pt}}
\newcommand{\dbe}{\end{enumerate}}

\newcommand{\E}{\mathcal{E}}

\def\F{{\cal F}}

\def\pp{{\cal P}}

\def\rr{{\mathbb R}}
\def\pp{\mathcal{P}(\mathbb{R}^d)}

\def\P{{\mathbb P}}

\def\E{\mathbb{E}}

\def\I{\mathsf{1}}

\def\leftB{[\![}
\def\rightB{]\!]}

\newcommand \A[1]{{\bf (#1)}}

\def\F{{\mathcal F}}

\def\R{{\mathbb{R}} }
\def\N{{\mathbb{N}} }
\def\E{{\mathbb{E}}  }

\def\P{{\mathbb{P}}  }

\def\I{{\mathbf{1}}}

\def\bint#1^#2{\displaystyle{\int_{#1}^{#2}}}
\def\bsum#1^#2{\displaystyle{\sum_{#1}^{#2}}}

\def\xdt_#1{X_#1(\Delta t)}

\newtheorem{REM}{Remark}[section]

\def\0{{\mathbf{0}}}

\def\d{{{\rm det}}}

\def\db{{\mathbf{d}}}

\begin{document}

\title{ Well-posedness of some non-linear stable driven SDEs}

%\author{}
%\address{Paul-Eric Chaudru de Raynal,
% LAMA, UMR 5127, UniversitÃÂÃÂÃÂÃÂÃÂÃÂÃÂÃÂÃÂÃÂÃÂÃÂÃÂÃÂÃÂÃÂÃÂÃÂÃÂÃÂÃÂÃÂÃÂÃÂÃÂÃÂÃÂÃÂÃÂÃÂÃÂÃÂ Savoie Mont Blanc, Campus Scientifique, 73376 Le Bourget-du-Lac Cedex}
%\curraddr{}
%\email{pe.deraynal@univ-savoie.fr}

\author{Noufel Frikha}
\address{Noufel Frikha, Universit\'e de Paris, Laboratoire de Probabilit\'es, Statistiques et Mod\'elisation, F-75013 Paris, France.}
%\curraddr{}
\email{frikha@math.univ-paris-diderot.fr}

%\curraddr{}

\author{Valentin Konakov}
\address{Valentin Konakov,
Laboratory of Stochastic Analysis, Higher School of Economics,
Pokrovsky boulevard 11, Moscow, Russian federation. }
%\curraddr{}
\email{vkonakov@hse.ru}

\author{St\'ephane Menozzi}
\address{St\'ephane Menozzi, Laboratoire de Mod\'elisation Math\'ematique d'Evry (LaMME), UMR CNRS 8070, Universit\'e d'Evry Val d'Essonne, Universit\'e Paris-Saclay, 23 Boulevard de France 91037 Evry and 
Laboratory of Stochastic Analysis, Higher School of Economics,
31 Shabolovka, Moscow, Russian federation. }
%\curraddr{}
\email{stephane.menozzi@univ-evry.fr}

%\author{Libo Li}
%\address{Libo Li, Department of Mathematics and Statistics, University of New South Wales, Sydney, Australia}
%\curraddr{}
%\email{libo.li@unsw.edu.au }

%    General info

\date{\today}


\begin{abstract} We prove the well-posedness of some non-linear stochastic differential equations in the sense of McKean-Vlasov driven by non-degenerate symmetric $\alpha$-stable L\'evy processes with values in $\rr^d$ under some mild H\"older regularity assumptions on the drift and diffusion coefficients with respect to both space and measure variables. The methodology developed here allows to consider unbounded drift terms even in the so-called super-critical case, i.e. when the stability index $\alpha \in (0,1)$. New strong well-posedness results are also derived from the previous analysis.
\end{abstract}

\maketitle

\section{Introduction}

In this article, we are interested in some non-linear stochastic differential equations (SDEs for short) with dynamics
\begin{equation}
\label{SDE:MCKEAN}
X_t = \xi + \int_0^t b(s,X_s, [X_s]) ds +\int_0^t  \sigma(s,X_{s^-},[X_s]) dZ_s, \quad [\xi] = \mu \in {\mathcal P}({ \R^d}) 
\end{equation}

\noindent driven by a $d$-dimensional $\alpha$-stable process $Z$, $\alpha \in (0,2)$, with coefficients $b:\rr_+\times \rr^d \times \pp\rightarrow \rr^d$ and $\sigma: \rr_+ \times \rr^d \times \pp \rightarrow \rr^d \otimes \rr^d$. Here and throughout the article, we will denote by ${\mathcal P}(\R^d) $ the space of probability measures on $\R^d $ and by $[\theta]$ the law of a random variable $\theta$. 

This type of dynamics commonly referred to in the literature as McKean-Vlasov, distribution dependent or mean-field SDEs naturally appears as the limit equation of an individual particle evolving within a large system of particles interacting through its empirical measure, as the size of the population grows to infinity, following the well-known propagation of chaos phenomenon. We refer to the original works by Kac \cite{kac1956} in kinetic theory, by McKean \cite{McKean:1966}, \cite{mckean1967propagation} in non-linear parabolic partial differential equations and by Sznitman \cite{Sznitman} in the diffusive case $\alpha=2$. 

The well-posedness theory in the weak or strong sense of such dynamics has been an active research topic during the last decades, see e.g. Funaki \cite{Funaki1984}, Oelschlaeger \cite{oelschlager1984}, G\"artner \cite{gartner}, \cite{Sznitman}, Jourdain \cite{jourdain:1997} and more recently, Li and Min \cite{Li:min:2}, Chaudru de Raynal \cite{CHAUDRUDERAYNAL2019}, Mishura and Veretennikov \cite{mishura:veretenikov}, Lacker \cite{la:18}, Chaudru de Raynal and Frikha \cite{chau:frik:18}, R\"ockner and Zhang \cite{2018arXiv180902216R} for a short sample in the diffusive setting. Let us also mention the book of Kolokolstov \cite{Kolokolstov:book} for some related works on kinetic equations. In the current strictly $\alpha$-stable context $\alpha \in (0,2)$, we mention the work of Jourdain \textit{et al.} \cite{jour:mele:woyc:08} with applications to the fractional porous media equation on the real-line and of Jourdain \textit{et al.} \cite{Jourdain:meleard:woyczynski:2005}, \cite{Jourdain:meleard:woyczynski:2005b} for a probabilistic approach to some non-linear equations involving the fractional Laplacian operator. Let us also mention the work of Graham \cite{Graham:1992} who derives existence and uniqueness results for non-linear diffusions with Lipschitz coefficients and bounded jump rates.

The classical well-posedness theory of such dynamics is now well-understood in the standard Cauchy-Lipschitz framework on the space $\rr^d \times \mathcal{P}_p(\rr^d)$, $p\geq1$, $\mathcal{P}_p(\rr^d)$ being the space of probability measures with finite $p$ moments equipped with the Wasserstein distance of order $p$.

Trying to go beyond the previous setting by weakening the regularity assumptions on the coefficients with respect to both space and measure variables remains an interesting and challenging question that has attracted the attention of the research community. In the diffusive setting $\alpha=2$, let us mention \emph{e.g.} the work \cite{scheutzow_1987} by Scheutzow where a counterexample to uniqueness for the SDE \eqref{SDE:MCKEAN} is exhibited. Namely, when $\sigma \equiv 0$ and $b(t, x, m) = b(m) = \int_{\rr} \bar{b}(y) m(dy)$ for some bounded and locally Lipschitz function $\bar{b}$, the non-linear SDE \eqref{SDE:MCKEAN} with random initial condition has several solutions. However, when $\sigma \equiv 1$, the noise helps to restore existence and uniqueness. Indeed, the well-posedness of the corresponding martingale problem when $b(t, x, m)= b(x, m) = \int_{\rr^d} \bar{b}(x, y) \, m(dy)$ has been established by Shiga and Tanaka \cite{Shiga1985}. Such a result has been extended by Jourdain \cite{jourdain:1997} where uniqueness is  proved under the more general assumption that $b$ is bounded measurable and Lipschitz with respect to the total variation metric on the space of probability measures, the diffusion coefficient being Lipschitz in space and independent of the measure argument. Still in the non degenerate diffusive framework, some extensions have been recently obtained in \cite{mishura:veretenikov} who established some weak and strong well-posedness results for a diffusion coefficient depending only on time and space variables, in \cite{la:18} and also in \cite{2018arXiv180902216R} for dynamics with a singular convolution type interaction kernel in the drift coefficient. We finally mention the recent contribution \cite{chau:frik:18} where well-posedness results in the weak and strong sense are established through a fixed point approach on a suitable complete metric space of probability measures under mild H\"older type regularity assumptions on the coefficients with respect to both space and measure arguments. In particular, weak well-posedness is established when the diffusion coefficient is not Lipschitz with respect to the underlying Wasserstein distance. Some new propagation of chaos results for the approximation of the dynamics \eqref{SDE:MCKEAN} by a system of particles have recently been obtained by the same authors in their companion paper \cite{chaudru:frikha:18:b}. In \cite{jour:mele:woyc:08}, the authors extend the classical strong well-posedness results of \cite{Sznitman} for non-linear SDEs when both initial condition and L\'evy driving process are square integrable. When the L\'evy process is not square integrable, as it is the case in our current $\alpha$-stable framework, existence without uniqueness is proved when both coefficients are Lipschitz with respect to some bounded modification of the Wasserstein metric of order $1$.   

We here revisit the problem of the unique solvability of the SDE \eqref{SDE:MCKEAN} by tackling the corresponding formulation of the non-linear martingale problem under some rather mild assumptions on the coefficients in both the spatial and measure arguments. Namely, we will consider H\"older continuous in space coefficients with possibly unbounded drift term $b$ and a Lipschitz regularity condition w.r.t. a suitable H\"older type metric (which actually heavily depends on the spatial H\"older exponents of the coefficients) for the measure arguments. We will also assume some \textit{usual} non-degeneracy and boundedness conditions on the  L\'evy measure $\nu $ of the driving symmetric stable process $Z$  and the \textit{diffusion} coefficient $\sigma$ appearing in \eqref{SDE:MCKEAN}.

%Precisely, we are interested in finding under 

In the current setting, we thus aim at finding a unique probability measure $\P $ on ${\mathscr D}(\R^+ ,\R^d) $ (Skorokhod space of c\`adl\`ag functions) such that denoting by $(\P(t))_{t\geq0}$ its time marginals, $(y_t)_{t\ge 0} $ the associated canonical process and by $(\mathcal{F}_t)_{t\geq0}$ its filtration, $\P(0)=\mu$ and for any $\varphi \in C_0^{1,2}(\R^+ \times {\R}^d, \rr)$ the process
\begin{equation}
\label{eq:martingale:problem:formulation}
\varphi(t, y_t) - \varphi(0, y_0) - \int_0^{t}\big(\partial_r+ \mathscr{A}^{\P}_r\big) \varphi(r, y_r) dr
\end{equation}

\noindent is an $((\mathcal{F}_t)_{t\geq0}, \P)$-martingale starting from 0 at time $0$. In the above formulation, $ \mathscr{A}^{\P}_r$ is the following integro-differential operator
\begin{equation}\label{FIRST_ID_OP_LIN}
{\mathscr A}^{\P}_r\varphi(r, x)=\langle b(r, x, \P(r)),D_x\varphi(r, x)\rangle+\int_{\R^{d}\backslash\{ 0\}} \big(\varphi(r, x+\sigma(r, x, \P(r))z)-\varphi(r, x)- \langle \nabla \varphi (r, x), \sigma(r, x, \P(r))z\rangle \I_{|z|\le 1} \big)\nu(dz)
\end{equation}

\noindent acting on smooth test functions and in \eqref{FIRST_ID_OP_LIN} $\nu $ is a symmetric stable L\'evy measure on $\R^{d} $, i.e. for $\alpha\in (0,2) $, writing in polar coordinates  $\zeta=r \xi $, $(r,\xi) \in \R_+^*\times {\mathbb S}^{d-1}$:
\begin{equation}\label{stab_decomp}
\nu(d\zeta)=\frac{dr}{r^{1+\alpha}}\omega(d\xi),
\end{equation}
where $\omega $ is a symmetric measure on the sphere ${\mathbb S}^{d-1} $ satisfying some non-degeneracy assumptions.

Our main idea consists in applying a fixed point argument to a family of \textit{proxy linearized} martingale problems on a suitable complete metric space of probability measures. In the current strictly stable setting $\alpha \in (0,2)$, for the well-posedness in the weak and strong sense of the corresponding linear martingale problem, \emph{i.e.} when the dependence with respect to the measure argument is frozen, we can mention among others Tanaka \emph{et al.} \cite{tanaka1974}, Bass \cite{Bass2003}, Bass \emph{et al.} \cite{BASS20041}, Mikulevicius and Pragarauskas \cite{miku:prag:14}, Priola \cite{prio:18} or Zhang \textit{et al.} \cite{hao:wu:zhan:19} for the specific (possibly degenerate) stable case. In our previously mentionned H\"older setting, the main idea to complete our fixed point approach, consists in exploiting some parabolic Schauder estimates but under some \textit{rough} regularity conditions on the final condition, whose regularity will be somehow related to the spatial one of the coefficients. Such estimates are obtained through a forward parametrix perturbation argument in the same spirit as those established by Chaudru de Raynal \emph{et al} \cite{chau:hono:meno:18}, \cite{chau:meno:prio:19} in the Kolmogorov degenerate diffusive setting or in the $\alpha$-stable supercritical case. 

As a by-product of the usual strong uniqueness results on linear stable driven SDEs, the strong well-posedness of the corresponding non-linear SDE is also derived under suitable regularity assumptions on the coefficients and a non-degeneracy assumption on the L\'evy measure. Again, such results will be obtained for a bounded non-degenerate $\sigma$, with $b,\sigma $ being H\"older continuous in space and for potentially unbounded in space drift terms $b$.  Moreover, concerning the measure argument, the maps $\pp \ni m \mapsto b(t, x, m), \, \sigma(t, x, m)$ are also assumed to have a bounded linear functional (or flat) derivative denoted by $[\delta b/\delta m](t, x, m)(y)$ and $[\delta \sigma/\delta m](t, x, m)(y)$  (see Section \ref{sec:overview:assumption:main:result} for a precise definition) which is  H\"older-continuous w.r.t. both variables $x$ and $y$. 

 Compared to the aforementioned references in the strictly $\alpha$-stable context, our methodology allows to considerably weaken the regularity assumptions on the coefficients $b$ and $\sigma$ with respect to both space and measure variables. It also covers a large class of interaction in the measure argument with H\"older type regularity. In particular, no Lipschitz regularity with respect to Wasserstein distance is needed which, to the best of our knowledge, appears to be new. Let us indicate as well that, even in the stable Gaussian case,  i.e. $\alpha=2$ for which  $Z=B$ (standard Brownian motion), our approach applies and would provide well-posedness for non-linear SDEs with unbounded drift terms and full dependence on the measure in the diffusion coefficient (see Remark \ref{remarks:possible:improvements} below).

The article is organized as follows. The basic definitions together with the assumptions, the main results and the strategy of proof are described in Section \ref{sec:overview:assumption:main:result}. The sensitivity analysis w.r.t the measure argument of the semigroups generated by the linearized $\alpha$-stable driven SDE is carried out in Section \ref{PROOF_GROSSE_PROP}. The proof of some useful technical results are given in Appendix.

\section{Overview: definitions, assumptions and main results}\label{sec:overview:assumption:main:result}

%\begin{definition}{\textbf{Well-posedness for abstract non-linear martingale problems.}}\label{DEF_NL_MP}
%Let $\big(\mathscr{A}(t,\cdot ,  \rho; \partial_x)\big)_{t\ge 0} $ be a family of abstract integro-differential operators where $\rho\in {\mathcal P}(\R^d)$ is a measure argument. For a given initial distribution $\mu $, we say that the non-linear martingale problem associated with $\big(\mathscr{A}(t,\cdot ,  \rho; \partial_x)\big)_{t\ge 0} $ is well posed if there exists a unique probability measure $\P $ on 
%${\mathscr D}(\R^+ ,\R^d) $ s.t. denoting by $(y_t)_{t\ge 0} $ the associated canonical process, $[y_0]=\mu $ and for any $g\in C_0^{1,2}(\R^+,{\R}^d)$ the process:
%$$g(t,y_t)-g(0,y_0)-\int_0^{t}\big(\partial_r+ \mathscr{A}(r,y_r , [y_r]; \partial_x)\big) g(r,y_r) dr.$$
%is a $\P $-martingale starting at 0.%\footnote{\textcolor{red}{Attention ici, en fonction des exemples lorsque ${\mathcal D}\neq \R^d $: processus tues necessitent une localisation, processus reflechis seront a valeurs dans $\bar {\mathcal D} $. A discuter.}}
%\end{definition}

In order to tackle the non-linear martingale problem as introduced in \eqref{eq:martingale:problem:formulation}, we will proceed through a fixed point procedure on a suitable complete metric space of probability measure valued flows. Let us now describe this functional space.

\bigskip

\noindent  \textbf{The metric space of measure valued flows.} We first endow ${\mathcal P}(\R^d) $ with the following metric. Fix $\beta \in (0,1] $ and define for ${\mu, \nu}\in {\mathcal P}(\R^d) $,
\begin{equation}
\label{DIST_CONT_MEAS}
\db_\beta(\mu,\nu) := \sup_{f: %|f|_\infty + |f|_{\eta}
\|f\|_{{\mathcal C}^\beta}\leq 1}\int f(z) (\mu-\nu)(dz)  ,
\end{equation}
\noindent  where $\|f\|_{{\mathcal C}^\beta}:=|f|_\infty +\sup_{(x,y)\in {(\R^d)}^2,\ x\neq y}\frac{|f(x)-f(y)|}{|x-y|^\beta} $ denotes the usual H\"older norm on ${\mathcal C}^\beta(\R^d,\R) $, where $|\cdot| $ stands for the usual Euclidean norm of $\R^d $ (see e.g. Krylov \cite{kryl:96}). For $\beta= 1$, $\db_\beta$ is also known as the Fortet-Mourier distance. Let us introduce now the cost function $c_\beta(x, y) = |x-y|^\beta \wedge 1$, $(x, y)\in \R^d$. Denote by $\Pi(\mu, \nu)$ the set of all transport plans from $\mu$ to $\nu$ w.r.t $c_\beta$. Then, for any $\pi \in \Pi(\mu, \nu)$ and any $f$ satisfying $\|f\|_{{\mathcal C}^\beta}\leq 1$, it is readily seen
$$
|\int f(z) (\mu-\nu)(dz)| = |\int (f(x)-f(y)) \pi(dx, dy)| \leq 2\int c_\beta(x, y) \pi(dx, dy)
$$

\noindent so that optimizing w.r.t $\pi$ and $f$
$$
\db_\beta(\mu, \nu) \leq 2 \widetilde{W}_\beta(\mu, \nu):= 2 \inf_{\Pi(\mu, \nu)}\int c_\beta(x, y) \pi(dx, dy).
$$

Let us note that $({\mathcal P}(\R^d),\db_\beta)$ is a complete metric space\footnote{If $(\mu_k)_{k\in \N}$ is a Cauchy sequence in $\mathcal{P}(\R^d)$ equipped with $\db_\beta$, then a slight modification of the proof of Lemma 6.12 in \cite{Villani:old:and:new} allows to conclude that it is tight so that it admits a subsequence (still denoted by $(\mu_k)_{k\in \N}$) which weakly converges to some measure $\mu$. For each $k$, introduce an optimal transport plan $\pi_k$ between $\mu_k$ and $\mu$ (see e.g. Theorem 4.1 in \cite{Villani:old:and:new} for the existence of such optimal coupling) for the induced related cost function $c_\beta$. By Lemma 4.4 in  \cite{Villani:old:and:new}, $(\pi_k)_{k \in \N}$ is tight in $\mathcal{P}(\R^d \times \R^d)$. Hence, up to an extraction, we may assume that $\pi_k \rightarrow \pi$ weakly as $k\rightarrow \infty$. Since each $\pi_k$ is optimal, Theorem 5.19 of \cite{Villani:old:and:new} guarantees that $\pi$ is an optimal (trivial) coupling between $\mu$ and $\mu$ so that $\lim\sup_{k \rightarrow \infty}\db_\beta(\mu_k, \mu) \leq 2 \lim\sup_{k\rightarrow \infty} \int_{(\rr^d)^2} c_\beta(x,y) \pi_k(dx, dy) = 0$. Since $(\mu_k)_{k\in \N}$ is a Cauchy sequence with a converging subsequence (w.r.t the metric $\db_\beta$), it eventually converges to $\mu$.} and that $\db_\beta$ metrizes the weak convergence of probability measures. %modulus of $f$ w.r.t the distance $d$, i.e. $|f|_\eta:=\sup_{(x,y)\in {\mathcal D}^2,\ x\neq y}\frac{|f(x)-f(y)|}{d^\eta(x,y)} $.

For fixed $0 \leq s <t < \infty$, we then introduce the set of continuous probability measure valued flows $\mathcal{C}([s,t],\mathcal{P} (\R^d))$. We will from now on denote the elements of $\mathcal{C}([s,t],\mathcal{P} (\R^d))$ with a bold capital letter and  equip it with the following uniform metric. For all $\bP,\bP'\in \mathcal{C}([s,t],\mathcal{P} (\R^d)) $,
\begin{equation}
\label{DIST_CONT_MEAS_FLOWS}
\db_{\beta, s,t}({\mathbf P},{\mathbf P}') := \sup_{r\in [s,t]} \left\{ \db_\beta({\mathbf P}(r),{\mathbf P}'(r)) \right\}.
\end{equation}
Since $(\mathcal{P}(\rr^d), \db_\beta)$ is a complete metric space, it follows that $\mathcal{C}([s, t],\mathcal{P} (\R^d)) $ equipped with $\db_{\beta, s,t}$ is also complete.  
We refer to Villani \cite{Villani:old:and:new} for related issues concerning the aforementioned points. %\footnote{\textcolor{red}{One needs to verify the aforementioned duality principle}}. 
There are many topologies with which we can equip the space $\mathcal{C}([s,t],\mathcal{P} (\R^d))$. The above choice of metric strongly derives from the fact that a H\"older-type regularity assumption on both coefficients $b$ and $\sigma$ appears to be sufficient to establish the well-posedness in the weak sense of standard linear stable driven SDEs (see e.g. \cite{miku:prag:14}). Let us also mention that, beyond our H\"older setting, the choice of the metric is very much related to the spatial smoothness of the coefficients. In particular, when little or no regularity is available in space for the coefficients (bounded or $L^q-L^p $, see e.g. \cite{mishura:veretenikov} or \cite{2018arXiv180902216R}) the \textit{natural} metric for the measure appears to be the total variation one.

\smallskip

%Let $\nu$ be a $\sigma$-finite measure on $\mathcal{D}$ and $d$ a metric on $\mathcal{D}$. 
Now, for a given $\mu \in \mathcal{P}(\R^d)$, we introduce the subset $\mathscr{\bA}_{s, t, \mu}$ of $\mathcal{C}([s,t],\mathcal{P} (\R^d))$ defined by
\begin{align}
\mathscr{\bA}_{s,t, \mu} & := \left\{ {\mathbf P} \in \mathcal{C}([s,t], \mathcal{P}(\R^d)): {\mathbf P}(s)= \mu \right\} \label{space:prob:mes}.
\end{align}
 The subspace $\mathscr{\bA}_{s,t, \mu} $ equipped with the metric defined in \eqref{DIST_CONT_MEAS_FLOWS} can also be viewed as a complete metric space.
 \bigskip
 
 In order to rigorously formulate our regularity assumption on the coefficients w.r.t the measure argument, we introduce the notion of flat derivative of a continuous map $U: \mathcal{P}(\R^d) \rightarrow \R$, $\mathcal{P}(\R^d)$ being equipped with the distance $\db_\beta$ for some $\beta \in (0,1]$. This notion will play a key role in our analysis. We refer to Chapter 5 of the monograph by Carmona and Delarue \cite{Carmona:Delarue:book:1} for a more detailed discussion on the notion of differentiability of functions of probability measures.

\begin{definition}[Flat derivative of $U$] \label{def:flat:derivative}The continuous map $U: \pp \rightarrow \rr$ is said to have a continuous linear functional derivative if there exists a continuous function $\delta U/\delta m: \pp \times \rr^d \rightarrow \rr$ such that $y\mapsto [\delta U/\delta m](m)(y)$ is bounded, uniformly in $m$ for $m\in \mathcal{K}$, $\mathcal{K}$ being any compact subset of $\pp$ and such that for any $m, m' \in \pp$,
$$
\lim_{\varepsilon \downarrow 0} \frac{U((1-\varepsilon) m + \varepsilon m') - U(m)}{\varepsilon} = \int_{\rr^d} \frac{\delta U}{\delta m}(m)(y) \, (m'-m)(dy).
$$ 

The map $y\mapsto [\delta U/\delta m](m)(y)$ being defined up to an additive constant, we will follow the usual normalization convention $\int_{\rr^d} [\delta U/\delta m](m)(y) \, m(dy) = 0$.
\end{definition}

Let us illustrate the structure of this notion of differentiation with a couple of commonly encountered examples.
\begin{example} 
\begin{enumerate}

\item If the function $U$ is of scalar form, namely, $U(m) = \int_{\rr^d} h(y) m(dy)$ for some real-valued bounded and continuous function $h$ defined on $\R^d$ then it is readily seen that 
$$
\frac{\delta U}{\delta m}(y) = h(y).
$$ 

\smallskip

\item If the function $U$ is of quadratic type w.r.t the measure, namely, denoting by $\star$ the usual convolution operator, $U(m) = \int_{\R^d} [h\star m](x) m(dx) = \int_{(\rr^d)^2}h(x-y) \, m(dx)m(dy)$ for some bounded and continuous function $h$ defined on $\R^d$ then one has
$$
\frac{\delta U}{\delta m}(y) = \int_{\rr^d} h(x-y) \, m(dx) + \int_{\rr^d} h(y-x) \, m(dx).
$$
\end{enumerate}

\end{example}
 \bigskip
 Remark that if $U$ has a linear functional derivative then for any $m, m' \in \pp$, it holds
 \begin{equation}\label{taylor:exp:mes}
 U(m) - U(m') = \int_0^1 \int_{\rr^d} \frac{\delta U}{\delta m}(\lambda m + (1-\lambda) m')(y) \, (m-m')(dy)\, d\lambda.
 \end{equation}

 \noindent
 \textbf{The family of inhomogeneous integro-differential operators.} In order to solve locally in time the non-linear martingale problem \eqref{eq:martingale:problem:formulation}, we will perform a fixed point argument in the metric space $\mathscr{\bA}_{s,t, \mu} $. We thoroughly need to consider some inhomogeneous \textit{linearized} integro-differential operators associated with a given measure flow ${\boldsymbol \mu} \in  \mathcal{C}([s,t], \mathcal{P}(\R^d))$. Namely, for 
 $(r, x) \in [s, t] \times \R^d$ and all $\varphi\in C_0^{1, 2}([s, t] \times \R^d,\R) $, we abuse the definition introduced in \eqref{FIRST_ID_OP_LIN}\footnote{Observe indeed that $\P $ in \eqref{FIRST_ID_OP_LIN} is a probability measure on $\mathscr D (\R^+,\R)$ whereas $\bmu\in {\mathcal C}([s,t],\mathcal P(\R^d))$ is a flow of $\R^d$-valued probability measures.} and write
\begin{equation}\label{THE_GEN_STABLE}
\mathscr{A}_r^{\bmu}\varphi(r, x)\equiv \Big(\mathscr{A}(r, \cdot, \bmu; D_x)\Big)\varphi(r, x) =L_r^{\bmu,\alpha} \varphi(r, x)  +  \langle b(r, x, \bmu(r)), D_x\varphi(r, x) \rangle,
\end{equation}
where %for all $\varphi \in C_0^{2}(\R^d,\R) $,  $x\in \R^d $:
\begin{equation}\label{OP_STABLE_MCKEAN}
L_r^{\bmu,\alpha}\varphi(r, x)={\rm {p.v.}}\int_{\R^d} \big(\varphi(x+\sigma(r,x,\bmu(r)) \zeta)-\varphi(x) \big) \nu(d\zeta),
\end{equation}
and $\nu$ is the L\'evy measure given by \eqref{stab_decomp}.
%of a non-degenerate $\R^d $-valued symmetric $\alpha $-stable process as in \eqref{stab_decomp} with $\alpha\in (0,2) $, i.e. in polar coordinates writing $\zeta=r \xi $, $(r,\xi) \in \R_+^*\times {\mathbb S}^{d-1}$:
%\begin{equation*}%\label{stab_decomp}
%\nu(d\zeta)=\frac{dr}{r^{1+\alpha}}\omega(d\xi).
%\end{equation*}
%Namely, in \eqref{stab_decomp}, the spherical measure $\omega $ on ${\mathbb S}^{d-1} $ (which also corresponds up to a multiplicative constant to the spectral measure of the underlying process) is assumed to satisfy the following non-degeneracy assumption.

We now introduce some non-degeneracy and regularity assumptions on the L\'evy measure and coefficients.
\subsubsection*{Non-degeneracy assumptions}
\begin{trivlist}
\item[\A{ND}] There exists  $ \kappa \ge 1$ s.t. for all $z \in \R^d$:
\begin{equation}
\label{NON_DEG_STABLE}
\kappa^{-1}|z|^\alpha\le \int_{{\mathbb S}^{d-1}}|\langle z ,\xi  \rangle |^\alpha
\omega(d\xi)\le \kappa|z|^\alpha.
\end{equation}
\end{trivlist}
Note that a wide family of spectral measures satisfy condition \eqref{NON_DEG_STABLE}, from absolutely continuous ones to (very) singular ones. For instance, taking $\omega=\Lambda_{{\mathbb S}^{d-1}} $ (Lebesgue measure on the sphere) readily yields the previous control. On the other hand, any measure $\omega=\sum_{i=1}^d c_i (\delta_{e_i}+\delta_{-e_i}) $ with positive coefficients $(c_i)_{i\in \leftB 1,d \rightB}$, where the $(e_i)_{i\in \leftB 1,d\rightB} $ stand for the canonical vectors, also satisfies \eqref{NON_DEG_STABLE}.  

In order to solve the non-linear martingale problem, we will also assume that:
\begin{trivlist}
\item[\A{AC}] The jump measure writes: $\nu(dy)=f(y)dy $ for a continuously differentiable function  $f:\R^d\rightarrow \R$ with bounded and Lipschitz first order derivative. Since $\nu $ is a symmetric stable L\'evy measure (see \eqref{stab_decomp}), we also have that:
\begin{eqnarray}
\label{EXPR_FOR_F}
f(y)=\frac{g(\frac{y}{|y|})}{|y|^{d+\alpha}},
\end{eqnarray}
where $g$ is an even and continuously differentiable spherical function with a bounded and Lipschitz first order derivative.%\footnote{\textcolor{red}{Je n'arrive pas sans cela a recuperer exactement la distance lorsque l'on compare deux operateurs integro-differentiels avec differents sigma. J'obtiens des puissances de la distance cela ne suffit pas.}}. 
\end{trivlist}

%Some results about the freezing process still hold without assuming \A{AC} which anyhow turn out to be quite crucial to establish some suitable sensitivity results w.r.t. the measure argument in the density of the proxy  process.

We also assume that the \textit{diffusion} coefficient $\sigma$ in \eqref{OP_STABLE_MCKEAN} satisfies the following uniform ellipticity condition. 

\begin{trivlist}
\item[\A{UE}]There exists $\Lambda \ge 1 $ s.t. for all $(t,z,\mu)\in  \R_+ \times \R^d \times{\mathcal P}(\R^d)$ and all $\xi \in \R^d $:
\begin{equation}
\label{COND_UE_S}
\Lambda^{-1} |\xi|^2 \le \langle (\sigma^*\sigma)(t,z,\mu)\xi, \xi \rangle \le \Lambda |\xi|^2.
\end{equation}
\end{trivlist}

\subsubsection*{Regularity assumptions on the coefficients}

We will also need some regularity assumptions on the coefficients $\sigma$ and $b$ w.r.t both the spatial and measure arguments. 

We assume that the diffusion coefficient $\sigma$ satisfies the following conditions:
\begin{trivlist}
\item{\A{D${}_H $}}
For any $(t, \mu) \in \rr_+\times \pp$, the mapping $x\mapsto \sigma(t, x, \mu) $ is bounded and $2\eta $-H\"older continuous, for some $\eta \in (0,1/2] $, uniformly in $(t,\mu) \in \rr_+\times \pp $. Namely, there exists $C\ge 1$ s.t., for all $(t, \mu)  \in \rr_+\times \pp$:
\begin{equation}\label{DIFF_HOLDER_STABLE}
\|\sigma(t ,\cdot,\mu)\|_{{\mathcal C}^{2\eta}}\le C,
\end{equation}

\noindent where ${\mathcal C}^{2\eta}(\R^d,\R^m):=\{\varphi\in {\mathcal C}(\R^d,\R^m): \|\varphi\|_\infty+\sup_{x\neq y}\frac{|\varphi(x)-\varphi(y)|}{|x-y|^{2\eta}}<+\infty \} $ with $m\in \{1,d\times d\} $ denotes the usual H\"older space. When there is no ambiguity, as in \eqref{DIFF_HOLDER_STABLE}, we simply denote ${\mathcal C}^{2\eta}:= {\mathcal C}^{2\eta}(\R^d,\R^m)$ to ease the reading.
 
For any $(t, x) \in \rr_+\times \rr^d$, the continuous map $\pp \ni m \mapsto \sigma(t, x, m) $ admits a bounded and continuous flat derivative denoted by the map $[\delta \sigma/\delta m](t, x, m)(.)$ such that $(x, y) \mapsto [\delta \sigma/\delta m](t, x, m)(y)$ is $2\eta$-H\"older continuous, uniformly w.r.t the variables $t$ and $m$.

We suppose that the drift coefficient $b$ fulfills the following conditions:
\begin{trivlist}
\item[\A{B${}_H $}] For any $(t,\mu)\in \R^+ \times \pp $, the mapping $x\mapsto b( t ,x,\mu) $ belongs to the homogeneous H\"older space $$\dot {\mathcal C}^{2\eta}(\R^d,\R^d):=\{\varphi \in  {\mathcal C}(\R^d, \R^d): \ \sup_{x\neq y}\frac{|\varphi(x)-\varphi(y)|}{|x-y|^{2\eta}} <+\infty \} ,\ \eta \in (0,\frac12]$$ uniformly w.r.t. the variables $(t, \mu)$. In particular, the drift $b$ may be unbounded in space.

For any $(t, x) \in \rr_+\times \rr^d$, the continuous map $\pp \ni m \mapsto b(t, x, m) $ admits a bounded and continuous flat derivative denoted by the map $[\delta b/\delta m](t, x, m)(.)$ such that $(x, y) \mapsto [\delta b/\delta m](t, x, m)(y)$ is $2\eta$-H\"older continuous, uniformly w.r.t the variables $t$ and $m$.
 \end{trivlist}

We eventually assume that the following condition, linking the spatial H\"older exponent $2\eta $ in assumptions \A{D${}_H$} and  \A{B${}_H$} and the stability index $\alpha $ is fulfilled:
\begin{trivlist}
\item[\A{L}] It holds that:
$$2\eta+\alpha>1. $$
\end{trivlist}

We importantly point out that the former condition appearing in \A{B${}_H$} is the natural constraint arising in \cite{chau:meno:prio:19} to derive Schauder estimates for a general, and potentially singular, non-degenerate spherical measure $\omega $ in the sense of \eqref{NON_DEG_STABLE} in the super-critical case.

% We also suppose:% similarly to the diffusive case that\footnote{On pourra grouper des hypotheses afin de gagner un peu de place.}:
%\begin{equation}
%\label{hyp:nonlinearity:drift:stable}
%\big|b\big(v, x, \bP_1(v)\big) - b\big(v, x, \bP_2(v)\big)\big| \leq C \bd_{\eta,s,t}(\bP_1, \bP_2).
%\end{equation}
%For $\alpha\in (\textcolor{red}{4/5},1] $, introducing the map $x \mapsto \delta b\big(v, x, \bP_1(v),\bP_2(v)\big) := b\big(v, x, \bP_{1}(v)\big) - b\big(v, x, \bP_{2}(v)\big)$, we require some additional stability properties assuming that:
%\begin{equation}
%\|\delta b\big(v, ., \bP_1(v),\bP_2(v) \big)\|_{{\mathcal C}^{\eta}} \leq C \bd_{ \eta, s, t}(\bP_1, \bP_2) \label{holder:reg:diff:mes:drift_stable}. 
%\end{equation}

We will say that assumption \A{A${}_{S}$} is satisfied if assumptions \A{ND}, \A{AC}, \A{UE}, \A{D${}_H $}, \A{B${}_H $} and \A{L} hold.

\bigskip

Before going further, we first derive an important control that will play a key role in our analysis of the well-posedness of the non-linear martingale problem related to the SDE \eqref{SDE:MCKEAN}. Under the regularity assumptions \A{D${}_H $} and \A{B${}_H $}, denote by $U$ one of the coefficients $\sigma$ or $b$ and  introduce the map $x \mapsto \delta U\big(t, x, \mu,  \nu \big) := U\big(t, x, \mu \big) - U\big(t, x, \nu \big)$, for any fixed $\mu,\, \nu \in \mathcal{P}(\R^d)$. From \eqref{taylor:exp:mes} and the very definition \eqref{DIST_CONT_MEAS} of $\db_{2\eta}$, one readily gets 
\begin{equation}
\forall \beta \in [0,1], \quad |\delta U\big(t, ., \mu, \nu \big)|_\infty \leq \sup_{t, m}\|\frac{\delta }{\delta m } U(t, ., m)\|_{\mathcal{C}^{2\eta}} \db_{2\eta}(\mu, \nu) \leq \sup_{t, m}\|\frac{\delta }{\delta m } U(t, ., m)\|_{\mathcal{C}^{2\eta}} \db_{(1-\beta) 2\eta}(\mu, \nu). \label{uniform:bound:diff:mes}
\end{equation}

Similarly,   
\begin{align*}
 \delta U\big(t, x, \mu, \nu\big) -   \delta U\big(t, x', \mu, \nu \big) & = \int_0^1\int_{\rr^d} \Big[\frac{\delta}{\delta m}U(t, x, \mu_\lambda) - \frac{\delta}{\delta m}U(t, x', \mu_\lambda)   \Big](y) (\mu-\nu)(dy) \, d\lambda
\end{align*}

\noindent where we introduced the notation $\mu_\lambda:= \lambda \mu + (1-\lambda) \nu$ for $\lambda \in [0,1]$. Using the fact that $(x, y) \mapsto \frac{\delta}{\delta m}U(t, x, m)(y)$ is $2\eta$-H\"older uniformly w.r.t. the variables $t$ and $m$, for any $\beta \in [0,1]$, we obtain
\begin{align*}
\Big| \Big[\frac{\delta}{\delta m}U(t, x, \mu_\lambda) -&  \frac{\delta}{\delta m}U(t, x', \mu_\lambda)   \Big](y) - \Big[\frac{\delta}{\delta m}U(t, x, \mu_\lambda) - \frac{\delta}{\delta m}U(t, x', \mu_\lambda)   \Big](y') \Big| \\
& \leq 2 \sup_{t, m}\|\frac{\delta}{\delta m}U(t, ., m)(.)\|_{\mathcal{C}^{2\eta}} |x-x'|^{2\eta} \wedge |y-y'|^{2\eta}\\
& \leq 2 \sup_{t, m}\|\frac{\delta}{\delta m}U(t, ., m)(.)\|_{\mathcal{C}^{2\eta}} |x-x'|^{\beta 2 \eta} |y-y'|^{(1-\beta)2\eta}.
\end{align*}

\noindent Now, from the boundedness of $ \frac{\delta}{\delta m}U$ and the uniform $2\eta$-H\"older regularity of $x\mapsto \frac{\delta}{\delta m}U(t, x, \mu)$, we similarly get
$$
|\big[\frac{\delta}{\delta m}U(t, x, \mu_\lambda) -  \frac{\delta}{\delta m}U(t, x', \mu_\lambda)\big](.)|_\infty \leq2 \sup_{t, m}\|\frac{\delta}{\delta m}U(t, ., m)(.)\|_{\mathcal{C}^{2\eta}}  |x-x'|^{\beta 2\eta} 
$$

\noindent for any $\beta \in [0,1]$. From the preceding estimates, one thus deduces
\begin{equation}
\forall \beta \in [0,1], \, \|\delta U\big(t, ., \mu, \nu \big)\|_{{\mathcal C}^{\beta 2\eta}} \leq 2 \sup_{t, m}\|\frac{\delta}{\delta m}U(t, ., m)(.)\|_{\mathcal{C}^{2\eta}} \bd_{ (1-\beta) 2\eta}(\mu, \nu) \label{holder:reg:diff:mes:b:or:sigma_stable}. 
\end{equation}
\end{trivlist}

Under \A{A${}_S $},  for any fixed  measure flow  ${\mathbf Q}$ in $\mathcal{C}([0,\infty), \mathcal{P}(\rr^d))$, with time marginal ${\mathbf Q}(t)$, and any initial conditions $(s, \mu) \in [0, \infty) \times \mathcal{P}(\rr^d)$, there exists a unique solution to the linearized martingale problem, that is, the martingale problem associated with the operator $\big(\mathscr{A}^{\mathbf{Q}}_t\equiv \mathscr{A}(t, \cdot, {\mathbf Q}; D_x)\big)_{t\ge s}$ introduced in \eqref{THE_GEN_STABLE} is well-posed. We can refer to \cite{chau:meno:prio:19} which gives the corresponding Schauder estimates which in turn allows to derive the aforementioned well-posedness following the procedure described in e.g. Mikulevicus and Pragarauskas \cite{miku:prag:14} or Priola \cite{prio:12}. % For the diffusive case, under \A{A${}_D $}, the well posedness of the martingale problem follows for instance from \cite{chau:hono:meno:18}

Moreover,  denoting by $\P^{\bf{Q}}_{s,x}$  the unique solution to the martingale problem starting from $\delta_x$ at time $s$, we have that the associated canonical process is strong Markov.  The corresponding time-inhomogeneous Markov semigroup will be denoted by $(P^{\bQ}_{s,t})_{t\ge s} $.\\

\begin{theorem}[\textbf{Well-posedness of the non-linear martingale problem}] \label{THM_WP_NL}
Under assumption {\bf(A${}_S $)}, for any $s\geq0$ and any initial distribution $\mu\in {\mathcal P}({\R^d})$, the non-linear martingale problem starting from $\mu$ at time $s$ is well-posed for any $\eta \in (0,1/2] $ if $\alpha\ge 1 $ and under the condition $\alpha>2\eta \vee (1-\eta )$ if $\alpha<1 $. 
%. In other words, there exists a unique solution $\P_s$ to the non-linear martingale problem for the operator $({\mathscr L}_t)_{t\ge s}=\big(\mathscr{A}(t,  \cdot, [y_t]; \partial_x) \big)_{t\ge s}$ introduced in \eqref{NL_GEN} recalling that $(y_t)_{t\ge s} $ denotes the canonical process and $[y_t] $ its law at time $t$.  
\end{theorem}

Observe that, in the super-critical case $\alpha<1 $ the lowest possible stable index $\alpha $ we can attain is strictly greater than $2/3$.

\begin{corollary}[Strong uniqueness]\label{COROL_STRONG}
 Under \A{A${}_S $}, if $2\eta\ge 1-\frac \alpha 2 $, then equation \eqref{SDE:MCKEAN} admits a unique strong solution for any  $\alpha\ge 1 $ and   when $\alpha<1 $ under the additional condition $\alpha> 2\eta \vee (1- \eta )$.
\end{corollary}

\begin{proof}[Proof of Theorem \ref{THM_WP_NL}] Let $T>s$ be fixed and such that $T-s $ is \textit{small}. As already mentioned, the central idea is to use the Banach fixed point theorem on the complete metric space $\mathcal{\boldsymbol A}_{s, T, \mu}$ equipped with  $\bd$.

We consider the map $\mathscr{\boldsymbol T}: \mathscr{\boldsymbol A}_{s,T, \mu} \rightarrow \mathscr{\boldsymbol A}_{s, T, \mu}$ which to the measure flow ${\mathbf Q} \in \mathscr{\boldsymbol A}_{s,T, \mu}$ associates the measure flow $\mathscr{\boldsymbol T}({\mathbf Q}) \in \mathscr{\boldsymbol A}_{s,T, \mu}$ induced by the probability measure $\P_s^{{\mathbf Q}}$ given by the unique solution to the linear martingale problem associated with the operator $\big(\mathscr{A}(t,\cdot ,{\mathbf Q};  D_x\big)$ given by \eqref{THE_GEN_STABLE} starting from the initial distribution $\mu$ at time $s$. Namely, one has $\big(\mathscr{\boldsymbol T}({\mathbf Q})\big)(t) = [X^{\mathbf{Q}}_t]$, for any $t\in [s,T]$, where $(X^{\mathbf{Q}}_t)_{t\in [s,T]}$ is the unique weak solution to the linear SDE with dynamics
$$
X^{\mathbf{Q}}_t = \xi + \int_s^t b(r,X^{\mathbf{Q}}_r, \mathbf{Q}(r)) dr +\int_s^t  \sigma(r,X^{\mathbf{Q}}_{r^-}, \mathbf{Q}(r)) dZ_s,\  [\xi]=\mu.
$$
In particular, one has ${\boldsymbol T}(\bQ)(s)=\mu $
%and for all $(v,v')\in [s,t]^2 $, $\bd_\eta({\boldsymbol T}(\bQ)(v'),{\boldsymbol T}(\bQ)(v)) =\sup_{\|h\|_{{\mathcal C}^\eta}\le 1 }\int_{{\rr^d}} \big(P_{s,v}^{\bQ}h(z)-P_{s,v'}^{\bQ}h(z)\big)\underset{v\rightarrow v'}{\rightarrow}0$ because  $\P^\bQ $ is a Feller semigroup.
and
\begin{equation}
 \int h(y) \mathscr{\boldsymbol T}({\mathbf Q})(t)(dy) = \int P_{s, t}^{\bQ}h( y) \mu(dy ) ,  \label{THE_CORRESP_ITERATION}
\end{equation}

\noindent recalling that $(P^{\bQ}_{s, t})_{t\geq s}$ denotes the strong Markov semigroup associated with the operator $\big(\mathscr{A}(t, \cdot,{\mathbf Q}; D_x)\big)_{t\ge s}$.
% We also point out that in the second equality of \eqref{THE_CORRESP_ITERATION}, we have written with a slight abuse of notation $(y_t)_{t\ge s} $ for the canonical process under $\P_{s,\tilde y}^{\bQ}$.

%Let ${\mathbf P}_i, i=1,2 $ be two measure flows on $\mathscr{A}_{s, T, \mu} $. Write
%\begin{eqnarray*}
% \int h(y) \mathscr{\boldsymbol T}({\mathbf P}_i)(t)(dy)&=&\int h(y)[y_t^{\bP_i}](dy)
% =\int  \E^{\P_{s,\tilde y}^{\bP_i}}[h(y_t)]\mu(d\tilde y)\\
%&=& \int \mu(d\tilde y ) P_{s,t}^{\bP_i}h(\tilde y),  
%\end{eqnarray*}
%where we denote by $(P^{i}_{s, t})_{t\geq s}$ the strong Markov semigroup associated with the operator $\big(\mathscr{A}(t, \cdot,{\mathbf P}_i(t); \partial_x)\big)_{t\ge s}$.

Let now ${\mathbf P}_i, i=1,2 $, be two measure flows on $\mathscr{\boldsymbol A}_{s, T, \mu} $. Write from \eqref{THE_CORRESP_ITERATION} and \eqref{DIST_CONT_MEAS_FLOWS}:
\begin{eqnarray}
\db_{\eta, s,T}\big({\mathbf T}({\mathbf P}_1),{\mathbf T}({\mathbf P}_2)\big)&=&\sup_{t\in [s,T]} \sup_{h: %|f|_\infty + |f|_{\eta}
\|h\|_{{\mathcal C}^\eta}\leq 1}\int h(y) \big({\mathbf T}({\mathbf P}_1)(t)-{\mathbf T}({\mathbf P}_2)(t)\big)(dy)\nonumber\\
&=&\sup_{t\in [s,T]} \sup_{h: %|f|_\infty + |f|_{\eta}
\|h\|_{{\mathcal C}^\eta}\leq 1}\int  \Big( P_{s,t}^{\bP_1}-P_{s,t}^{\bP_2} \Big)h( y) \mu(dy ).\label{TO_CONCLUDE_FIXED_POINT}
\end{eqnarray}

 Let us now write for notational convenience $P_{s,t}^{\bP_i}\equiv P_{s,t}^{i}, \ t\in [s,T] $
 for $i=1, 2$. 
The following Proposition, whose proof is postponed to Section \ref{PROOF_GROSSE_PROP} is the key
to prove that $ {\mathbf T}$ generates a contraction on $\mathscr{\boldsymbol A}_{s, T, \mu} $ in small time.

\begin{prop}[Sensitivity Analysis of the semi-groups w.r.t. the measure argument]\label{PROP_STAB}
Under \A{A${}_S $} and the additional assumption $\alpha \geq 2\eta \vee (1- \frac{\eta}{2})$ if $\alpha< 1 $, for any fixed $T>0$, there exists a constant $C:=C(T, \A{A{}_S}, \alpha, \eta)\ge 1$ ($T\mapsto C(T, \A{A{}_S}, \alpha, \eta)$ being non-decreasing) and an exponent $\tilde \zeta>0 $ s.t. for all $h\in {\mathcal C}^{\eta}(\R^d,\R),\ \|h\|_{{\mathcal C}^\eta}\le 1 $ and for any $(t,x)\in [s,T]\times \R^d $:
\begin{equation}\label{THE_EQ_SENSI_SG_MES}
|(P^{1}_{s, t} - P^{2}_{s, t}) h(x)|\le C(t-s)^{\tilde \zeta}\bd_{\eta,s,t}(\bP_1,\bP_2).
\end{equation}
\end{prop}

\noindent From \eqref{THE_EQ_SENSI_SG_MES} and \eqref{TO_CONCLUDE_FIXED_POINT}, recalling as well from \eqref{DIST_CONT_MEAS_FLOWS} that $[s, T] \in t\mapsto \bd_{\eta,s,t}(\bP_1,\bP_2) $ is a non-decreasing function of $t$, we eventually derive:
$$
\bd_{\eta, s, T}(\mathscr{\mathbf T}(\bP_1),\mathscr{\mathbf T}(\bP_2)) \leq C(T-s)^{\tilde \zeta}  \bd_{\eta, s, T}(\bP_1,\bP_2).
$$

Hence, we see that for $T$ sufficiently close to $s$, so that  $C(T-s)^{\tilde \zeta} < 1$, the map $\mathscr{\mathbf T}$ is a contraction on the complete metric space $\mathcal{\boldsymbol A}_{s, T, \mu}$. It thus admits a unique fixed point that we denote by $\bP^{0}$. For $n\geq1$, if $\bP^{n-1}$ is constructed, denote by $\bP^{n}$ the unique fixed-point of $\mathscr{\mathbf T}$ on $\mathcal{\boldsymbol A}_{nT, (n+1)T, \bP^{n-1}(nT)}$. Now set $\bP(t) := \bP^{n}(t)$ if $t\in [nT,(n+1)T]$ and denote by ${\mathbb P}_s^{\bP}$ the unique solution to the linear martingale problem starting from the distribution $\mu$ at time $s$ in which the measure dependence is given by $\bP$. By the characterization of the martingale problem, it is readily seen that for any integer $n$ and any $t\in [nT, (n+1)T]$, one has $({\mathbb P_s^{\bP}})^{-1}\circ y_{ t}^{\bP} = \bP^{n}(t) = \bP(t)$. By the uniqueness of the fixed point, it turns out that $\P_s:={\mathbb P}_s^{\bP}$ is also the unique solution to the non-linear martingale problem. This completes the proof.

\end{proof}

At this step, we importantly emphasize that Proposition \ref{PROP_STAB} allows to conclude that Theorem \ref{THM_WP_NL} is valid under the additional stronger condition $\alpha > 2\eta \vee (1-\frac{\eta}{2})$ which appears to be quite stringent in the supercritical case $\alpha <1$. Indeed, in this context, the minimal attainable stability index $\alpha$ is strictly greater than $4/5$. In order to establish the well-posedness of the non-linear martingale problem under the sole condition $\alpha > 2\eta \vee (1-\eta)$, one has to suitably modify the above argument by performing the fixed point argument using the map $\mathbf{T}^2 = \mathbf{T} \circ \mathbf{T}$ still on the complete metric space $\mathcal{\boldsymbol A}_{s, T, \mu}$ equipped with the distance $\db_{\beta, s, T}$ for some well-chosen $\beta \in (0,2\eta)$. For the sake of clarity, the main modifications of the proof of Proposition \ref{PROP_STAB} needed to establish that the iterated procedure provides a contraction are presented in Appendix \ref{sec:modification:proof}.

\begin{proof}[Proof of Corollary \ref{COROL_STRONG}]
Theorem \ref{THM_WP_NL} provides weak uniqueness for the McKean-Vlasov SDE \eqref{SDE:MCKEAN}. The dependence in the law can then classically be viewed as an inhomogeneous time dependence (see e.g.  \cite{chau:frik:18}). Namely,  \eqref{SDE:MCKEAN} can be rewritten as:
$$X_t=\xi+\int_0^t \widetilde b(s,X_s) ds+\int_0^t \widetilde \sigma (s,X_{s^-}) dZ_s, $$
where $\widetilde b(s,X_s)=b(s,X_s, [X_s]),\ \widetilde \sigma(s,X_s)=\sigma(s,X_s, [X_s]) $. To derive strong well-posedness, we can therefore now rely on the classical inhomogeneous linear setting.

If the coefficients are bounded, the result then readily follows from Chen \textit{et al.} \cite{chen:zhan:zhao:17} who proved the complete strong well-posedness for stable driven SDEs, including the super-critical case up to $\alpha \in (0,2) $.
It can also be easily checked that the procedure developed in the homogeneous case by \cite{prio:12} for $\alpha\ge 1 $ can readily be extended to the current time dependent framework. The methodology therein indeed only depends on \textit{a priori} smoothness controls that have been established under the current assumptions in \cite{chau:meno:prio:19}. Hence, the methodology of \cite{prio:12} to derive strong uniqueness still applies and extends to the current framework with unbounded coefficients and some super-critical cases. 
\end{proof}

\newpage
\begin{remark}[On possible extensions]\label{remarks:possible:improvements}
Before going to the proofs of our main results, let us discuss some possible extensions of our main existence and uniqueness results.
\begin{itemize}
\item[$\bullet$] Let us first indicate that the global spatial H\"older continuity of the drift assumed in \A{B${}_H $} could be weakened to a \textit{local} H\"older continuity condition. Namely,  we could only impose that there exists $K_0$ s.t. uniformly in $(t,\mu)\in \R_+\times \mathcal P(\R^d) $,
$$|b(t, x,\mu)-b(t, y,\mu)|\le K_0|x-y|^{2\eta} , |x-y|\le 1.$$ 
This is precisely the condition appearing in \cite{chau:meno:prio:19}. The drawback of considering the above condition is that it would induce a localization of our arguments and therefore additional technicality. We preferred to avoid it to focus on the specific non-linear aspects. Anyhow, such an assumption would allow to consider drifts which have linear growth in space, like e.g. in the work by Mishura and Veretennikov \cite{mishura:veretenikov}.   
\item[$\bullet$] In the subcritical case $\alpha \in (1,2)$, we believe that the approach of \cite{chau:frik:18} could be adapted  in order to consider for the drift term a Lipschitz continuity condition w.r.t the total variation metric on the space of probability measures uniformly w.r.t. the time and space variables. In this context, one still has to perform the above fixed point procedure on the space $\mathscr{\boldsymbol A}_{s,T, \mu}$ now equipped with the total variation distance. This will concern future investigations.
\item[$\bullet$] Let us point out that the condition $\alpha > 2\eta$ appears in Theorem \ref{THM_WP_NL} only for integrability purposes in order to establish our Schauder estimates in the context of unbounded drift coefficients. In particular, it could also be removed by employing a localization technique\footnote{which would actually be the same that the one allowing to consider the local H\"older condition for the drift described in the first point} in the same spirit of the one employed in \cite{chau:meno:prio:19} or by assuming that the drift coefficient is globally bounded. 
\item[$\bullet$] Even though we only addressed the pure jump case, the methodology developed below (relying mainly on the control of the explosions of the gradient of the solution of a PDE with H\"older terminal condition and coefficients) would also apply in the diffusive stable setting corresponding to $\alpha=2 $. Coupled to the previously mentioned localization procedure, this could allow to consider non-linear drifts, locally H\"older continuous in space and with spatial linear growth as well as a diffusion coefficient also depending  on the measure argument.
\end{itemize}
\end{remark}

\section{Sensitivity of the semi-groups w.r.t. the measure argument}
\label{PROOF_GROSSE_PROP}
The main purpose of this section is to prove Proposition \ref{PROP_STAB}. To derive the key control \eqref{THE_EQ_SENSI_SG_MES} the main point is formally to expand one semi-group around the other and to exploit some Schauder like controls in the same spirit of those established in \cite{chau:meno:prio:19} under the current assumption \A{A${}_S $}. With the previous notations, defining for any fixed  $t\in [0,T] $ with $T\leq1$ small enough to be specified later on, $ (s,x)\in [0,t]\times \R^d $ and $ i\in \{1,2\}$, $u_i(s,x):= P_{s,t}^{\bP_i}h(x)$ it is clear that $u_i$ is a mild solution of equation:
\begin{equation}\label{EDP_U_I}
\begin{cases}
\Big(\partial_s +\mathscr{A}_s^{\bP_i}\Big)u_i(s, x) =0,\ (s, x)\in [0,t)\times \R^d,\\
u_i(t,x)=h(x),\ x\in \R^d.
\end{cases}
\end{equation}
with $\mathscr{A}_s^{\bP_i} $ as in \eqref{THE_GEN_STABLE} and $h\in \mathcal{C}^\gamma(\rr^d,\rr)$ where $\gamma \in (0,1)$ corresponds to the index of the underlying distance employed for the fixed point procedure described in the proof of Theorem \ref{THM_WP_NL}. Namely, for the proof of Proposition \ref{PROP_STAB}, that is, under the assumption $\alpha > 2\eta \vee (1-\frac{\eta}{2})$, we will take $\gamma=\eta$. This choice notably allows to derive the announced result in the subcritical case $\alpha \geq1$ and in the super-critical one up to $\alpha>\frac 45 $. In order to handle the supercritical case $\alpha <1$ under the sole condition $\alpha >2\eta \vee (1-\eta)$, we will take $\gamma =2\eta$ for the first iterate of the map $\bold{T}$ and then $\gamma=(1-\lambda)2\eta$ for a suitable choice of $\lambda \in [0,1]$ for the second iterate of $\bold{T}$. The main required modification for the latter procedure is precisely described in Appendix \ref{sec:modification:proof}. 

It now follows from the Schauder estimates established in the above reference that, if $h\in C^{\alpha+2\eta}(\R^d,\R) $ then $u_i\in L^\infty([0,t],C^{\alpha+2\eta}(\R^d,\R)) $. We carefully indicate that the results in \cite{chau:meno:prio:19} are stated in the super-critical case $\alpha<1 $ but still hold for $\alpha\in [1,2) $ (See Remark 6 therein). Also, the estimates of \cite{chau:meno:prio:19} are established for a constant \textit{diffusion} coefficient $\sigma $. It is anyhow also indicated in Remark 5 of that work that those estimates would extend to dynamics with a variable H\"older in space diffusion coefficient following the usual perturbation approach in that case (see e.g. \cite{miku:prag:14} or \cite{zhan:zhao:18}). Hence, the estimates are still valid under our current assumption \A{A${}_S $}. Furthermore, the techniques we develop below also emphasize how the dependence on the diffusion coefficient must be handled to derive Schauder estimates (see e.g. Lemmas \ref{NEW_LEMMA_R} and \ref{LEMME_SENSI_MES_ET_SPACE} below which precisely investigate the sensitivity of the non-local term with respect to a diffusion coefficient and the proof of Lemma \ref{CHARACTERIZATION_OF_GRADIENT_EXPLOSION_AND_HOLDER_MODULUS}).
%\footnote{\textcolor{red}{Il faut faire une ref a la remarque pour le $\sigma $. Dans \cite{chau:meno:prio:19} on a traite le cas d'un bruit additif pour simplifier l'ecriture. Cela ne pose pas de difficultes additionnelles. Quitte a refaire la preuve ici en annexe}}. 

The point now is that $h$ is here only assumed to be in $C^{\gamma}(\R^d,\R) $ (and not in $C^{\alpha+2\eta}(\R^d,\R) $). For the analysis let us now consider a mollification $h^\delta\in C^{\alpha+2\eta}(\R^d,\R) $ of $h \in C^{\gamma}(\R^d,\R)$, i.e. $h^\delta=h\star \rho_\delta $ with $\star $ standing for the convolution and for $\rho_\delta(\cdot)=\delta^{-d}\rho(\frac{\cdot}{\delta}) $ where $\rho:\R^d\rightarrow \R_+ $ is a smooth compactly supported function. In particular, there exists $C\ge 1$ s.t. $\|h^\delta\|_{{\mathcal C}^{{\gamma}}}\le C $ and $\|h^\delta-h\|_{{\mathcal C}^{{\gamma}}} \underset{\delta \rightarrow 0}{\longrightarrow } 0$. 
Denoting by $u_i^\delta(s,x):=P_{s,t}^{\bP_i}h^\delta(x) $, it is then clear from the uniform H\"older continuity of $h$ that $(u_i^\delta-u_i)(s, x)=P_{s, t}^{\bP_i}(h^\delta-h)(x)\underset{\delta \rightarrow 0}{\longrightarrow} 0 $, uniformly in $(s,x)\in [0,t]\times \R^d $. Hence, it suffices to establish \eqref{THE_EQ_SENSI_SG_MES} for the mollified final functions uniformly w.r.t. the smoothing parameter $\delta$.

To this end, observe that  $u_i^\delta(s,x)$ satisfies
\begin{equation}\label{EDP_U_I_DELTA}
\begin{cases}
\Big(\partial_s +\mathscr{A}_s^{\bP_i}\Big)u_i^\delta(s, x) =0,\ (s,x)\in [0,t)\times \R^d,\\
u_i^\delta(t,x)=h^\delta(x),\ x\in \R^d,
\end{cases}
\end{equation}

\noindent and $u_i^\delta\in L^\infty([0,t],C^{\alpha+\gamma}(\R^d,\R)) $. 

To compare both semigroups, we now write the PDE satisfied by $w(s, x):=w^\delta(s, x)=(u_1^\delta-u_2^\delta)(s, x)$. Namely,
\begin{equation}
\label{PDE_W}
\left\{
\begin{split}
\Big(\partial_s +\mathscr{A}_s^{\bP_1})w(s,x)&=-\Big( \big[b(s,x,\bP_1)-b(s,x,\bP_2) \big]\cdot D_xu_2^\delta(s,x)\\
&\hspace*{.5cm}+(L_s^{\bP_1,\alpha}-L_s^{\bP_2,\alpha})u_2^\delta(s,x) \Big)=: -H_{\bP_1,\bP_2} u^{\delta}_2(s,x),\ (s,x)\in [0,t)\times \R^d,\\
w(t,x)&=0,\ x\in \R^d.
\end{split}
\right.
\end{equation}

The important point is now to observe that, proceeding this way we have isolated the terms with the difference of the reference measures in the source (term $H_{\bP_1,\bP_2} u^{\delta}_2$ in \eqref{PDE_W}). Hence, the above left hand side depends only on one reference measure, here $\bP_1 $, which will serve in order to introduce a proxy associated with the unbounded coefficients of $\mathscr{A}_s^{\bP_1} $ similarly to what was previously performed in \cite{chau:meno:prio:19} to derive quantitative estimates. Observe that, for the  PDE \eqref{EDP_U_I_DELTA} with mollified terminal condition, $Du_1^\delta$ and $\big(L_s^{\bP_i,\alpha}u_1^\delta\big)_{i\in \{1,2\}}$ are indeed well defined pointwise.

The following results will play a central role in our analysis. The first lemma controls the explosions of the gradient and its H\"older modulus of continuity for the solution of \eqref{EDP_U_I_DELTA}. The second one gives controls  of the source $H_{\bP_1,\bP_2} u^{\delta}_2 $ in \eqref{PDE_W} in terms of the distance between the two reference measures, up to multiplicative and integrable time singularities.

\begin{lem}[Regularity of the semi-groups for a H\"older source]\label{CHARACTERIZATION_OF_GRADIENT_EXPLOSION_AND_HOLDER_MODULUS}
Assume that $\alpha + \gamma>1$ and $\gamma \in (0,2\eta]$. With the notations of \eqref{EDP_U_I_DELTA}, there exists a constant $C\ge 1$ s.t. for any $\bP_i \in {\mathcal C}([0, t],{\mathcal P}(\R^d))$, $i=1,2$, for all $(s, x)\in [0,t)\times \R^d $ and all $\delta>0 $:
\begin{equation}\label{BD_CTR_U_I_DELTA_GRADIENT}
|u_i^\delta (s, x)|+(t-s)^{%(
\frac 1\alpha-\frac{\gamma}{\alpha}
%)\I_{\alpha\le 1}
}|D_xu_i^\delta(s, x)|\le C.
\end{equation}
Also, for all $\varepsilon\in (0,1-(\frac 1\alpha-\frac{\gamma}{\alpha})) $, there exists $C_\varepsilon$ s.t. for all $(s,x,x')\in [0,t)\times (\R^d)^2 $:
\begin{equation}
\label{BD_CTR_U_I_DELTA_MOD_HOLDER_GRADIENT}
(t-s)^{1-\varepsilon }\frac{|D_x u_i^\delta(s, x)-D_x u_i^\delta(s, x')|}{|x-x'|^{\vartheta}}\le C_\varepsilon,\ \vartheta:=\vartheta(\alpha,\gamma,\varepsilon)=\alpha\Big[1-\varepsilon-\big(\frac 1\alpha-\frac{\gamma}{\alpha}\big)\Big]=\alpha(1-\varepsilon)-1+\gamma.
\end{equation}
\end{lem}
\begin{REM}[About the explosion rates of the gradient]
Observe from equation \eqref{BD_CTR_U_I_DELTA_GRADIENT} that  the gradient explodes at an integrable rate in time, precisely from the condition $\alpha+\gamma >1 $ (of course this is meaningful only in the super-critical case $\alpha<1 $, since when $\alpha\ge 1 $ it is readily fulfilled by any $\gamma>0 $). Equation \eqref{BD_CTR_U_I_DELTA_MOD_HOLDER_GRADIENT} then specifies how far we can obtain H\"older moduli in space that explodes at an integrable rate. Namely, the gradient can be seen as a function of regularity order $\alpha+\gamma-1 $. We manage to control the corresponding H\"older modulus with an integrable singularity in time up to an additional \textit{small} factor $\varepsilon \alpha $ (exponent $\vartheta $ above), where in turn $\varepsilon $ quantifies the distance to explosion.
\end{REM}

\begin{lem}[Controls for the source involving the different measures]\label{LEM_DIFF_H}
Let $u_i^\delta $ solve \eqref{EDP_U_I_DELTA}. Taking then, with the notation introduced in \eqref{BD_CTR_U_I_DELTA_MOD_HOLDER_GRADIENT},  
%$\varepsilon $ small enough to guarantee 
$\vartheta>\alpha-1 $, it holds:
\begin{equation}
\label{CTR_DIST_OP}
| H_{\bP_1,\bP_2} u_i^\delta(s, x)|\le C\bd_{2\eta}(\bP_1(s),\bP_2(s))\Big((t-s)^{-1+\varepsilon}\I_{\left\{\alpha \ge 1\right\}}+(t-s)^{-\big(\frac 1\alpha-\frac{\gamma}{\alpha}\big)}%\I_{\alpha<1}
\Big),
\end{equation}

and for all $ (x,x')\in (\R^d)^2$, for any $\lambda \in [0,1]$
\begin{align}
|& H_{\bP_1,\bP_2} u_i^\delta(s, x)  - H_{\bP_1,\bP_2} u_i^\delta(s, x')|\notag \\
& \le C\bd_{(1-\lambda) 2\eta}(\bP_1(s),\bP_2(s)) |x-x'|^{2\lambda\eta} \Big((t-s)^{-1+\varepsilon}\I_\seq{\alpha \ge 1}+(t-s)^{-\big(\frac 1\alpha-\frac \gamma \alpha\big)}
%\I_\seq{\alpha<1}
\Big) \label{CTR_DIST_OP_HOLDER}\\
& \quad + C \bd_{2\eta}(\bP_1(s),\bP_2(s)) (|x-x'|^{\vartheta} + |x-x'|^{1+\vartheta-\alpha}) (t-s)^{-1+\varepsilon}.\notag 
\end{align}
\end{lem}

\subsection{Proxy and a priori controls}
\label{THE_SEC_FOR_W}
For given \textit{freezing} parameters $(\tau,\xi)\in [0,T]\times \R^d$ to be specified later on, introduce the flow:
% \textcolor{red}{From P.E. to S. and E. : we do not define $\theta$ for %$s<\tau$ which is sometime needed in the sequel. Is it trivial or do you %think that we should define it ?}
 \begin{equation}\label{DEF_FLOW}
\theta_{s,\tau}(\xi)=\xi+\int_{\tau}^sb(v,\theta_{v,\tau}(\xi),\bP_1(v))dv, \;\; s \ge \tau,
\end{equation} 
 $\theta_{s, \tau}(\xi) = \xi$ for $s < \tau$. In the current setting (H\"older drift in space), this object exists from the Peano theorem but is not necessarily unique.
 Rewrite then \eqref{PDE_W} as follows
\begin{equation}
\label{FROZEN_Asso_PDE}
\left\{
\begin{array}{l}
\Big(\partial_v   +  \tilde L_v^{\bP_1,\alpha,(\tau,\xi)}  \Big)w(v, x) + b(v,\theta_{v,\tau}(\xi),\bP_1(v)) \cdot D_x w(v,x) \\
\hspace*{.5cm}= -\Big(H_{\bP_1,\bP_2} u^{\delta}_2(v, x)+R_{\bP_1}^{\alpha,(\tau,\xi)}(v,x)\Big),\ (v,x)\in [0,t)\times \R^d, \\
w(t,x) =0,\quad \text{on }\R^d,
\end{array} \right. 
\end{equation}
where for all $\varphi \in C_0^2(\R^d,\R) $
\begin{equation}
\label{DEF_OPERATEUR_GELE}
\tilde L_v^{\bP_1,\alpha,(\tau,\xi)} \varphi(x):={\rm {p.v.}}\int_{\R^d} \big(\varphi(x+\sigma(v,\theta_{v,\tau}(\xi),\bP_1(v)) \zeta)-\varphi(x) \big) \nu(d\zeta)
\end{equation}
and
\begin{equation}\label{THE_TERM_PERTURB_FIXED_MEASURE}
R_{\bP_1}^{\alpha,(\tau,\xi)}(v,x):=\big(b(v,x,\bP_1(v))-b(v,\theta_{v,\tau}(\xi),\bP_1(v)) \big) \cdot D_x w(v, x)+\Big(L_s^{\bP_1,\alpha}-\tilde L_s^{\bP_1,\alpha,(\tau,\xi)}\Big) w(v,x).
\end{equation}

Precisely, the contribution $R_{\bP_1}^{\alpha,(\tau,\xi)}$ corresponds, for the measure argument fixed to $\bP_1 $, to the difference of the generator with variable coefficients and the proxy one which is frozen along the deterministic flow \eqref{DEF_FLOW}.
The following Lemma proved in Appendix \ref{SEC_PROOF_NEW_LEMMA_R} gives quantitative bounds on this remainder term.
 \begin{lem}[Bounds for the difference of the generators for a fixed measure argument]
\label{NEW_LEMMA_R}
There exists a constant $C$ s.t. for all $(\tau,\xi),\ (v, y)\in [0,T]\times \R^d $:
\begin{eqnarray}\label{CTR_PT_R_LEM}
 \big|R_{\bP_1}^{\alpha,(\tau,\xi)}(v,y)\big|
\le C\Big(\|b(v,\cdot,\bP_1(v))\|_{\dot {\mathcal C}^{2\eta}}|D_y w(v, y)|+ \|\sigma(v,\cdot,\bP_1(v))\|_{{\mathcal C}^{2\eta}}\big[ (\|w\|_\infty+\|D_x w(v,\cdot)\|_\infty)\I_{\left\{\alpha<1\right\}}\notag\\
+\|D_xw(v,\cdot)\|_{{\mathcal C}^{\vartheta}}\I_{\left\{\alpha>1\right\}}+\|w(v,\cdot)\|_{{\mathcal C}^{1+\vartheta}}\I_{\left\{\alpha=1\right\}}\big]\Big)|y-\theta_{v,\tau}(\xi)|^{2\eta}.
\end{eqnarray}
 \end{lem}

Also, under \A{A${}_S $}, it is clear that the time-dependent operator $\tilde L_v^{\bP_1,\alpha,(\tau,\xi)}+  b(v, \theta_{v,\tau}(\xi),\bP_1 ) \cdot D_x $ generates a family of transition probability (or two parameter transition semi-group) $\big(\tilde P_{s,v}^{\bP_1,\alpha,(\tau,\xi)})_{0\le s\le v\le t}$. Because of the non-degeneracy conditions of \eqref{NON_DEG_STABLE} and \eqref{COND_UE_S} in \A{ND} and \A{UE}, the associated heat-kernel exists (see e.g. \cite{kolo:00}, \cite{wata:07} or Appendix \ref{APP_STABLE_CORRESPONDENCE} below)
 and, for fixed $0\le s< v\le t $, it writes: 
\begin{equation}\label{CORRESP_SHIFTED}
\tilde p^{\bP_1,\alpha,(\tau,\xi)}(s,v,x,y)= p_{\Theta_{s,v,(\tau,\xi),\bP_1}}\Big(y-m_{v,s}^{(\tau,\xi)}(x)\Big),
\end{equation}
where we denoted:
\begin{eqnarray}
m_{v,s}^{(\tau,\xi)}(x)&:=&x+\int_s^v b(r,\theta_{r,\tau}(\xi),\bP_1(r))dr,\label{FROZEN_FLOW}\\
\Theta_{s,v,(\tau,\xi),\bP_1}&:=&\int_s^v \sigma(r,\theta_{r,\tau}(\xi),\bP_1(r))dZ_r,\label{DEF_THETHA_RV_FROZEN}
\end{eqnarray}

\noindent and where $Z$ is a stable process  with L\'evy measure $\nu $ defined on some probability space $(\Omega, {\mathcal A},\P) $.

Observe that from the definition \eqref{FROZEN_FLOW} of the \textit{shift} $m_{v,s}^{(\tau,\xi)}(x) $ we have the important property 
 \begin{equation}
\label{CORRES_LINEARIZED_FLOW_AND_FLOW}
 m_{v,s}^{(\tau,\xi)}(x)|_{(\tau,\xi)=(s,x)}=\theta_{v,s}(x).
 \end{equation}

Let us first give some bounds on the density function of the stochastic integral \eqref{DEF_THETHA_RV_FROZEN} and its derivatives. They are somehow classic under the absolute continuity condition \A{AC} for the L\'evy measure (see e.g. \cite{kolo:00}). For the sake of completeness, the proof of the following result is postponed to Appendix \ref{APP_STABLE_CORRESPONDENCE}. .

\begin{lem}[Controls on the proxy density and its derivatives]\label{DER_DENS_PROXY}
Assume \A{ND} and \A{UE} hold. Then, there exists a constant $C$ and a probability density $\bar q :\R_*^+\times \R^d\rightarrow \R$ satisfying for all $\gamma<\alpha $ the integrability property 
\begin{equation}
\label{INT}
\int_{\R^d} |z|^{\gamma}\bar q(t-s,z) dz
\le C(t-s)^{\frac \gamma \alpha},
\end{equation}
s.t. for all multi-index $\beta, |\beta|\le 3 $ and for any $0\le s<t, \, (\tau,\xi)\in \R_+\times \R^d, \bmu\in {\mathcal C}(\R^+,{\mathcal P}(\R^d)), \, y\in \R^d $:
\begin{equation}
\label{THE_CTR_DER_DENS_PROXY}
|D^\beta_y p_{\Theta_{s,t,(\tau,\xi),\bmu}}(y)|\le \frac{C}{(t-s)^{\frac{|\beta|}{\alpha}}}\bar q(t-s,y). 
\end{equation}
Also, any \textit{diagonal perturbation} does not affect the previous estimate. Namely, for $z\in \R^d$ such that $|z|\le (t-s)^{\frac 1\alpha} $, it holds:
\begin{equation}
\label{THE_CTR_DER_DENS_PROXY_DIAG_PERTURB}
|D^\beta_y p_{\Theta_{s,t,(\tau,\xi),\bmu}}(y+z)|\le \frac{C}{(t-s)^{\frac{|\beta|}{\alpha}}}\bar q(t-s, y). 
\end{equation}

\end{lem}

As a very important consequence of this lemma, we also derive from \eqref{CORRESP_SHIFTED}-\eqref{FROZEN_FLOW} that for all $(x,y)\in (\R^d)^2 $:
\begin{eqnarray}\label{THE_CORRESP_DER_INT_STO}
|D_x^\beta \tilde p^{\bP_1,\alpha,(\tau,\xi)}(s,v,x,y)|&:=&|D_x^\beta p_{\Theta_{s,v,(\tau,\xi),\bP_1}}\big(y- m_{v,s}^{(\tau,\xi)}(x)\big)| \le \frac{C}{(v-s)^{\frac{|\beta|}{\alpha}}}\bar q\big(v-s,y-m_{v,s}^{(\tau,\xi)}(x)\big).
\end{eqnarray}

\subsection{Proof of Proposition \ref{PROP_STAB}}
Throughout this subsection, we set $\gamma = \eta$, that is, $h\in {\mathcal C}^{\eta}(\R^d,\R)$ with $\|h\|_{{\mathcal C}^\eta}\le 1 $. With the notations and controls of the previous subsection at hand, we can now restart from \eqref{FROZEN_Asso_PDE} to write for all $(s,x)\in [0,t]\times \R^d $,
\begin{eqnarray}
w(s,x)&=&\int_s^t dv \Big(\tilde  P_{s,v}^{\bP_1,\alpha,(\tau,\xi)}\big(H_{\bP_1,\bP_2} u^{\delta}_2(v,\cdot)+R_{\bP_1}^{\alpha,(\tau,\xi)}(v,\cdot)\big)\Big)(x)\notag\\
&=&\int_s^t dv \int_{\R^d }\tilde  p^{\bP_1,\alpha,(\tau,\xi)}(s,v,x,y) \big(H_{\bP_1,\bP_2} u^{\delta}_2(v,y)+R_{\bP_1}^{\alpha,(\tau,\xi)}(v,y)\big)dy.\label{PREAL_CIRC}
\end{eqnarray}
From \eqref{PREAL_CIRC} and the definition of $R_{\bP_1}^{\alpha,(\tau,\xi)} $ in \eqref{THE_TERM_PERTURB_FIXED_MEASURE}, we see that $w$ appears in both sides of the equality. The point is now to perform a circular argument to control $w$. To this end, we also point out that the a priori controls of Lemma \ref{CHARACTERIZATION_OF_GRADIENT_EXPLOSION_AND_HOLDER_MODULUS} will be useful to control $H_{\bP_1,\bP_2}u_2^\delta $ introduced in \eqref{PDE_W}. To complete the proof of Proposition \ref{PROP_STAB} we aim at showing the following: 
\begin{eqnarray}
|w(s,x)|\le  C\bd_{s,t,\eta}(\bP_1,\bP_2)(t-s)^{\tilde \zeta},\label{THE_PREAL_CTR_ON_W}
\end{eqnarray}

\noindent for some ${\tilde \zeta}>0$, which precisely gives \eqref{THE_EQ_SENSI_SG_MES} for $h$ replaced by $h^\delta $.

Write first that, from \A{B${}_H $}, the definition of $H_{\bP_1,\bP_2}u_2^\delta$ and Lemma \ref{CHARACTERIZATION_OF_GRADIENT_EXPLOSION_AND_HOLDER_MODULUS} (see equations \eqref{BD_CTR_U_I_DELTA_MOD_HOLDER_GRADIENT} and \eqref{CTR_DIST_OP} with $\gamma = \eta$), it holds that there exists a constant $C$ s.t. for all $(v,y)\in [s,t)\times \R^d $:
\begin{eqnarray}\label{CTR_H_BSUP}
|H_{\bP_1,\bP_2} u^{\delta}_2 (v, y)|\le C \bd_\eta(\bP_1(v),\bP_2(v)) \Big((t-v)^{-1+\varepsilon}\I_{\left\{ \alpha \ge 1 \right\}}+(t-v)^{-\big(\frac 1\alpha-\frac \eta \alpha\big)} \Big),
\end{eqnarray}
using as well for the above inequality that, from the definitions  in \eqref{DIST_CONT_MEAS}-\eqref{DIST_CONT_MEAS_FLOWS}, 
\begin{equation}\label{DOM_DIST}
\forall \eta\in (0,\frac 12],\ \bd_{2\eta}(\bP_1(v),\bP_2(v))\le \bd_\eta(\bP_1(v),\bP_2(v)) .
\end{equation} 
This bound will be frequently used in the sequel.

On the other hand, from Lemma \ref{NEW_LEMMA_R} %\eqref{THE_TERM_PERTURB_FIXED_MEASURE} 
and \eqref{THE_CORRESP_DER_INT_STO}:%, \textcolor{red}{similarly to \eqref{THE_CTR_PERTURB_1}}:
\begin{eqnarray}
\Big|\int_s^t dv \int_{\R^d }\tilde  p^{\bP_1,\alpha,(\tau,\xi)}(s,v,x,y) R_{\bP_1}^{\alpha,(\tau,\xi)}(v,y)dy\Big|\label{THE_CTR_PERTURB_1}\\
\le C\int_s^t dv \int_{\R^d} dy\Big(\|b(v,\cdot,\bP_1(v))\|_{\dot {\mathcal C}^{2\eta}}|D_y w(v, y)|+ \|\sigma(v,\cdot,\bP_1(v))\|_{{\mathcal C}^{2\eta}}\big[ (\|w\|_\infty+\|D_x w(v,\cdot)\|_\infty)\I_{\left\{\alpha<1\right\}}\notag\\
+\|D_xw(v,\cdot)\|_{{\mathcal C}^{\vartheta}}\I_{\left\{\alpha>1\right\}}+\|w(v,\cdot)\|_{{\mathcal C}^{1+\vartheta}}\I_{\left\{\alpha=1\right\}}\big]\Big)|y-\theta_{v,\tau}(\xi)|^{2\eta} \bar q\big(v-s,y-m_{v,s}^{(\tau,\xi)}(x)\big),\notag
\end{eqnarray}
denoting here and for the rest of the section with a slight abuse of notation $\|w\|_\infty:=\sup_{v\in [s,t]}\|w(v,\cdot)\|_\infty $.

\noindent Hence,  choosing $(\tau,\xi)=(s,x) $ and exploiting \eqref{CORRES_LINEARIZED_FLOW_AND_FLOW} and Lemma \ref{DER_DENS_PROXY}, we derive:
\begin{eqnarray}
\Big|\int_s^t dv \int_{\R^d }\tilde  p^{\bP_1,\alpha,(\tau,\xi)}(s,v,x,y) R_{\bP_1}^{\alpha,(\tau,\xi)}(v,y)dy\Big| \Bigg|_{(\tau,\xi)=(s,x)}\notag\\
\le C\int_s^t dv (v-s)^{\frac{2\eta}{\alpha}}\Big( \|w\|_\infty\I_{\left\{\alpha\le 1\right\}}+\|D_xw(v,\cdot)\|_\infty+\|D_xw(v,\cdot)\|_{\dot {\mathcal C}^{\vartheta}}\I_{\left\{\alpha\ge 1\right\}}\Big).\label{PREAL_CIRCULAR_2}
\end{eqnarray}
From \eqref{PREAL_CIRCULAR_2}, \eqref{CTR_H_BSUP} and \eqref{PREAL_CIRC}, one derives:
\begin{eqnarray}
|w(s,x)|&\le& C\bigg(\bd_{\eta,s,t}(\bP_1,\bP_2)\big[(t-s)^{\varepsilon\wedge (1-(\frac{1}{\alpha}-\frac{\eta}{\alpha}))}\I_{\left\{\alpha\ge 1\right\}}+(t-s)^{(1-(\frac{1}{\alpha}-\frac{\eta}{\alpha}))}%\I_{\left\{\alpha<1\right\}}
\big]+(t-s)^{}\|w\|_\infty\I_{\left\{\alpha\le1\right\}}\notag\\
&&+\int_s^t dv (v-s)^{\frac{2\eta}{\alpha}}\Big( \|D_xw(v,\cdot)\|_\infty+\|D_xw(v,\cdot)\|_{\dot {\mathcal C}^{\vartheta}}\I_{\left\{\alpha\ge 1\right\}}\Big)\bigg).%\notag\\
\label{ALL_READY_FOR_CIRCULAR}
\end{eqnarray}
It is now clear that we actually have to estimate the gradient of $w$, i.e. we need to differentiate \eqref{PREAL_CIRC}. We will now split the rest of the analysis in function of the index $\alpha $. Indeed, for $\alpha>1 $, differentiating \eqref{PREAL_CIRC} yields an integrable singularity in time and the previous controls do not change much. The counterpart is that we also need to give a bound on the H\"older modulus of continuity.
On the other hand, for $\alpha \le 1 $, the induced time singularity is not integrable and it is crucial to equilibrate it through adequate cancellation techniques.

\subsubsection*{The case $\alpha\ge 1 $}%\textcolor{red}{Il faut rajouter la norme infinie de $w$ si $\alpha=1 $ ou dire que l'on ne fait que $\alpha>1 $ par commodite d'ecriture.}
From \eqref{PREAL_CIRC}, we can differentiate w.r.t $x$ and also use a cancellation technique for the source term $H_{\bP_1,\bP_2}u_2^\delta $. Namely,
\begin{equation}
D_xw(s,x)
=\int_s^t dv \int_{\R^d } D_x \tilde p^{\bP_1,\alpha,(\tau,\xi)}(s,v,x,y) \Big(\big[H_{\bP_1,\bP_2}u_2^\delta(v, y)-H_{\bP_1,\bP_2}u_2^\delta(v,\theta_{v,\tau}(\xi))\big]+R_{\bP_1}^{\alpha,(\tau,\xi)}(v,y)\Big)dy.\label{DIFF_WITH_CANCEL}
\end{equation}
Similarly to \eqref{CTR_H_BSUP}, we derive from \A{B${}_H $}, equation \eqref{CTR_DIST_OP_HOLDER} (with $\lambda = 1/2$) of Lemma 
\ref{CHARACTERIZATION_OF_GRADIENT_EXPLOSION_AND_HOLDER_MODULUS} and \eqref{DOM_DIST} that:
 \begin{eqnarray}\label{CTR_H_BSUP_HOLDER_alpha_GE_1_WITH_ALL}
|H_{\bP_1,\bP_2}u_2^\delta(v, y) - H_{\bP_1,\bP_2}u_2^\delta(v,\theta_{v,\tau}(\xi))|&\le& C\bd_\eta(\bP_1(v),\bP_2(v)) \Big[ (|y-\theta_{v,\tau}(\xi)|^\eta\wedge 1) \Big((t-v)^{-1+\varepsilon}+(t-v)^{-\big(\frac 1\alpha-\frac \eta \alpha\big)}\Big)\notag\\
&&+(t-v)^{-1+\varepsilon}\big((|y-\theta_{v,\tau}(\xi)|^\vartheta\wedge 1)+(|y-\theta_{v,\tau}(\xi)|^{1+\vartheta-\alpha}\wedge 1)\big) \Big]\\
&\le& C\bd_\eta(\bP_1(v),\bP_2(v)) \Big[ (|y-\theta_{v,\tau}(\xi)|^\eta\wedge 1) \Big((t-v)^{-1+\varepsilon}+(t-v)^{-\big(\frac 1\alpha-\frac \eta \alpha\big)}\Big)\notag\\
&&+(t-v)^{-1+\varepsilon}(|y-\theta_{v,\tau}(\xi)|^{1+\vartheta-\alpha}\wedge 1)\Big],\notag
\end{eqnarray}
recalling that since $\alpha\ge  1$, $1+\vartheta-\alpha\le \vartheta $ for the last inequality. From the definition in \eqref{BD_CTR_U_I_DELTA_MOD_HOLDER_GRADIENT} we thus get that, for $\varepsilon $ meant to be small,
\begin{equation}
\label{ALPHA_GE_1_EFF_REG_PIRE_POUR_CANCEL}
1+\vartheta-\alpha= 1-\alpha +\alpha\Big[1-\varepsilon-\big(\frac 1\alpha-\frac \eta\alpha\big)\Big]=\eta-\varepsilon \alpha.
\end{equation}
We thus obtain the following control
\begin{equation}\label{CTR_H_BSUP_HOLDER}
|H_{\bP_1,\bP_2}u_2^\delta(v,y)-H_{\bP_1,\bP_2}u_2^\delta(v,\theta_{v,\tau}(\xi))| \le C\bd_\eta(\bP_1(v),\bP_2(v)) |y-\theta_{v,\tau}(\xi)|^{\eta-\varepsilon \alpha} (t-v)^{-1+\varepsilon}.
\end{equation}

Hence, from \eqref{DIFF_WITH_CANCEL}, \eqref{CTR_H_BSUP_HOLDER}, %\eqref{THE_TERM_PERTURB_FIXED_MEASURE} 
Lemma \ref{NEW_LEMMA_R}
and \eqref{THE_CTR_DER_DENS_PROXY}, we get:
\begin{eqnarray}
|D_xw(s, x)|\label{THE_CTR_PERTURB_1_GRAD}\\
\le C\Big(\bd_{\eta,s,t}(\bP_1,\bP_2)\int_{s}^t \frac{dv}{(v-s)^{\frac 1\alpha}}(t-v)^{-1+\varepsilon}\int_{\R^d } dy |y-\theta_{v,\tau}(\xi)|^{\eta-\varepsilon \alpha} \bar q\big(v-s,y-m_{v,s}^{(\tau,\xi)}(x)\big) \notag\\
+\int_s^t \frac{dv}{(v-s)^{\frac 1\alpha}} \int_{\R^d} dy\Big(\|b(v,\cdot,\bP_1(v))\|_{\dot {\mathcal C}^{2\eta}}|D_y w(v, y)|\notag\\
+ \|\sigma(v,\cdot,\bP_1(v))\|_{{\mathcal C}^{2\eta}}\big[ %(\|w\|_\infty+\|D_x w(v,\cdot)\|_\infty)\I_{\alpha<1} On est sur \alpha\ge 1
%\notag\\
\|D_xw(v,\cdot)\|_{{\mathcal C}^{\vartheta}}\I_{\alpha>1} 
+\|w(v,\cdot)\|_{{\mathcal C}^{1+\vartheta}}\I_{\alpha=1}\big]\Big)|y-\theta_{v,\tau}(\xi)|^{2\eta} \bar q\big(v-s,y-m_{v,s}^{(\tau,\xi)}(x)\big)\Big).\notag
\end{eqnarray}

Taking $(\tau,\xi)=(s,x) $, it follows again from \eqref{CORRES_LINEARIZED_FLOW_AND_FLOW} and Lemma \ref{DER_DENS_PROXY} that:
\begin{eqnarray}
|D_x w(s, x)|&\le& C\bigg(\bd_{\eta,s,t}(\bP_1,\bP_2) \int_s^t dv(v-s)^{-\frac 1\alpha+\frac \eta \alpha-\varepsilon} (t-v)^{-1+\varepsilon}\notag\\
&&+\int_s^t dv (v-s)^{-\frac 1\alpha+\frac{2\eta}{\alpha}}\Big( \|D_x w(v,\cdot)\|_\infty+\|D_xw(v,\cdot)\|_{\dot {\mathcal C}^{\vartheta}}+\|w(v,\cdot)\|_{\infty}\I_{\alpha=1}\Big)\bigg)\notag\\
&\le &C\bigg(\bd_{\eta,s,t}(\bP_1,\bP_2) (t-s)^{-\frac 1\alpha+\frac \eta \alpha}%+(t-s)^{1-(\frac1\alpha-\frac {2\eta}\alpha)}\|Dw\|_\infty
\notag\\
&&+\int_s^t dv (v-s)^{-\frac 1\alpha+\frac{2\eta}{\alpha}}\Big(\|D_xw(v,\cdot)\|_\infty+\|D_x w(v,\cdot)\|_{\dot {\mathcal C}^{\vartheta}}+\|w(v,\cdot)\|_{\infty}\I_{\alpha=1}\Big)\bigg).\label{PREAL_CTR_ALPHA_GE_1_GRAD}
\end{eqnarray}
From  \eqref{PREAL_CTR_ALPHA_GE_1_GRAD}, we also need to estimate the gradient through a circular argument.

Let us now make a short comment before proceeding further. We cannot expect for the gradient of $w$ a better behaviour than the one provided by the \textit{main} term of the perturbative expansion \eqref{PREAL_CIRC}, i.e. the one associated with $ H_{\bP_1,\bP_2}u_2^\delta $. Observe now that 
we get: $-\frac 1\alpha+\frac{\eta}{\alpha}=:\zeta<0$. This exponent being negative, we cannot expect to have a pointwise bound for the gradient of $w$. Having in mind that we want to keep track in the above r.h.s. of the product of the distance $\bd_{\eta,s,t}(\bP_1,\bP_2) $ and a contribution of type $(t-s)^{\bar \zeta},\ \bar \zeta>0$, which provides a smoothing in time effect, we will therefore investigate the behavior of  
\begin{equation}\label{DEF_PHI}
\Phi(v):=(t-v)^{\Xi}\sup_{x\in \R^d}|Dw(v,x)|, \ \Xi:=\frac 1\alpha-\frac \eta\alpha+\frac \varepsilon 2,\ v\in [s,t],
\end{equation}
To investigate the explosion of the corresponding H\"older modulus, let us introduce as well 
 \begin{equation}\label{DEF_PSI}
 \Psi(v):=(t-v)^{\Xi+\frac\vartheta\alpha}\|Dw(v,\cdot)\|_{{\mathcal C}^\vartheta}, \ v\in [s,t].
 \end{equation}
 Let us note that $\sup_{v \in [s, t]} \Phi(v) < \infty$ and $\sup_{v \in [s, t]} \Psi(v) < \infty$. Indeed, from the definition of $H_{\bP_1,\bP_2}u_2^\delta $ in \eqref{PDE_W},  since from \cite{chau:meno:prio:19} $u_2^\delta \in L^\infty\big([0,T],C^{\alpha+2\eta}(\R^d,\R)\big)$ (with corresponding norm which a priori explodes with $\delta $), it holds that, if $b$ is bounded in space, $H_{\bP_1,\bP_2}u_2^\delta \in L^\infty\big([0,T],C^{2\eta \wedge (\alpha+2\eta-1)}(\R^d,\R)\big) $. Hence, the previously invoked Schauder estimates still apply for $w$ solving \eqref{PDE_W} (again with corresponding $L^\infty\big([0,T],C^{\alpha+2\eta\wedge (\alpha+2\eta-1)}(\R^d,\R)\big) $ norm exploding with $\delta $). If now $b$ is not bounded in space, recall that \textit{a priori}, we assumed in \A{B${}_H $} that $b(t,\cdot,\mu)\in \dot C^{2\eta}(\R^d,\R^d) $, we would then have $H_{\bP_1,\bP_2}u_2^\delta \in L^\infty\big([0,T],\dot C^{2\eta \wedge (\alpha+2\eta-1)}(\R^d,\R)\big) $. In that case, from the integrability constraint $\alpha>2\eta $, the arguments of \cite{chau:meno:prio:19} could be reproduced to derive the finiteness of $\sup_{v \in [s, t]} \Phi(v), \sup_{v \in [s, t]} \Psi(v) $ which involve the gradient which is in that case bounded whereas the function itself is not (see e.g. Krylov and Priola \cite{kryl:prio:10} for similar issues in the diffusive setting).

From the previous definitions, a key point, for our circular argument to work is that $\Xi+\frac \vartheta\alpha<1 $. Indeed, from \eqref{PREAL_CTR_ALPHA_GE_1_GRAD}, this is necessary in order to make the time singularities generated by $\Phi$ and $\Psi$ integrable. We now recall from the definition of $\vartheta $ in \eqref{BD_CTR_U_I_DELTA_MOD_HOLDER_GRADIENT} that:
 \begin{equation}\label{EQ_XI_PLUS_THETA}
 \Xi+\frac \vartheta\alpha=\frac 1\alpha-\frac \eta\alpha+\frac \varepsilon2+1-\varepsilon-(\frac 1\alpha-\frac \eta\alpha)=1-\frac \varepsilon2<1.
 \end{equation}

%Hence, it would be tempting to choose the largest possible $\varepsilon$ to guarantee $\Xi+\frac \theta\alpha<1 $. On the other hand, in the current case $\alpha>1 $, we also need to fulfill the constraint $\theta>\alpha-1 $ this means that the largest possible $\varepsilon $ writes $\varepsilon=\frac \eta \alpha-o_2 $, for $o_2 $ meant to be small. This eventually yields:
%\begin{equation}
%\Xi+\frac \theta\alpha=1-(\frac \eta\alpha-o_2)(2-o_1)+\frac \eta\alpha=1-\frac \eta\alpha-o_3<1,
%\end{equation}
% since $o_3$ can indeed be chosen small enough.
% 
With these notations at hand, we can from \eqref{ALL_READY_FOR_CIRCULAR} get rid of the supremum norm $\|w\|_\infty $ appearing in the case $\alpha=1 $ and upper bound it in terms of the functions $\Phi,\Psi $. Namely, for $\varepsilon $ and $T$ small enough (so that $t-s<1/2$),
\begin{eqnarray}
\|w\|_\infty&\le& C\bigg(\bd_{\eta,s,t}(\bP_1,\bP_2)(t-s)^{\varepsilon}
+(t-s)^{}\|w\|_\infty\I_{\left\{\alpha=1\right\}}\notag\\
&&+\int_s^t dv (v-s)^{\frac{2\eta}{\alpha}}\Big((t-v)^{-\frac \vartheta\alpha}\Phi(v)+(t-v)^{-(\Xi+\frac \vartheta\alpha)}\Psi(v)\Big)\bigg).\notag\\
&\le &C\bigg(\bd_{\eta,s,t}(\bP_1,\bP_2)(t-s)^{\varepsilon}
+\sup_{v\in [s,t]}\Phi(v) + \sup_{v\in [s,t]}\Psi(v)\Bigg).
\label{ALL_READY_FOR_CIRCULAR_SUP}
\end{eqnarray}
 
 \noindent up to a modification of the constant $C$ for the last inequality.

 We now have to distinguish the \textit{diagonal} and \textit{off-diagonal} regimes w.r.t. the current considered times. Namely, for given $(x,x')\in (\R^d)^2$, if $|x-x'|\ge (t-s)^{\frac 1\alpha} $ we say that the \textit{off-diagonal} regime holds. In this case, we readily get from \eqref{PREAL_CTR_ALPHA_GE_1_GRAD} and \eqref{ALL_READY_FOR_CIRCULAR_SUP}:
\begin{eqnarray}
(t-s)^{\frac \vartheta\alpha}\frac{|D_x w(s, x)-D_x w(s, x')|}{|x-x'|^\vartheta}&\le& |Dw(s,x)-Dw(s,x')|\le |Dw(s,x)|+|Dw(s,x')|\notag\\ 
&\le &C\bigg(\bd_{\eta,s,t}(\bP_1,\bP_2) (t-s)^{\zeta}\notag\\
&&+\int_s^t dv (v-s)^{-\frac 1\alpha+\frac{2\eta}{\alpha}}\Big((t-v)^{- \Xi}\Phi(v)+(t-v)^{-(\Xi+\frac \vartheta\alpha)}\Psi(v)\Big)\bigg),\notag\\
(t-s)^{\Xi+\frac \vartheta\alpha}\frac{|D_x w(s, x)-D_x w(s, x')|}{|x-x'|^\vartheta}&\le &C\bigg(\bd_{\eta,s,t}(\bP_1,\bP_2) (t-s)^{\Xi+\zeta}+(t-s)^{1-\frac 1\alpha+\frac{2\eta}\alpha}\sup_{v\in [s,t]}\Phi(v)\notag\\
&&+(t-s)^{1-(\frac 1\alpha+\frac \vartheta\alpha)+\frac{2\eta}\alpha}\sup_{v\in [s,t]}\Psi(v)\bigg).\label{PREAL_CTR_ALPHA_GE_1_HOLD_HD}
\end{eqnarray}
Let us now turn to the global \textit{diagonal regime}, $|x-x'| \le (t-s)^{\frac 1\alpha}$. This case is more subtle. Indeed, due to the time integration, i.e. for $v\in [s,t]$, there is a \textit{local} off-diagonal regime, namely for $v\in [s,s +c_0|x-x'|^{\alpha}]$ for a constant $c_0 \in (0,1)$ to be specified, and then a \textit{true} \textit{diagonal regime} w.r.t. the integration variable $v$ which is the one providing the smoothing effects through the heat kernel. The delicate point is that, in order to exploit the off-diagonal bounds, the natural choice of freezing spatial point consists in considering precisely the  spatial point in argument of the function. On the other hand, in the \textit{true} diagonal regime the freezing spatial point must be the same for the two quantities to expand. This naturally induces to consider an intermediate  quantity for the regime change.

Introduce for a couple $(\tau,\xi')\in \R_+\times \R^d $ of freezing points,  the following \textit{Green kernel}. : 
\begin{eqnarray*}
\forall 0\le s<r\le t,\ x\in \R^d,\ \tilde G_{s,r}^{\bP_1,\alpha,(\tau,\xi')} f(s,x)&:=&\int_s^{r} dv\int_{\R^{d}} dy\,  \tilde p^{\bP_1,\alpha,(\tau,\xi')}(s,v, x,y)f(v,y). %\label{GREEN_SUR_SEGMENT_EN_TEMPS}
\end{eqnarray*}
For  $x,x'\in \R^d $ s.t. $|x-x'|^\alpha\le  (t-s)$ we can write, similarly to the proof of Proposition 12 in \cite{chau:meno:prio:19} {(Duhamel formula with change of freezing point)}, that:
\begin{eqnarray}
&& D_x w(s,x')\notag\\
&=&\big(D_x\tilde G_{s,\tau_0}^{\bP_1,\alpha,(\tau,\xi')} (H_{\bP_1,\bP_2}u_2^\delta(s,x')\big)_{(\tau_0,\tau,\xi')=(t_0,s,x')}+\big(D_x\tilde G_{\tau_0,t}^{\bP_1,\alpha,(\tau,\tilde \xi')} (H_{\bP_1,\bP_2}u_2^\delta)(s,x')\big)_{(\tau_0,\tau,\tilde \xi')=(t_0,s,x)}\notag\\
&&+\big(D_x\tilde P_{s,\tau_0}^{\bP_1\alpha,(\tau,\xi')} w(\tau_0,x')\big)_{(\tau_0,\tau,\xi')=(t_0, s, x')}-\big(D_x\tilde P_{s,\tau_0}^{\bP_1,\alpha,(\tau,\tilde \xi')} w(\tau_0, x')\big)_{(\tau_0,\tau,\tilde \xi')=(t_0, s, x)}\notag\\
&&+\int_s^t dv\, \int_{\R^{d}}dy\Bigg(\I_{\left\{v\le \tau_0\right\}}D_x\tilde p^{\bP_1,\alpha,(\tau,\xi')}(s, v, x', y) R_{\bP_1}^{\alpha,(\tau, \xi')}(v,y)\notag\\
&&\quad +\I_{\left\{v>\tau_0\right\}} D_x\tilde p^{\bP_1,\alpha,(\tau,\tilde\xi')}(s,v,x',y) R_{\bP_1}^{\alpha,(\tau,\tilde \xi')}(v,y)   \Bigg)_{(\tau_0,\tau,\xi',\tilde \xi')=(t_0, s,x',x)},
\label{INTEGRATED_DIFF_BXI_GRAD}
\end{eqnarray}
where
\begin{equation}\label{DEF_T_O}
t_0=s+c_0|x-x'|^\alpha.
\end{equation}

The point is now precisely to use the two expansions of the gradient in \eqref{PREAL_CIRC} (that we first differentiate w.r.t. $x$ choosing then $(\tau,\xi)=(s,x) $) and \eqref{INTEGRATED_DIFF_BXI_GRAD}, keeping in mind that the additional term $$\big(D_x\tilde P_{s,\tau_0}^{\bP_1\alpha,(\tau,\xi')} w(\tau_0,x')\big)_{(\tau_0,\tau,\xi')=(t_0,s,x')}-\big(D_x\tilde P_{s,\tau_0}^{\bP_1,\alpha,(\tau,\tilde \xi')} w(\tau_0, x')\big)_{(\tau_0,\tau,\tilde \xi')=(t_0,s,x)}, $$ arising from the change of freezing points precisely needs to be analyzed.

We get:
\begin{eqnarray}
&&D_x w(s,x)-D_xw(s,x')\notag\\
&=&%-\textcolor{red}{\Bigg[}
\Big(\int_s^t dv \I_{\left\{v\le \tau_0\right\}}\int_{\R^d }D_x\tilde  p^{\bP_1,\alpha,(\tau, \xi)}(s,v,x,y)\big(H_{\bP_1,\bP_2}u_2^\delta(v,y)+R_{\bP_1}^{\alpha,(\tau,\xi)}(v,y)\big)dy\notag\\
&&-\int_s^t dv \I_{\left\{v\le \tau_0\right\}}\int_{\R^d }D_x\tilde  p^{\bP_1,\alpha,(\tau, \xi')}(s,v,x',y)\big(H_{\bP_1,\bP_2}u_2^\delta(v,y)+R_{\bP_1}^{\alpha,(\tau,\xi')}(v,y)\big)dy\Big)\Big|_{(\tau_0,\tau,\xi,\xi')=(t_0,s,x,x')}\notag\\
&&-\Big(\big(D_x\tilde P_{s,\tau_0}^{\bP_1\alpha,(\tau,\xi')} w(\tau_0,x')\big)_{(\tau_0,\tau,\xi')=(t_0,s,x')}-\big(D_x\tilde P_{s,\tau_0}^{\bP_1,\alpha,(\tau,\tilde \xi')} w(\tau_0, x')\big)_{(\tau_0,\tau,\tilde \xi')=(t_0,s,x)}\Big)\notag\\
&&+\int_s^t dv \I_{\left\{v>\tau_0\right\}}\int_{\R^d }\big(D_x\tilde  p^{\bP_1,\alpha,(\tau, \xi)}(s, v, x, y) - D_x\tilde  p^{\bP_1,\alpha,(\tau,\xi)}(s,v,x',y) \big) \notag\\
&&\hspace*{2cm}\times\big(H_{\bP_1,\bP_2}u_2^\delta(v, y)+R_{\bP_1}^{\alpha,(\tau,\xi)}(v, y)\big)dy\Big|_{(\tau_0,\tau,\xi)=(t_0,s,x)}%\textcolor{red}{\Bigg]}
\notag\\
&=:&%\textcolor{red}{-[}
\Delta w_{1}(s,x,x')+\Delta w_{2}(s,x,x')+\Delta w_{3}(s,x,x')%\textcolor{red}{]}
.\notag\\
\label{DECOUP_HOLDER_MOD_DIAG}
 \end{eqnarray} 
 The term $\Delta w_{1}(s,x,x')$, corresponding to the local off-diagonal regime within the global diagonal one, can be analyzed as above. Proceeding as in  \eqref{CTR_H_BSUP_HOLDER}-\eqref{PREAL_CTR_ALPHA_GE_1_GRAD}, and with the bound of \eqref{ALL_READY_FOR_CIRCULAR_SUP} we get:
 \begin{eqnarray*}
&& |\Delta w_{1}(s,x,x')|\\
 &\le& \Big|\int_s^{\tau_0} dv \int_{\R^d }D_x\tilde  p^{\bP_1,\alpha,(\tau, \xi)}(s,v,x,y)\big([H_{\bP_1,\bP_2}u_2^\delta(v,y)-H_{\bP_1,\bP_2}u_2^\delta(v,\theta_{v,\tau}(\xi))]+R_{\bP_1}^{\alpha,(\tau,\xi)}(v,y)\big)dy\Big| \bigg|_{(\tau_0,\tau,\xi)=(t_0,s,x)}\notag\\
&&+\Big|\int_s^{\tau_0} dv \int_{\R^d }D_x\tilde  p^{\bP_1,\alpha,(\tau, \xi')}(s,v,x',y)\big([H_{\bP_1,\bP_2}u_2^\delta(v,y)-H_{\bP_1,\bP_2}u_2^\delta(v,\theta_{v,\tau}(\xi'))]+R_{\bP_1}^{\alpha,(\tau,\xi')}(v,y)\big)dy\Big|\bigg|_{(\tau_0,\tau,\xi')=(t_0,s,x')}\\
&\le &C\bigg(\int_s^{t_0}\frac{dv}{(v-s)^{\frac 1\alpha-(\frac \eta \alpha-\varepsilon)}} \bd_{\eta}(\bP_1(v),\bP_2(v))%\frac{|x-x'|^\theta}{(v-s)^{\frac \theta\alpha}}
(t-v)^{-1+\varepsilon}\\
%&&+\int_{s}^{t_0}\frac{dv}{(v-s)^{\frac 1\alpha}}\big(\|w\|_\infty+(v-s)^{\frac {2\eta}\alpha}\|Dw\|_\infty\big)\frac{|x-x'|^\theta}{(v-s)^{\frac \theta \alpha}}\notag\\
&&+\int_s^{t_0} dv (v-s)^{-\frac 1\alpha+\frac{2\eta}{\alpha}}\big[(t-v)^{-\Xi}\Phi(v)+(t-v)^{-(\Xi+\frac \vartheta\alpha)}\Psi(v)%+\textcolor{red}{\|w\|_\infty\I_{\alpha=1}}
\big]\frac{|x-x'|^\vartheta}{(v-s)^{\frac \vartheta \alpha}}\bigg).
 \end{eqnarray*}
 From \eqref{DEF_T_O}, we then get that for $v\in [s,t_0]$, $(t-v)\ge (1-c_0)(t-s) $ and therefore, recalling from \eqref{DEF_T_O} that $|x-x'|^\vartheta=(c_0^{-1}(t_0-s))^{\frac \vartheta\alpha} $:
\begin{eqnarray*}
 \frac{|\Delta w_{1}(s,x,x')|}{|x-x'|^\vartheta}&\le &C\bigg((t-s)^{-1+\varepsilon }\bd_{\eta,s,t}(\bP_1,\bP_2)(t_0-s)^{1+(\frac \eta\alpha-\varepsilon)-(\frac 1\alpha+\frac \vartheta\alpha)}
 \\
&&
+(t-s)^{-\Xi}\int_s^{t_0} dv (v-s)^{-\frac 1\alpha+\frac{2\eta}{\alpha}}\Phi(v)\frac{1}{(v-s)^{\frac \vartheta \alpha}}%\bigg)
\\
&&+(t-s)^{-(\Xi+\frac \vartheta\alpha)}\int_s^{t_0} dv (v-s)^{-\frac 1\alpha+\frac{2\eta}{\alpha}}\Psi(v)\frac{1}{(v-s)^{\frac \vartheta \alpha}}%\textcolor{red}{+\|w\|_\infty\I_{\alpha=1}\int_s^{t_0} dv(v-s)^{-\frac 1\alpha+\frac{2\eta}{\alpha}-\frac \vartheta\alpha }}
\bigg)\\
&\le &C\bigg((t-s)^{-1+\varepsilon }\bd_{\eta,s,t}(\bP_1,\bP_2)%\textcolor{red}{+\|w\|_\infty\I_{\alpha=1}(t-s)^{1-\frac 1\alpha+\frac{2\eta}{\alpha}-\frac \vartheta \alpha}}
%(t-s)^{1+\frac \eta\alpha-(\frac 1\alpha+\frac \theta\alpha)}
 \\
&&+(t-s)^{-\Xi+1-\frac 1\alpha+\frac{2\eta}{\alpha}-\frac \vartheta \alpha}\sup_{v\in [s,t]}\Phi(v)%\\
%&&
+(t-s)^{-(\Xi+\frac \vartheta\alpha)+1-\frac 1\alpha+\frac{2\eta}{\alpha}-\frac \vartheta \alpha}\sup_{v\in [s,t]}\Psi(v)\bigg),
 \end{eqnarray*}
 recalling for the last inequality that:
 \begin{equation}\label{BDINTEGR_ANA_DIAG}
 1+(\frac \eta\alpha-\varepsilon)-(\frac1\alpha+\frac \vartheta\alpha)=1+(\frac \eta \alpha-\varepsilon)-\big(\frac 1\alpha+1-\varepsilon-(\frac 1\alpha-\frac \eta\alpha)\big)=0.
 \end{equation}
  This eventually gives:
 \begin{eqnarray*}
(t-s)^{\Xi+\frac \vartheta\alpha} \frac{|\Delta w_{1}(s,x,x')|}{|x-x'|^\vartheta}&\le &C\Big((t-s)^{\Xi+\frac \vartheta \alpha- 1+\varepsilon }\bd_{\eta,s,t}(\bP_1,\bP_2)%\textcolor{red}{+\|w\|_\infty\I_{\alpha=1}(t-s)^{\Xi+1-\frac 1\alpha+\frac{2\eta}{\alpha}}}
%\Big[(t-v)^{-1+\varepsilon}+(t-v)^{-1+(\frac 1\alpha-\frac \eta\alpha)} \Big]
 \notag\\
&&+(t-s)^{1-\frac 1\alpha+\frac{2\eta}{\alpha}}\sup_{v\in [s,t]}\Phi(v)+(t-s)^{1-(\frac 1\alpha+\frac \vartheta \alpha)+\frac{2\eta}{\alpha}}\sup_{v\in [s,t]}\Psi(v)\Big). \end{eqnarray*}

\noindent Recalling from \eqref{DEF_PHI} that $\Xi:=\frac 1\alpha-\frac \eta \alpha+\frac \varepsilon 2$, from \eqref{BDINTEGR_ANA_DIAG} we get:
\begin{equation}\label{DEF_BAR_ZETA}
\Xi+\frac \vartheta \alpha- 1+\varepsilon=(\frac 1\alpha+\frac \vartheta \alpha)-\frac \eta\alpha-1 +\frac 32 \varepsilon=1-\varepsilon+\frac \eta\alpha-\frac \eta\alpha-1 +\frac 32 \varepsilon=\frac \varepsilon 2=:\bar \zeta.
\end{equation} 
Hence
 \begin{eqnarray}
(t-s)^{\Xi+\frac \vartheta\alpha} \frac{|\Delta w_{1}(s,x,x')|}{|x-x'|^\vartheta}&\le &C\Big((t-s)^{\bar \zeta}\bd_{\eta,s,t}(\bP_1,\bP_2)%\Big[(t-v)^{-1+\varepsilon}+(t-v)^{-1+(\frac 1\alpha-\frac \eta\alpha)} \Big]
% \textcolor{red}{+\|w\|_\infty\I_{\alpha=1}(t-s)^{\Xi+1-\frac 1\alpha+\frac{2\eta}{\alpha}}}
 \notag\\
&&+(t-s)^{1-\frac 1\alpha+\frac{2\eta}{\alpha}}\sup_{v\in [s,t]}\Phi(v)+(t-s)^{1-(\frac 1\alpha+\frac \vartheta \alpha)+\frac{2\eta}{\alpha}}\sup_{v\in [s,t]}\Psi(v)\Big).\label{CTR_HOLDER_DIAG_HORS_DIAG_1}
 \end{eqnarray}

Turning now to $\Delta w_{3} $ in \eqref{DECOUP_HOLDER_MOD_DIAG}, expanding the frozen densities, exploiting as well \eqref{THE_CTR_DER_DENS_PROXY_DIAG_PERTURB} which gives that a diagonal perturbation of the density does not affect the related estimates, we write:
\begin{eqnarray*}
 |\Delta w_{3}(s,x,x')|&\le& \Big|\int_{\tau_0}^t dv\int_{0}^1 d\lambda  \int_{\R^d }D^2_x \tilde  p^{\bP_1,\alpha,(\tau, \xi)}(s, v, x+\lambda (x'-x),y)(x'-x)\\
 &&\times \big([H_{\bP_1,\bP_2}u_2^\delta(v, y)-H_{\bP_1,\bP_2}u_2^\delta(v,\theta_{v,\tau}(\xi))]+R_{\bP_1}^{\alpha,(\tau,\xi)}(v, y)\big)dy\Big| \bigg|_{(\tau_0,\tau,\xi)=(t_0, s, x)}\notag\\
&\le &C\bigg(|x-x'| \bd_{\eta, s,t }(\bP_1,\bP_2)\int_{t_0}^t\frac{dv}{(v-s)^{\frac 2\alpha-(\frac \eta\alpha-\varepsilon)}}(t-v)^{-1+\varepsilon}\\
&&+|x-x'|^{\vartheta}\int_{t_0}^t dv (v-s)^{-\frac 2\alpha+\frac{2\eta}{\alpha}}\Big((t-v)^{-\Xi}\Phi(v)+ (t-v)^{-(\Xi+\frac \vartheta\alpha)}\Psi(v)+%\textcolor{red}{\|w\|_\infty\I_{\alpha=1}}
\Big)|x-x'|^{1-\vartheta}\bigg),
\end{eqnarray*}
where we also used \eqref{CTR_DIST_OP_HOLDER}, for $\gamma=\eta,\lambda=\frac 12$, and \eqref{DOM_DIST} for the last inequality. Splitting the first time integral in the right-hand side of the above inequality into the two disjoint intervals $[t_0, \frac{t+s}{2}]$ and $(\frac{t+s}{2}, t]$, noting that in the current diagonal regime, for $c_0 \in (0,\frac12)$, $t_0= s+ c_0 |x-x'|^\alpha \leq (t+s)/2$, we get (recall that we are \textit{far} from the singularity):
$$
\int_{t_0}^{\frac{t+s}{2}} dv (v-s)^{-\frac 2\alpha+(\frac \eta\alpha-\varepsilon)} (t-v)^{-1+\varepsilon} \leq C (t-s)^{-1+\varepsilon}{(t_0-s)^{-\frac 2\alpha+(\frac \eta \alpha-\varepsilon)+1}}
$$  

\noindent and
$$
\int_{\frac{t+s}{2}}^{t} dv (v-s)^{-\frac 2\alpha+(\frac \eta\alpha-\varepsilon)} (t-v)^{-1+\varepsilon} \leq C(t-s)^{\varepsilon}{(t-s)^{-\frac 2\alpha+(\frac \eta \alpha-\varepsilon)}}= C (t-s)^{-\frac 2\alpha+\frac \eta \alpha}.
$$
Recalling that, for $v\in [t_0,t], |x-x'|\le \big((v-s)/c_0\big)^{\frac 1\alpha} $, we derive:
\begin{eqnarray*}
 |\Delta w_{3}(s,x,x')|&\le& 
C\bigg(\bd_{\eta,s,t}(\bP_1,\bP_2)|x-x'|[(t-s)^{-1+\varepsilon}{(t_0-s)^{-\frac 2\alpha+(\frac \eta \alpha-\varepsilon)+1}} + (t-s)^{-\frac 2\alpha+\frac \eta \alpha}] \\
&&+|x-x'|^\vartheta c_0^{\frac{\vartheta-1}{\alpha}}\int_{t_0}^t dv (v-s)^{-\frac 1\alpha-\frac{\vartheta}{\alpha}+\frac{2\eta}{\alpha}}\big[(t-v)^{-\Xi}\Phi(v)+(t-v)^{-(\Xi+\frac \vartheta\alpha)}\Psi(v)%+\textcolor{red}{\|w\|_\infty\I_{\alpha=1}}
\big]\bigg). 
 \end{eqnarray*}

\noindent Recall, from \eqref{BDINTEGR_ANA_DIAG}, $- \frac 2\alpha+(\frac \eta \alpha-\varepsilon)+1=\frac \vartheta\alpha-\frac 1\alpha$. This yields $(t_0-s)^{-\frac 2\alpha+(\frac \eta \alpha-\varepsilon)+1}=(c_0^{\frac 1\alpha} |x-x'|)^{-1+\vartheta} $. Also in the current diagonal regime $|x-x'|^{1-\vartheta} (t-s)^{-\frac{2}{\alpha}+\frac{\eta}{\alpha}}\leq (t-s)^{\frac{1-\vartheta}{\alpha}-\frac{2}{\alpha}+\frac{\eta}{\alpha}} = (t-s)^{-\frac{\vartheta}{\alpha}- \frac{1}{\alpha}+\frac{\eta}{\alpha}}$. Eventually, since from \eqref{BDINTEGR_ANA_DIAG} and \eqref{DEF_BAR_ZETA} the other time singularities are integrable, we get: 
\begin{eqnarray}
(t-s)^{\Xi+\frac \vartheta \alpha}\frac{ |\Delta w_{3}(s,x,x')|}{|x-x'|^\vartheta}&\le & C\bigg(\bd_{\eta,s,t}(\bP_1,\bP_2) [{c_0^{\frac{\vartheta-1}{\alpha}}}(t-s)^{\Xi+\frac \vartheta \alpha-1+\varepsilon}+(t-s)^{\Xi-\frac{1}{\alpha}+\frac{\eta}{\alpha}}]\notag\\
&&+{c_0^{\frac{\vartheta-1}{\alpha}}}\Big[%\textcolor{red}{(t-s)^{\Xi+1-\frac 1\alpha+\frac {2\eta}\alpha}\|w\|_\infty\I_{\alpha=1}}
%\notag\\
%&&+(t-s)^{1-\frac 1\alpha}\big(\|w\|_\infty+(t-s)^{\frac {2\eta}\alpha}\|Dw\|_\infty\big)\notag\\
%&&
(t-s)^{1-\frac 1\alpha+\frac {2\eta}\alpha}\sup_{v\in [s,t]} \Phi(v)+(t-s)^{1-(\frac1\alpha+\frac\vartheta\alpha)+\frac {2\eta}\alpha}\sup_{v\in [s,t]} \Psi(v)\Big]\bigg).\label{CTR_HOLDER_DIAG_HORS_DIAG_3}
\end{eqnarray}

The term $\Delta w_{2}(s,x,x')$ in \eqref{DECOUP_HOLDER_MOD_DIAG}, corresponding to the change of freezing point remains to be analyzed.
A key point for this contribution is to note from \eqref{CORRESP_SHIFTED}-\eqref{FROZEN_FLOW} {(affine structure in $x'$ of $m_{\tau_0,s}^{(\tau,\xi')}(x') $)} that:
\begin{eqnarray}
\label{CORRESP_DER}
D_x \tilde p^{\bP_1,\alpha,(\tau,\xi')}(s,\tau_0,x',y)&=&D_x\Big( p_{\Theta_{s,\tau_0,(\tau,\xi'),\bP_1}}\big(\tau_0-s,y-m_{\tau_0,s}^{(\tau,\xi')}(x')\big)\Big)\notag\\
&=&D_x\Big( p_{\Theta_{s,\tau_0,(\tau,\xi'),\bP_1}}\big(\tau_0-s, y-x'+\int_s^{\tau_0} b(r,\theta_{r,\tau}(\xi'),\bP_1(r)) dr\big)\Big)\notag
\\
&=&-D_y\Big( p_{\Theta_{s,\tau_0,(\tau,\xi'),\bP_1}}\big(\tau_0-s, y-x'+\int_s^{\tau_0}b(r,\theta_{r,\tau}(\xi'),\bP_1(r)) dr\big)\Big)\notag\\
&=&-D_y \tilde p^{\bP_1,\alpha,(\tau,\xi')}(s,\tau_0,x',y).
\end{eqnarray}
 
Hence, with the notation of \eqref{FROZEN_FLOW}, we get:
\begin{eqnarray*}
&&\Delta w_{2}(s,x,x')\\
&=&\Big(\int_{\R^d} dy \,  \tilde p^{\bP_1,\alpha,(\tau,\xi')}(s,\tau_0,x',y)D_y w (\tau_0,y)-\int_{\R^d} dy \,   \tilde p^{\bP_1,\alpha,(\tau,\tilde \xi')}(s,\tau_0,x',y)D_y w (\tau_0,y)\Big)_{(\tau,\tau_0,\xi',\tilde \xi')=(s,t_0,x',x)}\\
&=&\Big(\int_{\R^d} dy\,   \tilde p^{\bP_1,\alpha,(\tau,\xi')}(s,\tau_0,x',y)\Big[D_y w (\tau_0,y)-D_yw(\tau_0,\tilde m_{s,\tau_0}^{\tau,\xi'}(x'))\Big]\Big)_{(\tau,\tau_0,\xi')=(s,t_0,x')}\\
&&+\Big(D_yw(\tau_0,\tilde m_{s,\tau_0}^{\tau,\xi'}(x'))-D_yw(\tau_0,\tilde m_{s,\tau_0}^{\tau,\tilde \xi'}(x'))\Big)_{(\tau,\tau_0,\xi',\tilde \xi')=(s,t_0,x',x)}\\
&&-\Big(\int_{\R^d} dy  \, \tilde p^{\bP_1,\alpha,(\tau,\tilde \xi')}(s,\tau_0,x',y)\Big[D_yw (\tau_0,y)-D_yw(\tau_0,\tilde m_{s,\tau_0}^{\tau,\tilde \xi'}(x'))\Big]\Big)_{(\tau,\tau_0,\tilde \xi')=(s,t_0,x)},\\
\end{eqnarray*}

\noindent recalling that  $\int_{\R^d} dy \, \tilde p^{\bP_1,\alpha,(\tau,\xi')}(s,\tau_0,x',y)=1$ and $\int_{\R^d} dy\tilde p^{\bP_1,\alpha,(\tau,\tilde \xi')}(s,\tau_0,x',y)=1 $ for the last equality.

From \eqref{CORRESP_SHIFTED}, {Lemma \ref{DER_DENS_PROXY} and the definitions in \eqref{DEF_PSI}}, we again derive:
\begin{eqnarray*}
|\Delta w_{2}(s,x,x')|&\le& C\sum_{z\in \{x,x'\}}\int_{\R^d} dy\,  \bar q(t_0-s,y-\tilde m_{s,t_0}^{s,z}(x'))|y-\tilde m_{s,t_0}^{s,z}(x')|^{\vartheta}\Psi(t_0)(t-t_0)^{-(\Xi+\frac\vartheta\alpha)}\\
&&+|\tilde m_{s,t_0}^{s,x'}(x')-\tilde m_{s,t_0}^{s,x}(x')|^{\vartheta}\Psi(t_0)(t-t_0)^{-(\Xi+\frac\vartheta\alpha)}\\
&\le& \Psi(t_0)(t-t_0)^{-(\Xi+\frac\vartheta\alpha)}\Big( C|t_0-s|^{\frac{\vartheta}{\alpha}}+\Big|\int_s^{t_0}(b(v,\theta_{v,s}(x'),\bP_1(v))-b(v,\theta_{v,s}(x),\bP_1(v)))dv\Big|^{\vartheta}\Big).
\end{eqnarray*}
Now, from Lemma 11 in \cite{chau:meno:prio:19} we have $|\theta_{v,s}(x')-\theta_{v,s}(x)|\le  C(|x-x'|+(v-s)^{\frac 1\alpha})$ and therefore:
\begin{eqnarray*}
\Big|\int_s^{t_0}(b(v,\theta_{v,s}(x'),\bP_1(v))-b(v,\theta_{v,s}(x),\bP_1(v)))dv\Big|&\le&
C\int_s^{t_0}|\theta_{v,s}(x')-\theta_{v,s}(x)|^{2\eta} dv\\
&\le& C(t_0-s)(|x-x'|+(t_0-s)^{\frac 1\alpha})^{2\eta}\\
&\le&C c_0|x-x'|^{\alpha+2\eta}\le C c_0|x-x'|,
\end{eqnarray*}
recalling that, $\alpha+2\eta>1$ and since we are in the diagonal regime and that the time $t-s\le 1 $ we indeed have $|x-x'|\le 1 $.
This finally gives:
\begin{eqnarray*}
|\Delta w_{2}(s,x,x')|\le C\Psi(t_0)(t-t_0)^{-(\Xi+\frac\vartheta\alpha)}c_0^{\frac \vartheta\alpha}|x-x'|^\vartheta.
\end{eqnarray*}
Recalling that $t-t_0=t-s- (t_0-s)\ge t-s-c_0|x-x'|^\alpha\ge (1-c_0)(t-s) $, this finally yields:
\begin{equation}
\label{CTR_HOLDER_DIAG_HORS_DIAG_2}
(t-s)^{\Xi+\frac{\vartheta}{\alpha}}\frac{|\Delta w_{2}(s,x,x')|}{|x-x'|^\vartheta}\le Cc_0^{\frac\vartheta\alpha}(1-c_0)^{-(\Xi+\frac{\vartheta}{\alpha})}\Psi(t_0).
\end{equation}

Plugging \eqref{CTR_HOLDER_DIAG_HORS_DIAG_1}, \eqref{CTR_HOLDER_DIAG_HORS_DIAG_3} and \eqref{CTR_HOLDER_DIAG_HORS_DIAG_2} into \eqref{DECOUP_HOLDER_MOD_DIAG}, we get that in the diagonal case $|x-x'|\le (t-s)^{\frac 1\alpha} $ and with the definition of $\bar \zeta=\frac \varepsilon2 $ in \eqref{DEF_BAR_ZETA}:
\begin{eqnarray}
(t-s)^{\Xi+\frac{\vartheta}{\alpha}}\frac{|D_x w(s,x)-D_x w(s, x')|}{|x-x'|^\vartheta}&\le& C\Big({(1+c_0^{\frac{\vartheta-1}\alpha})\Big[}\bd_{\eta,s,t}(\bP_1,\bP_2) (t-s)^{\bar \zeta}+(t-s)^{1-\frac 1\alpha+\frac {2\eta}\alpha}\sup_{v\in [s,t]} \Phi(v)\notag\\
&&+(t-s)^{1+\frac {2\eta}\alpha-(\frac1\alpha+\frac\vartheta\alpha)}\sup_{v\in [s,t]} \Psi(v)+%\textcolor{red}{(t-s)^{\Xi+1-\frac 1\alpha+\frac {2\eta}\alpha}\|w\|_\infty\I_{\alpha=1}}
{\Big]}%\notag\\
%&&
+c_0^{\frac\vartheta\alpha}(1-c_0)^{-(\Xi+\frac{\vartheta}{\alpha})}\Psi(t_0)\Big)\label{BD_HOLDER_DIAG_FINAL_ALPHA_GE_1}.
\end{eqnarray}
Putting together \eqref{BD_HOLDER_DIAG_FINAL_ALPHA_GE_1} and \eqref{PREAL_CTR_ALPHA_GE_1_HOLD_HD}, {corresponding respectively to the control of the normalized H\"older modulus of the gradient in the off-diagonal and diagonal regimes}, we eventually derive:
\begin{eqnarray}
\Psi(s)&\le& C\Big(\bd_{\eta,s,t}(\bP_1,\bP_2) {(1+c_0^{\frac{\vartheta-1}\alpha})}(t-s)^{\bar \zeta}+{(1+c_0^{\frac{\vartheta-1}\alpha})}(t-s)^{1-\frac 1\alpha+\frac {2\eta}\alpha}\sup_{v\in [s,t]} \Phi(v)%\notag\\
%&&
\notag\\
&&+\sup_{r\in [s,t]}\Psi(r)\big({(1+c_0^{\frac{\vartheta-1}\alpha})}(t-s)^{1-\frac {1+\vartheta}\alpha+\frac{2\eta}\alpha}+c_0^{\frac\vartheta\alpha}(1-c_0)^{-(\Xi+\frac{\vartheta}{\alpha})}\big)%\notag\\
%&&+\textcolor{red}{(t-s)^{\Xi+1-\frac 1\alpha+\frac {2\eta}\alpha}\|w\|_\infty\I_{\alpha=1}}
\Big)\label{TO_PSI}.
\end{eqnarray}
Equation \eqref{TO_PSI} could be established similarly for $s$ replaced by any $v\in [s,t] $ in the above l.h.s. This in turn yields, taking ${T\ge} t-s$ and $c_0$ small enough {s.t. $ (1+c_0^{\frac{\vartheta-1}{\alpha}})T^{\bar \zeta}$ is also \textit{small}}, up to a modification of $C$: 
\begin{equation}
\sup_{v\in [s,t]} \Psi(v)\le C{(1+c_0^{\frac{\vartheta-1}\alpha})}\Big(\bd_{\eta,s,t}(\bP_1,\bP_2) (t-s)^{\bar \zeta}+(t-s)^{1-\frac 1\alpha+\frac {2\eta}\alpha}\sup_{v\in [s,t]} \Phi(v)%+\textcolor{red}{(t-s)^{\Xi+1-\frac 1\alpha+\frac {2\eta}\alpha}\|w\|_\infty\I_{\alpha=1}}
\Big).\label{TO_PSI_2}
\end{equation}
The point is now to plug \eqref{TO_PSI_2} into \eqref{PREAL_CTR_ALPHA_GE_1_GRAD} to complete the circular argument. {Recalling as well that $1-(\frac 1\alpha+\frac\vartheta\alpha)+\frac{2\eta}\alpha=\frac \eta\alpha-\varepsilon$, this yields}:
%Rewrite now \textcolor{red}{from \eqref{PREAL_CTR_ALPHA_GE_1_GRAD}}:
\begin{eqnarray*}
|D_xw(s,x)|
&\underset{{\eqref{PREAL_CTR_ALPHA_GE_1_GRAD}}}{\le} &C\bigg(\bd_{\eta,s,t}(\bP_1,\bP_2) (t-s)^{-\frac 1\alpha+\frac \eta \alpha }%+(t-s)^{1-(\frac1\alpha-\frac {2\eta}\alpha)}\|Dw\|_\infty\notag\\
+\sup_{v\in [s,t]}\Phi(v)(t-s)^{1-\frac 1\alpha+\frac{2\eta}\alpha-\Xi}+\sup_{v\in [s,t]}\Psi(v)(t-s)^{1-\frac 1\alpha+\frac{2\eta}\alpha-(\Xi+\frac \vartheta \alpha)}%\\
%&&%+\textcolor{red}{\|w\|_\infty (t-s)^{1-\frac 1\alpha+\frac{2\eta}\alpha} \I_{\alpha=1}}
\bigg)%\notag
\\
&\underset{{\eqref{TO_PSI_2}}}{\le} & C\bigg(\bd_{\eta,s,t}(\bP_1,\bP_2) {\Big[}(t-s)^{-\frac 1\alpha+\frac \eta \alpha }{+(1+c_0^{\frac{\vartheta-1}\alpha})(t-s)^{\bar \zeta+\frac \eta\alpha-\varepsilon-\Xi}\Big]}\\%+(t-s)^{1-(\frac1\alpha-\frac {2\eta}\alpha)}\|Dw\|_\infty\notag\\
&&+\sup_{v\in [s,t]}\Phi(v){\Big[}(t-s)^{1-\frac 1\alpha+\frac{2\eta}\alpha-\Xi}+{(1+c_0^{\frac{\vartheta-1}\alpha})(t-s)^{(1-\frac 1\alpha+\frac{2\eta}\alpha)+\frac \eta\alpha-\varepsilon-\Xi}\Big]}%\\
%+\sup_{v\in [s,t]}\Psi(v)(t-s)^{1-\frac 1\alpha+\frac{2\eta}\alpha-(\Xi+\frac \vartheta \alpha)}\\
%&&+\textcolor{red}{\|w\|_\infty \I_{\alpha=1} \Big[(t-s)^{1-\frac 1\alpha+\frac{2\eta}\alpha}+(1+c_0^{\frac{\vartheta-1}\alpha})(t-s)^{(1-\frac 1\alpha+\frac{2\eta}\alpha)+\frac \eta\alpha-\varepsilon}\Big] }
\bigg).
\label{PREAL_CTR_ALPHA_GE_1_GRAD_2}
\end{eqnarray*}
Again the previous equation would also hold for any $v\in [s,t] $ instead of $s$ in the previous l.h.s.
Thus,  from the definition on $\Xi $ in \eqref{DEF_PHI} and {$\bar \zeta $ in \eqref{DEF_BAR_ZETA}}, 
\begin{eqnarray*}
\sup_{v\in [s,t]}\Phi(v)&:=&\sup_{v\in [s,t]}(v-s)^{\Xi}|D_xw(s,x)|\le C{(1+c_0^{\frac{\vartheta-1}\alpha})}\bigg(\bd_{\eta,s,t}(\bP_1,\bP_2) (t-s)^{\bar \zeta}\\
&&+\sup_{v\in [s,t]}\Phi(v)(t-s)^{1-\frac 1\alpha+\frac{2\eta}\alpha}%+\textcolor{red}{\|w\|_\infty\I_{\alpha=1}(t-s)^{\Xi+1-\frac 1\alpha+\frac{2\eta}\alpha}}
\bigg).
\end{eqnarray*}
{Taking again} ${T\ge} t-s$ and $c_0$ small enough {and s.t. $ (1+c_0^{\frac{\vartheta-1}{\alpha}})T^{\bar \zeta}$ is sufficiently \textit{small}}, we derive:
\begin{eqnarray}
\sup_{v\in [s,t]}\Phi(v) \le C{(1+c_0^{\frac{\vartheta-1}\alpha})}
%\bigg(
\bd_{\eta,s,t}(\bP_1,\bP_2) (t-s)^{\bar \zeta}%\textcolor{red}{\|w\|_\infty\I_{\alpha=1}(t-s)^{\Xi+1-\frac 1\alpha+\frac{2\eta}\alpha}}
%\bigg)
\label{TO_PHI}.
\end{eqnarray}
Plugging first \eqref{TO_PHI} into \eqref{TO_PSI_2}, we get:
\begin{equation}
\sup_{v\in [s,t]} \Psi(v)\le C{(1+c_0^{\frac{\vartheta-1}\alpha})^2}
%\Big(
\bd_{\eta,s,t}(\bP_1,\bP_2) (t-s)^{\bar \zeta}
%+\textcolor{red}{\|w\|_\infty\I_{\alpha=1}(t-s)^{\Xi+1-\frac 1\alpha+\frac{2\eta}\alpha}}\Big)
.\label{TO_PSI_3}
\end{equation}
%Plugging now \eqref{TO_PSI_3} into \eqref{TO_PHI}, we obtain:
%\begin{eqnarray}
%\sup_{v\in [s,t]}\Phi(v) \le C\bigg(\bd_{\eta,s,t}(\bP_1,\bP_2) (t-s)^{\bar \zeta}\bigg)\label{TO_PHI_2}.
%\end{eqnarray}

It now remains to plug \eqref{TO_PHI} and \eqref{TO_PSI_3} into \eqref{ALL_READY_FOR_CIRCULAR}-\eqref{ALL_READY_FOR_CIRCULAR_SUP}
\begin{eqnarray}
|w(s,x)|&\le& \|w\|_\infty\le C\bigg(\bd_{\eta,s,t}(\bP_1,\bP_2)(t-s)^{\varepsilon}\notag\\
&&+\int_s^t dv (v-s)^{\frac{2\eta}{\alpha}}\Big((t-v)^{-\Xi} \Phi(v)+(t-v)^{-(\Xi+\frac \vartheta\alpha)}\Psi(v)%+\textcolor{red}{\|w\|_\infty \I_{\alpha=1}}
\Big)\bigg)\notag\\
%&\le& C\textcolor{red}{(1+c_0^{\frac{\vartheta-1}\alpha})^2}\bigg(\bd_{\eta,s,t}(\bP_1,\bP_2) (t-s)^{\textcolor{red}{\varepsilon}}+\textcolor{red}{\|w\|_\infty\I_{\alpha=1}(t-s)^{\Xi+1-\frac 1\alpha+\frac{2\eta}\alpha}}\bigg)\notag\\
&\le& C{(1+c_0^{\frac{\vartheta-1}\alpha})^2}\bd_{\eta,s,t}(\bP_1,\bP_2) (t-s)^{\varepsilon}.%\notag\\
\label{FINAL_SMOOTHING_BOUND}
\end{eqnarray}
Observe indeed carefully that, even though the time contribution in front of the distance in \eqref{TO_PHI} and \eqref{TO_PSI} appears at a coarser rate, namely $(t-s)^{\bar \zeta}=(t-s)^{\frac \varepsilon 2} $, since we take out the supremum of $\Phi ,\Psi$  and integrate once again in time in the above first inequality, we eventually get a final contribution in $(t-s)^\varepsilon $ for the distance (from the previous definitions of $\Xi $, $\vartheta $ in \eqref{DEF_PHI} and \eqref{BD_CTR_U_I_DELTA_MOD_HOLDER_GRADIENT} respectively). This also illustrates the intuitive fact that the remainder terms in the perturbative expansion, those associated with $\Phi,\Psi $, yield somehow negligible contributions.

Equation \eqref{FINAL_SMOOTHING_BOUND} precisely gives the expected bound \eqref{THE_PREAL_CTR_ON_W} (with ${\tilde \zeta=\varepsilon } $) and completes the proof of Proposition \ref{PROP_STAB} in the case $\alpha\ge 1 $.

%Therefore, taking $(\tau,\xi)=(s,x) $, and since we are int he diagonal regime, we get from \eqref{THE_CTR_DER_DENS_PROXY_DIAG_PERTURB} and the previous controls \eqref{CTR_H_BSUP}, \eqref{PREAL_CIRCULAR_2} (with the additional induced singularity for the latter one):
%\begin{eqnarray}
%|Dw(s,x)-Dw(s,x')|\le C|x-x'|
%\end{eqnarray}

\subsubsection*{The case $\alpha< 1 $}
Restarting from \eqref{ALL_READY_FOR_CIRCULAR}, we write:
\begin{eqnarray}
|w(s,x)|&\le& C\bigg(\bd_{\eta,s,t}(\bP_1,\bP_2)(t-s)^{ 1-(\frac{1}{\alpha}-\frac{\eta}{\alpha})}+(t-s)\|w\|_\infty%\notag\\
%&&
+\int_s^t dv \, (v-s)^{\frac{2\eta}{\alpha}} \|Dw(v,\cdot)\|_\infty\bigg)
\end{eqnarray}
\noindent so that
\begin{eqnarray}
\|w\|_\infty&\le& C\bigg(\bd_{\eta,s,t}(\bP_1,\bP_2)(t-s)^{ 1-(\frac{1}{\alpha}-\frac{\eta}{\alpha})}%\notag\\
%&&
+\int_s^t dv \, (v-s)^{\frac{2\eta}{\alpha}} \|Dw(v,\cdot)\|_\infty\bigg),
\label{PREAL_CIR_SUPER_CRITICAL}
\end{eqnarray}

\noindent assuming that $T\ge t-s$ small enough.
Observe first that, for the first term of the above r.h.s., we have assumed that $(\frac{1}{\alpha}-\frac{\eta}{\alpha})<1\iff \alpha+\eta>1  $. This indeed allows to derive the smoothing effect for the contribution $H_{\bP_1,\bP_2}u_2^\delta $.

Let us now proceed  from the cancellation argument as in \eqref{DIFF_WITH_CANCEL}. We get from equations \eqref{CTR_DIST_OP} and \eqref{CTR_DIST_OP_HOLDER} with $\lambda=1/2$ (recalling that $\gamma=\eta$) {and \eqref{DOM_DIST}} that in the current case:
 \begin{eqnarray}\label{CTR_H_BSUP_HOLDER_alpha_LE_1_WITH_ALL}
|H_{\bP_1,\bP_2}u_2^\delta(v,y)-H_{\bP_1,\bP_2}u_2^\delta(v,\theta_{v,\tau}(\xi))|&\le& C\bd_\eta(\bP_1(v),\bP_2(v)) \Big[ |y-\theta_{v,\tau}(\xi)|^\eta (t-v)^{-\big(\frac 1\alpha-\frac \eta \alpha\big)}\notag\\
&&+|y-\theta_{v,\tau}(\xi)|^\vartheta(t-v)^{-1+\varepsilon} \Big]\notag\\
&\le & {C\bd_\eta(\bP_1(v),\bP_2(v))|y-\theta_{v,\tau}(\xi)|^{\alpha-1+\eta-\alpha\varepsilon}(t-v)^{-1+\varepsilon},}
\end{eqnarray}
{as soon as $\varepsilon \le 1-(\frac 1\alpha-\frac \eta \alpha) $},
recalling from \eqref{BD_CTR_U_I_DELTA_MOD_HOLDER_GRADIENT} that $\vartheta=\alpha\Big[1-\varepsilon-\big(\frac 1\alpha-\frac \eta\alpha\big)\Big]$ and $\alpha<1 $ for the last inequality.
Now, similarly to \eqref{PREAL_CTR_ALPHA_GE_1_GRAD}, using  \eqref{CTR_H_BSUP_HOLDER_alpha_LE_1_WITH_ALL} and again Lemma \ref{DER_DENS_PROXY}, we also get:
\begin{eqnarray}
|D_x w(s, x)|&\le& C\Big(\int_s^t dv \, (v-s)^{-\frac 1\alpha+1-\frac1\alpha+\frac \eta \alpha-\varepsilon}\bd_\eta(\bP_1(v),\bP_2(v)) (t-v)^{-1+\varepsilon}\notag\\
&&
+\int_s^t dv \, (v-s)^{-\frac 1\alpha+\frac{2\eta}{\alpha}}\big(\|w\|_\infty+\|Dw(v,\cdot)\|_\infty\big)\Big)\notag\\
&\underset{\eqref{PREAL_CIR_SUPER_CRITICAL}}{\le} &C\Big(\bd_{\eta,s,t}(\bP_1,\bP_2) (t-s)^{1-(\frac 2\alpha-\frac \eta \alpha)}
+\int_s^t dv \, (v-s)^{-\frac 1\alpha+\frac{2\eta}{\alpha}}\|Dw(v,\cdot)\|_\infty\Big),\label{CTR_GRAD_ALPHA_SUPER_CRITICAL}
\end{eqnarray}
provided that 
\begin{equation}\label{COND_SUPER_CRITIQUE}2-\frac 2\alpha+\frac \eta \alpha>0\iff  \alpha >1-\frac \eta 2
\end{equation}
in order to obtain integrable time singularities in the first integral (taking as well $\varepsilon $ small enough). Recalling that $\eta \in (0,1/2] $, and that we have also assumed that $\alpha>2\eta $ for integrability purposes, this means that $\alpha>2\eta \vee (1-\frac \eta 2) $.
% Hence, the minimal attainable stability index alpha is strictly greater than the value of $ 2\eta$ giving the equality in the previous r.h.s., i.e. $2\eta=1-\frac \eta 2\iff \eta=\frac 25 $, that is $\frac 45 $}.%\footnote{De S. a V. et N.: il y a ici quelque chose d'amusant. Avec la methode de Noufel, ou a la Kolokoltsov, le seuil final en sur-critique etait $\alpha >\frac{\sqrt{17}-1}{4} $ qui est somme toute assez proche de $3/4$. La diffÃÂrence est aussi que maintenant on peut gerer les coefficients non bornÃÂs.}.

With the notation \eqref{DEF_PHI} for $\big(\Phi(v)\big)_{v\in [s,t]} $,  taking in our current supercritical case  $\Xi=\frac 1\alpha-\frac \eta{2 \alpha} <1$ (from \eqref{COND_SUPER_CRITIQUE}), we derive from 
\eqref{PREAL_CIR_SUPER_CRITICAL} and \eqref{CTR_GRAD_ALPHA_SUPER_CRITICAL}:
\begin{eqnarray}
|w(s,x)|&\le& \|w\|_\infty\le C\bigg(\bd_{\eta,s,t}(\bP_1,\bP_2)(t-s)^{ 1-(\frac{1}{\alpha}-\frac{\eta}{\alpha})}
+\sup_{v\in [s,t]}\Phi(v)\int_s^t dv \, (v-s)^{\frac{2\eta}{\alpha}}(t-v)^{-\Xi} \bigg)\notag\\
&\le& C\bigg(\bd_{\eta,s,t}(\bP_1,\bP_2)(t-s)^{ 1-(\frac{1}{\alpha}-\frac{\eta}{\alpha})}
+\sup_{v\in [s,t]}\Phi(v) (t-s)^{1+2\frac \eta\alpha-\Xi} \bigg),%\notag\\
\label{PREAL_CIR_SUPER_CRITICAL_2}
\end{eqnarray}
and 
\begin{eqnarray*}
|D_x w(s,x)|
&\le &C\Big(\bd_{\eta,s,t}(\bP_1,\bP_2) (t-s)^{1-(\frac 2\alpha-\frac \eta \alpha)}%+\|w\|_\infty(t-s)^{1-(\frac 1\alpha-\frac{2\eta}\alpha)}
+\sup_{v\in [s,t]}\Phi(v)\int_s^t dv \, (v-s)^{-\frac 1\alpha+\frac{2\eta}{\alpha}}(t-v)^{-\Xi}\Big)\\
&\le &C\Big(\bd_{\eta,s,t}(\bP_1,\bP_2) (t-s)^{1-(\frac 2\alpha-\frac \eta \alpha)}%+\|w\|_\infty(t-s)^{1-(\frac 1\alpha-\frac{2\eta}\alpha)}
+\sup_{v\in [s,t]}\Phi(v)(t-s)^{1-(\frac 1\alpha-\frac{2\eta}\alpha)-\Xi}\Big), 
\end{eqnarray*}
which in turn gives:
\begin{eqnarray*}
(t-s)^{\Xi}|D_x w(s,x)|\le C\Big(\bd_{\eta,s,t}(\bP_1,\bP_2) (t-s)^{1-(\frac 2\alpha-\frac \eta \alpha)+\Xi}%+\|w\|_\infty(t-s)^{1-(\frac 1\alpha-\frac{2\eta}\alpha)+\Xi}
+\sup_{v\in [s,t]}\Phi(v)(t-s)^{1-(\frac 1\alpha-\frac{2\eta}\alpha)}\Big).
\end{eqnarray*}
Observe now that for the previous choice of $\Xi $ and from \eqref{COND_SUPER_CRITIQUE}, $1-(\frac 2\alpha-\frac \eta \alpha)+\Xi=1-(\frac 1\alpha-\frac \eta{2 \alpha})>0 $, so that we indeed have a regularizing effect. Furthermore, the above equation would still be valid with $s $ replaced by any $\tilde v\in [s,t]$ for the l.h.s. This yields that for all $\tilde v\in [s,t] $
\begin{eqnarray}
\Phi(\tilde v)\le C\Big(\bd_{\eta,s,t}(\bP_1,\bP_2) (t-s)^{1-(\frac 1\alpha-\frac \eta {2\alpha})}%+\|w\|_\infty(t-s)^{1-(\frac 1\alpha-\frac{2\eta}\alpha)+\Xi}
+\sup_{v\in [s,t]}\Phi(v)(t-s)^{1-(\frac 1\alpha-\frac{2\eta}\alpha)}\Big) \notag
\end{eqnarray}

\noindent which  yields
\begin{eqnarray}
\sup_{v\in [s,t]}\Phi( v)\le C%\Big(
\bd_{\eta,s,t}(\bP_1,\bP_2) (t-s)^{1-(\frac 1\alpha-\frac \eta {2\alpha})}%+\|w\|_\infty(t-s)^{1-(\frac 1\alpha-\frac{2\eta}\alpha)+\Xi}\Big).
\label{CTR_PHI_SUPER_CRITICAL}
\end{eqnarray}
Plugging \eqref{CTR_PHI_SUPER_CRITICAL} into \eqref{PREAL_CIR_SUPER_CRITICAL_2} eventually yields:
%\begin{eqnarray}
%|w(s,x)|
%&\le& C\Big(\bd_{\eta,s,t}(\bP_1,\bP_2)(t-s)^{ 1-(\frac{1}{\alpha}-\frac{\eta}{2\alpha})}\big(1+(t-s)^{1+2\frac \eta\alpha-\Xi}\big)+\|w\|_\infty (t-s)\Big)
%  \bigg).%\notag\\
%\label{PREAL_CIR_SUPER_CRITICAL_3}
%\end{eqnarray}
%Since again \eqref{PREAL_CIR_SUPER_CRITICAL_3} remains valid for the l.s.h. replaced by any $w(v,x), \ v\in [s,t]$, we finally obtain, recalling that $t-s $ is small and up to a modification of the constant $C$:
\begin{equation*}
\|w\|_\infty
\le C\bd_{\eta,s,t}(\bP_1,\bP_2)(t-s)^{ 1-(\frac{1}{\alpha}-\frac{\eta}{2\alpha})}.
\end{equation*}
This concludes the proof of Proposition \ref{PROP_STAB} in the super-critical case under the condition $\alpha > 2\eta \vee (1 -\frac{\eta}{2})$. {This naturally gives the constraint on the stability index $\alpha\in (4/5,1)$}. As already mentioned, in Appendix \ref{sec:modification:proof}, we remove this assumption and establish our main result, namely, Theorem \ref{THM_WP_NL} under the sole condition $\alpha > 2\eta \vee (1-\eta)$.

\appendix

\section{Additional controls on the density in the stable case}
\label{APP_STABLE_CORRESPONDENCE}

To derive the first main results of Lemma \ref{DER_DENS_PROXY}, the key idea consists in separating the small and the large jumps in the Fourier transform of the proxy $\Theta_{s,v,(\tau,\xi),\bmu} $ introduced in \eqref{DEF_THETHA_RV_FROZEN}. This \textit{trick}, which is often known as the It\^o-L\'evy decomposition, has been used in many contexts (see e.g. \cite{wata:07}, \cite{szto:10} 
 or \cite{huan:meno:16} for density estimates; \cite{huan:meno:prio:19} for $L^p$ bounds related to degenerate stably driven Ornstein-Uhlenbeck processes).

Fix now a final parameter $t$. Rewrite from \eqref{DEF_THETHA_RV_FROZEN},  for any $\zeta \in \R^d $:
\begin{eqnarray}
\varphi_{\Theta_{s,t,(\tau,\xi),\bmu}} (\zeta)&:=&\E\Big[\exp(i\langle \zeta,\Theta_{s,t,(\tau,\xi),\bmu}\rangle)\Big]=\exp\bigg(\int_s^t dv \int_{0}^{+\infty} \frac{dr}{r^{1+\alpha}}\int_{{\mathbb S}^{d-1}}\omega(d\lambda) \Big( \cos(\langle \zeta,\sigma\big(v,\theta_{v,\tau}(\xi),\bmu(v)\big)\lambda r\rangle ) -1\Big)\bigg) \notag\\
&=&\exp\bigg(\int_s^t dv \int_{0}^{(t-s)^{\frac 1\alpha}} \frac{dr}{r^{1+\alpha}}\int_{{\mathbb S}^{d-1}}\omega(d\lambda) \Big( \cos(\langle \zeta,\sigma\big(v,\theta_{v,\tau}(\xi),\bmu(v)\big)\lambda r\rangle ) -1\Big)\bigg)\notag\\
&&\times \exp\bigg(\int_s^t dv \int_{(t-s)^{\frac 1\alpha}}^{+\infty} \frac{dr}{r^{1+\alpha}}\int_{{\mathbb S}^{d-1}}\omega(d\lambda) \Big( \cos(\langle \zeta,\sigma\big(v,\theta_{v,\tau}(\xi),\bmu(v)\big)\lambda r\rangle ) -1\Big)\bigg)\notag\\
&=:&%\varphi_{M_{s,t,(\tau,\xi),\bmu}} (\zeta)\varphi_{N_{s,t,(\tau,\xi),\bmu}} (\zeta).
\widehat p_{M_{s,t,(\tau,\xi),\bmu}} (\zeta)\widehat P_{N_{s,t,(\tau,\xi),\bmu}} (\zeta).
\label{ITO_LEVY_CUT}
\end{eqnarray}
It precisely allows to write:
  $\Theta_{s,t,(\tau,\xi),\bmu}:=M_{s,t,(\tau,\xi),\bmu}+N_{s,t,(\tau,\xi),\bmu}$
where $M_{s,t,(\tau,\xi),\bmu}$ and $N_{s,t,(\tau,\xi),\bmu} $ are independent random variables    corresponding respectively to the small and large jumps part of the stochastic integral $\Theta_{s,t,(\tau,\xi),\bmu} $. 
The density of $\Theta_{s,t,(\tau,\xi),\bmu}$  writes for all $y\in \R^d $ as
\begin{equation}
\label{DECOMP_G_P_DENS}
p_{\Theta_{s,t,(\tau,\xi),\bmu}}(y)=\int_{\R^d} p_{M_{s,t,(\tau,\xi),\bmu}}(y-{z})P_{N_{s,t,(\tau,\xi),\bmu}}(d{z}).
\end{equation}
It is known, see e.g. Lemma B.1 in \cite{huan:meno:16}, Lemma A.2 in \cite{huan:meno:prio:19}, that the density $p_{M_{s,t,(\tau,\xi),\bmu}} $ associated with the small jumps of the stochastic integral is smooth and has a polynomial decay of arbitrary order (i.e. it belongs to the Schwartz class).%\footnote{\textcolor{red}{Attention tout cela etait avec le mapping pour construire une mesure stable mais cela ne passe pas super bien a la sensibilite.}}. 
On the other hand, $P_{N_{s,t,(\tau,\xi),\bmu}}$ is a Poisson measure s.t. for all $\beta\in (0,\alpha) $, there exists $C_{\beta,\alpha}\ge 1 $ s.t. $\E[|N_{s,t,(\tau,\xi),\bmu}|^{\beta}] =\int_{\R^d} |{z}|^\beta P_{N_{s,t,(\tau,\xi),\bmu}}(d{z})\le C_{\alpha,\beta} (t-s)^{\frac\beta \alpha} $.
Importantly, both components $p_{M_{s,t,(\tau,\xi),\bmu}}$ and $P_{N_{s,t,(\tau,\xi),\bmu}}$ (when the jump measure $\nu $ is absolutely continuous) can be dominated uniformly w.r.t. the freezing parameters $(\tau,\xi),\bmu $. For the sake of clarity, and since these steps are also crucial to study further the sensitivities of the quantities $p_{M_{s,t,(\tau,\xi),\bmu}}$, $P_{N_{s,t,(\tau,\xi),\bmu}}$ w.r.t. the measure argument, we prove these facts in the next subsection.

\subsection{First controls for the two parts of the frozen density}
We first give here some useful lemmas concerning the behavior of the laws of the independent random variables $ M_{s,t,(\tau,\xi),\bmu}$ and $N_{s,t,(\tau,\xi),\bmu} $ s.t.
$ \Theta_{s,t,(\tau,\xi),\bmu}=M_{s,t,(\tau,\xi),\bmu}+N_{s,t,(\tau,\xi),\bmu}$.
\begin{lem}[Density estimate on the Martingale part and associated derivatives.]\label{EST_DENS_MART}
For all $m\ge 1$, there exists $C_m\ge 1$ s.t. for all $0\le s<t\le T, y\in\R^{d}$, $(\tau,\xi)\in [0,T]\times \R^d $,
\begin{equation}\label{THE_BD_PM}
p_{M_{s,t,(\tau,\xi),\bmu}}(y) \le \frac{C_m}{ (t-s)^{\frac d\alpha}} \left( 1+ \frac{|y|}{(t-s)^{\frac 1\alpha}}\right)^{-m}=:\tilde C_m p_{\bar M_{s,t},m}(y),
\end{equation}
where $\int_{\R^d} p_{\bar M_{s,t},m}(y) dy=1$.

Also, for all $m\ge 1 $ and all multi-index $\beta ,\ |\beta|\le 3$,
$$|D_y^\beta p_{M_{s,t,(\tau,\xi),\bmu}}(y)|\le \frac{C_m}{ (t-s)^{\frac {d+|\beta|}\alpha}} \left( 1+ \frac{|y|}{(t-s)^{\frac {1}\alpha}}\right)^{-m}=\frac{\tilde C_m}{(t-s)^{\frac{|\beta|}\alpha}} p_{\bar M_{s,t},m}(y).$$
\end{lem}

\begin{proof}Inverting  the Fourier transform in \eqref{ITO_LEVY_CUT} write:
\begin{eqnarray*}
p_{M_{s,t,(\tau,\xi),\bmu}}(y)= \frac{1}{(2\pi)^{d}} \int_{\R^{d}}d\zeta e^{-i\langle \zeta,y \rangle}\exp\bigg( \int_s^t dv \int_{0}^{(t-s)^{\frac 1\alpha}} \frac{dr}{r^{1+\alpha}}\int_{{\mathbb S}^{d-1}}\omega(d\lambda) \Big( \cos(\langle \zeta,\sigma\big(v,\theta_{v,\tau}(\xi),\bmu(v)\big)\lambda r\rangle ) -1\Big) \bigg).
%\exp \left( - (T-t)\int_{\R^{nd}} \{1- \cos( \langle p, \zeta \rangle)\}   \ind_{\{|\zeta|\le (T-t)^{1/\alpha}\}} \nu_{S}(d\zeta)\right).
\end{eqnarray*}
%Recall that in polar coordinates, $\nu_{\S}^\kappa(d\eta) = \ind_{\{ |\eta| \le \kappa \}} \frac{d |\eta|}{|\eta|^{1+\alpha}} \mu_{\S}(d\bar \eta)$. Taking $t^{1/\alpha}\le \kappa$ and c
Setting $(t-s)^{1/\alpha}\zeta ={\tilde \zeta}$ yields:
\begin{eqnarray} \label{EXP_DENS_M}
p_{M_{s,t,(\tau,\xi),\bmu}}(y)= \frac{1}{(2\pi)^{d}} (t-s)^{-\frac d\alpha}\int_{\R^{d}}d{\tilde \zeta} e^{-i\langle {\tilde \zeta},\frac{y}{(t-s)^{\frac 1\alpha}} \rangle}\notag\\
\times \exp\bigg( \int_s^t dv \int_{0}^{(t-s)^{\frac 1\alpha}} \frac{dr}{r^{1+\alpha}}\int_{{\mathbb S}^{d-1}}\omega(d\lambda) \Big( \cos(\langle \frac {\tilde \zeta}{(t-s)^{\frac1\alpha}},\sigma\big(v,\theta_{v,\tau}(\xi),\bmu(v)\big)\lambda r\rangle ) -1\Big) \bigg).
\end{eqnarray}
%%%% Je ne comprends pas l'intÃÂÃÂrÃÂÃÂt de ce commentaire. A voir.
%Notice that we changed the cut-off by symmetry of $\mu_\S$.
Let us now denote 
$$\hat{f}_{s,t}({\tilde \zeta}) :=\exp\bigg( \int_s^t dv \int_{0}^{(t-s)^{\frac 1\alpha}} \frac{dr}{r^{1+\alpha}}\int_{{\mathbb S}^{d-1}}\omega(d\lambda) \Big( \cos(\langle \frac {\tilde \zeta}{(t-s)^{\frac1\alpha}},\sigma\big(v,\theta_{v,\tau}(\xi),\bmu(v)\big)\lambda r\rangle ) -1\Big) \bigg).$$
Since the L\'evy measure in the above expression has finite support, we get from  Theorem 3.7.13 in Jacob \cite{Jacob1} that $\hat f_{s,t}$ is infinitely differentiable as a function of ${\tilde \zeta}$.
%\footnote{Decider ensemble quelque chose pour les notations des gradients}.
Moreover,  
\begin{eqnarray*}
|D \hat f_{s,t}({\tilde \zeta})|&\le& \int_s^t dv \int_{0}^{(t-s)^{\frac 1\alpha}} \frac{dr}{r^{1+\alpha}}\int_{{\mathbb S}^{d-1}}\omega(d\lambda) \frac{|\sigma\big(v,\theta_{v,\tau}(\xi),\bmu(v)\big)\lambda| r}{(t-s)^{\frac 1\alpha}}\Big| \sin(\langle \frac {\tilde \zeta}{(t-s)^{\frac1\alpha}},\sigma\big(v,\theta_{v,\tau}(\xi),\bmu(v)\big)\lambda r\rangle ) \Big| \\
&&\times \exp\bigg( \int_s^t dv \int_{0}^{(t-s)^{\frac 1\alpha}} \frac{dr}{r^{1+\alpha}}\int_{{\mathbb S}^{d-1}}\omega(d\lambda) \Big( \cos(\langle \frac {\tilde \zeta}{(t-s)^{\frac1\alpha}},\sigma\big(v,\theta_{v,\tau}(\xi),\bmu(v)\big)\lambda r\rangle ) -1\Big) \bigg).
\end{eqnarray*}
Write now:
\begin{eqnarray*}
&&\int_s^t dv \int_{0}^{(t-s)^{\frac 1\alpha}} \frac{dr}{r^{1+\alpha}}\int_{{\mathbb S}^{d-1}}\omega(d\lambda) \frac{|\sigma\big(v,\theta_{v,\tau}(\xi),\bmu(v)\big)\lambda| r}{(t-s)^{\frac 1\alpha}}\Big| \sin(\langle \frac {\tilde \zeta}{(t-s)^{\frac1\alpha}},\sigma\big(v,\theta_{v,\tau}(\xi),\bmu(v)\big)\lambda r\rangle ) \Big|\\ 
&\le& C(t-s)\int_{r\le (t-s)^{\frac 1\alpha}}  \frac{dr}{r^{1+\alpha}}\frac{r}{(t-s)^{\frac 1 \alpha}}   (\I_{\alpha<1}+ \I_{\alpha\ge 1}|{\tilde \zeta}| \frac{r}{(t-s)^{\frac 1\alpha}})\le C(1+|{\tilde \zeta}|).
\end{eqnarray*}
Thus:
\begin{eqnarray*}
|D \hat f_{s,t}({\tilde \zeta})|
&\le & C(1+|{\tilde \zeta}|) \exp\left(  \int_s^t dv\int_{\R^{d}} \{\cos( \langle \sigma\big(v,\theta_{v,\tau}(\xi),\bmu(v)\big) q , \frac{{\tilde \zeta}}{(t-s)^{\frac 1\alpha}} \rangle)  -1\}\nu(dq)\right) \\ 
&&\times \exp(2(t-s)\nu(B(0,(t-s)^{\frac 1\alpha})^c))%\\
%&
\le 
%& 
C (1+|{\tilde \zeta}|)\exp(-C^{-1}|{\tilde \zeta}|^\alpha),\ C\ge 1,
\end{eqnarray*}
using for the last inequality  that from \eqref{NON_DEG_STABLE} 
\begin{eqnarray*}
&&\exp\left(  \int_s^t dv\int_{\R^{d}} \{\cos( \langle \sigma\big(v,\theta_{v,\tau}(\xi),\bmu(v)\big) q , \frac{{\tilde \zeta}}{(t-s)^{\frac 1\alpha}} \rangle)  -1\}\nu(dq)\right)\\
&=&\exp\left(-C_{\alpha,d}  \int_s^t dv\int_{{\mathbb S}^{d-1}} | \langle \sigma\big(v,\theta_{v,\tau}(\xi),\bmu(v)\big) \lambda , \frac{{\tilde \zeta}}{(t-s)^{\frac 1\alpha}} \rangle|^\alpha \omega(d\lambda)\right)\\
&\le& C\exp(-C^{-1}|{\tilde \zeta}|^\alpha),
\end{eqnarray*}
and also that, from the decomposition of $\nu $, $\nu(B(0,(t-s)^{1/\alpha})^c) \le C/(t-s)$.

Similarly, for all multi-index $\beta $ of length $|\beta|=l\in \N $: 
\begin{eqnarray*}
|D^\beta \hat f_{s,t}({\tilde \zeta})|&\le& 
C_l (1+|{\tilde \zeta}|^l)\exp(-C^{-1}|{\tilde \zeta}|^\alpha),\ C_l\ge 1.
\end{eqnarray*}
Thus, $\hat f_{s,t}$ belongs the Schwartz space. 
Denoting by $f_{s,t}$ its inverse Fourier transform, we have:
$$
\forall m \ge 0, \ \forall y \in \R^{d}, \exists C_m\ge 1\ s.t.: |f_{s,t}(y)| \le C_m (1+|y|)^{-m}.
$$
Now, since $p_M(t-s,y) = (t-s)^{-\frac d\alpha} f_{s,t}(y/(t-s)^{\frac 1\alpha})$, the announced bound follows. The control concerning the derivatives is derived similarly.

\end{proof}

\begin{lem}[Controls on the Poisson measure]\label{EST_POISSON}

Let $\nu$ be any symmetric stable jump measure satisfying \eqref{NON_DEG_STABLE}. Then, for all $\beta\in [0,\alpha) $, there exists  $C_{\alpha,\beta}\ge 1$ s.t. for all $0\le s<t, z\in \R^{d}$, 
\begin{equation}\label{GLOB_INT_POISSON}
\int_{\R^d} |y|^{\beta}P_{N_{s,t,(\tau,\xi),\bmu}}(dy) \le C_{\alpha,\beta} (t-s)^{\frac \beta\alpha}.
\end{equation}

Assume now the jump measure $\nu  $ is absolutely continuous, i.e. Assumption \A{AC} holds. Then, there exists a Poisson measure $P_{\bar N_{\rho_{s,t}}}:= \exp(-1)\sum_{n\ge 0}\frac{\rho_{s,t}^{\star n}}{n!} $ where $\rho_{s,t}(dy)=\kappa (t-s) \frac{dr }{r^{1+\alpha}}\I_{|r|\ge c(t-s)^{\frac 1\alpha}}\Lambda_{{\mathbb S}^{d-1}}(d\theta),\ y=r\theta, \ (r,\theta)\in \R_*^{+}\times {\mathbb S}^{d-1} $, with $c>0$ depending on the non-degeneracy constants in \eqref{NON_DEG_STABLE}, \eqref{COND_UE_S}, is a probability   measure and $\rho_{s,t}^{\star n} $ denotes its $n^{{\rm th}}$ fold convolution,  %and a normalizing constant $\kappa>0 $, 
for which it holds that:
\begin{equation}
P_{N_{s,t,(\tau,\xi),\bmu}}\le CP_{\bar N_{\rho_{s,t}}}.\label{DEF_POISSON_QUI_DOMINE}
\end{equation}

\end{lem}

\begin{proof}
Introduce $\bar \nu_{s,t}(dz)=\nu(dz)\I_{|z|\ge (t-s)^{\frac 1\alpha}} $ which is a finite measure and s.t. 
$ \bar \nu_{s,t}(\R^d)\le C(t-s)^{-1}$. With this notation at hand, write:
\begin{eqnarray*}
\widehat P_{N_{s,t,(\tau,\xi),\bmu}}(\zeta)&=&\exp\bigg(\int_s^t dv \int_{(t-s)^{\frac 1\alpha}}^{+\infty} \frac{dr}{r^{1+\alpha}}\int_{{\mathbb S}^{d-1}}\omega(d\lambda) \Big( \cos(\langle \zeta,\sigma\big(v,\theta_{v,\tau}(\xi),\bmu(v)\big)\lambda r\rangle ) -1\Big)\bigg)=\\
&=&\exp( \int_{s}^t dv \widehat {\bar \nu}_{s,t}(\sigma\big(v,\theta_{v,\tau}(\xi),\bmu(v)\big)^*\zeta)- (t-s) \bar \nu_{s,t}(\R^d))\\
&=&\exp( \int_{s}^t dv \widehat {\bar \nu}_{s,t,(v,(\tau,\xi),\bmu)}(\zeta)- (t-s) \bar \nu_{s,t}(\R^d)),
\end{eqnarray*}
defining $\bar \nu_{s,t,(v,(\tau,\xi),\bmu)}(A):=\bar \nu_{s,t}(\{y\in \R^d:\sigma\big(v,\theta_{v,\tau}(\xi),\bmu(v)\big)y \in A \} )$ and where $\widehat {\bar \nu}_{s,t}$, $\widehat {\bar \nu}_{s,t,(v,(\tau,\xi),\bmu)} $ denote the Fourier-Stieltjes transform of the considered measure. Hence, introducing the measure  $\zeta_{s,t,(\tau,\xi),\bmu}:=\int_{s}^t dv  {\bar \nu}_{s,t,(v,(\tau,\xi),\bmu)} $ and expanding the previous exponential and by termwise Fourier inversion, we get:
\begin{equation}
P_{N_{s,t,(\tau,\xi),\bmu}}=\exp(\zeta_{s,t,(\tau,\xi),\bmu} %\int_{s}^t dv  {\bar \nu}_{s,t,(v,(\tau,\xi),\bmu)}
-(t-s) \bar \nu_{s,t}(\R^d))=\exp(-(t-s)\bar \nu_{s,t}(\R^d))\sum_{n\ge 0}\frac{\big( \zeta_{s,t,(\tau,\xi),\bmu}%\int_{s}^t dv  {\bar \nu}_{s,t,(v,(\tau,\xi),\bmu)}
\big)^{\star n}}{n!},\label{EXPR_LAW_POISSON}
\end{equation}
where in the above equation we again used that, for a finite measure $\rho  $ on $\R^d $, $(\rho)^{\star n}:= \underset{n\ {\rm times}}{\rho\star\cdots \star \rho}$ denotes its $n^{{\rm th}}$ fold convolution. %Introduce now the measure  $\zeta_{s,t,(\tau,\xi),\bmu}:=\int_{s}^t dv  {\bar \nu}_{s,t,(v,(\tau,\xi),\bmu)} $ and 
Observe now from the definition of $\bar \nu_{s,t} $ and the non-degeneracy conditions on $\sigma $ (see eq. \eqref{COND_UE_S} in \A{UE}) that there exists $C\ge 1$ s.t. for all $0\le s<t,z\in \R^d,\ \bmu\in {\mathcal C}(\R^+,{\mathcal P}(\R^d)) $,  $C^{-1}\le \zeta_{s,t,(\tau,\xi),\bmu}(\R^d)=(t-s)\bar \nu_{s,t}(\R^d)\le C $. Introducing the normalized (probability) measure $\tilde \zeta_{s,t,(\tau,\xi),\bmu}:=\frac{\zeta_{s,t,(\tau,\xi),\bmu}}{\zeta_{s,t,(\tau,\xi),\bmu}(\R^d)}=\frac{\zeta_{s,t,(\tau,\xi),\bmu}}{(t-s)\bar \nu_{s,t}(\R^d)} $, one then rewrites:
\begin{eqnarray*}
\int_{\R^d} |y|^{ \beta} P_{N_{s,t,(\tau,\xi),\bmu}}(dy)\le \exp(-(t-s)\bar \nu_{s,t}(\R^d)) \sum_{n=0}^{\infty}\frac{\big((t-s)\bar \nu_{s,t}(\R^d)\big)^n}{n!} \int_{\R^d}|y|^{ \beta}(\tilde \zeta_{s,t,(\tau,\xi),\bmu})^{\star n}(dy).
\end{eqnarray*}
Introducing now on a some probability space $(\Omega,\F,\P) $ a sequence $(X_i) _{i\ge 0}$ of i.i.d. random variables with law $\tilde \zeta_{s,t,(\tau,\xi),\bmu}$, we can rewrite from the above control:
\begin{eqnarray}
\int_{\R^d} |y|^{ \beta} P_{N_{s,t,(\tau,\xi),\bmu}}(dy)\le e^{-C^{-1}} \sum_{n=0}^{\infty}\frac{C^n}{n!} \E[|\sum_{i=0}^n X_i|^\beta]\le e^{-C^{-1}} \sum_{n=0}^{\infty}\frac{C^n}{n!} (n+1)^{1\vee \beta}\E[| X_1|^\beta].\label{PREAL_BOUND_POISSON}
\end{eqnarray}
Now,
\begin{eqnarray}
\E[|X_1|^\beta]&=&\int_{\R^d} |y|^ \beta \tilde \zeta_{s,t,(\tau,\xi),\bmu}(dy)\le C\int_s^t dv \int_{\R^d} |y|^\beta \bar \nu_{s,t,(v,(\tau,\xi),\bmu)}(dy)\notag\\
&=&C\int_{s}^t dv \int_{|y|\ge (t-s)^{\frac 1\alpha}}|\sigma(v,\theta_{v,\tau}(\xi),\bmu(v))y|^\beta \nu(dy)\notag\\
&\le &C\int_{s}^t dv \int_{r\ge (t-s)^{\frac 1\alpha}}\frac{dr}{r^{1+\alpha}} \int_{{\mathbb S}^{d-1}}\omega(d\lambda) |\sigma(v,\theta_{v,\tau}(\xi),\bmu(v))\lambda r|^\beta \le C\int_{s}^t dv \int_{r\ge (t-s)^{\frac 1\alpha}}\frac{dr}{r^{1+\alpha-\beta}}\notag\\
&\le& C_{\alpha,\beta}(t-s)^{\frac \beta \alpha},\label{CTR_MOMENT_X1}
\end{eqnarray}
recalling that $\alpha>\beta $ for the last inequality. Observe as well that the last constant $C_{\alpha,\beta} $ depends on $\alpha $ and $\beta $ and explodes when $\beta \uparrow \alpha $.  
Plugging the above bound into \eqref{PREAL_BOUND_POISSON} we derive the first statement \eqref{GLOB_INT_POISSON} without any specific assumption on the jump measure.

Assume now that $\nu $ is  absolutely continuous. Then, still with the previous notations,  we can bound the measure $\int_{s}^t dv  {\bar \nu}_{s,t,(v,(\tau,\xi),\bmu)} $ uniformly in  $(v,(\tau,\xi),\bmu)$. Namely, for all $A\in {\mathcal B}(\R^d) $ and with the notations of assumption \A{AC}:
\begin{eqnarray*}
\Big(\int_s^t dv  {\bar \nu}_{s,t,(v,(\tau,\xi),\bmu)}\Big)(A)&=&\int_{s}^t dv \big(\bar \nu_{s,t,(v,(\tau,\xi),\bmu)}(A)\big)=\int_{s}^t dv \big( \bar \nu_{s,t}(\{y\in \R^d: \sigma(v,\theta_{v,\tau}(\xi),\bmu(v))y\in  A\})\big)\\
&=&\int_{s}^t dv \int_{|y|\ge (t-s)^{\frac 1\alpha}}\I_{ \sigma(v,\theta_{v,\tau}(\xi),\bmu(v))y\in A} \nu(dy)\\
&=&\int_{s}^t dv \int_{|\sigma(v,\theta_{v,\tau}(\xi),\bmu(v))^{-1}\tilde y|\ge (t-s)^{\frac 1\alpha}}\!\!\!\I_{\tilde y \in A} \frac{f(\sigma(v,\theta_{v,\tau}(\xi),\bmu(v))^{-1}\tilde y)}{|\sigma(v,\theta_{v,\tau}(\xi),\bmu(v))^{-1}\tilde y|^{d+\alpha}}\frac{d\tilde y}{{\rm det}(\sigma(v,\theta_{v,\tau}(\xi),\bmu(v))) } \\
&\le & C \int_{s}^t dv \int_{|\tilde y|\ge c(t-s)^{\frac 1\alpha}} \I_{\tilde y\in A}  g\Big(\frac{\sigma(v,\theta_{v,\tau}(\xi),\bmu(v))^{-1}\tilde y}{|\sigma(v,\theta_{v,\tau}(\xi),\bmu(v))^{-1}\tilde y|}\Big) \frac{d\tilde y}{|\tilde y|^{d+\alpha}}.
\end{eqnarray*}
Since $g$ is bounded on ${\mathbb S}^{d-1} $ (recall $g$ is Lipschitz continuous on ${\mathbb S}^{d-1} $), writing $\tilde y=r\theta,\ (r,\theta)\in \R_*^{+}\times {\mathbb S}^{d-1}$, we derive that there exists a constant $C$ and a probability measure $ \rho_{s,t}(dy)=\kappa (t-s) \frac{dr }{r^{1+\alpha}}\I_{|r|\ge c(t-s)^{\frac 1\alpha}}\Lambda_{{\mathbb S}^{d-1}}(d\theta) $, s.t. for all $(\tau,\xi,\bmu)\in [0,T]\times \R^d\times {\mathcal C}(\R^+,{\mathcal P}(\R^d))$, $A\in {\mathcal B}(\R^d) $:
\begin{equation*}
\Big(\int_s^t dv  {\bar \nu}_{s,t,(v,(\tau,\xi),\bmu)}\Big)(A)\le C (t-s)\int_{|r|\ge c(t-s)^{\frac 1\alpha}}\frac{dr}{r^{1+\alpha}}\int_{{\mathbb S}^{d-1}}\Lambda_{{\mathbb S}^{d-1}}(d\theta) \I_{r\theta\in A}=C\rho_{s,t}(A),
\end{equation*}
up to a modification of $C$. The above equation and \eqref{EXPR_LAW_POISSON} eventually give \eqref{DEF_POISSON_QUI_DOMINE}.

\end{proof}

\begin{remark}[Integrability of the stochastic integral]
As a direct corollary of Lemma \ref{EST_DENS_MART} and \ref{EST_POISSON} we derive that for any non-degenerate jump measure $\nu $ satisfying \eqref{NON_DEG_STABLE} and  all $\beta \in [0,\alpha) $ there exists $C_\beta \ge 1$ (increasing with $ \beta$) s.t. for all $0\le s<t, z\in \R^d,\bmu \in {\mathcal C}(\R^+,{\mathcal P}(\R^d)) $:
\begin{equation}
\label{INT_INT_STO}
\int_{\R^d}|y|^\beta p_{\Theta_{s,t,(\tau,\xi),\bmu}}(y)dy\le C_\beta (t-s)^{\frac \beta\alpha}.
\end{equation}
\end{remark}

More generally, under \A{AC}, Lemma \ref{DER_DENS_PROXY} readily follows from Lemmas \ref{EST_DENS_MART} and \ref{EST_POISSON} setting for all $z\in \R^d $, $\bar q(t-s,z):=\int_{\R^d}p_{\bar M_{s,t},m}(z-\bar \zeta)P_{\bar N_{\rho_{s,t}}}(d{\tilde \zeta}) $, for any arbitrary integer $m\ge d+1 $.\\

\subsection{Proof of Lemma \ref{CHARACTERIZATION_OF_GRADIENT_EXPLOSION_AND_HOLDER_MODULUS}. Quantitative bounds for the solution $u_1^\delta $ of the PDE \eqref{EDP_U_I_DELTA}}
To derive the announced estimates, we rewrite the solution $u_i^\delta \in {\mathcal C}^{\alpha+\eta}(\R^d,\R)$ of \eqref{EDP_U_I_DELTA} as follows:
\begin{equation}
\label{FROZEN_Asso_PDE_U_I}
\left\{
\begin{array}{l}
\Big(\partial_v   +  \tilde L_v^{\bP_i,\alpha,(\tau,\xi)} \Big) u_i^\delta(v,x) + b(v,\theta_{v,\tau}(\xi),\bP_i(v)) \cdot D_x u_{i}^\delta(v,x) \\
\hspace*{.5cm}= -R_{\bP_i}^{\alpha,(\tau,\xi)}(v,x),\ (v,x)\in [0,t)\times \R^d, \\
u_i^\delta(t,x) =h^\delta(x),\quad \text{on }\R^d,
\end{array} \right. 
\end{equation}
with the notations of \eqref{DEF_OPERATEUR_GELE} and \eqref{THE_TERM_PERTURB_FIXED_MEASURE} (replacing $w $ by $u_i^\delta$ in $R_{\bP_i}^{\alpha,(\tau,\xi)}$), i.e.

\begin{equation*}
%\label{DEF_OPERATEUR_GELE}
\tilde L_v^{\bP_i,\alpha,(\tau,\xi)} \varphi(x):={\rm {p.v.}}\int_{\R^d} \big(\varphi(x+\sigma(v,\theta_{v,\tau}(\xi),\bP_i(v)) \zeta)-\varphi(x) \big) \nu(d\zeta)
\end{equation*}

\noindent and
\begin{equation*}%\label{THE_TERM_PERTURB_FIXED_MEASURE}
R_{\bP_i}^{\alpha,(\tau,\xi)}(v,x):=\big(b(v,x,\bP_i(v))-b(v,\theta_{v,\tau}(\xi),\bP_i(v)) \big) \cdot Du_i^\delta(v, x)+\Big(L_s^{\bP_i,\alpha}-\tilde L_s^{\bP_i,\alpha,(\tau,\xi)}\Big) u_i^\delta(v,x).
\end{equation*}

\noindent The PDE \eqref{FROZEN_Asso_PDE_U_I} suggests to rewrite $u_i^\delta$ through a Duhamel type expansion (which has been fully justified in Proposition 8 of \cite{chau:meno:prio:19}). This yields for all $(s,x)\in [0,t]\times \R^d $,
\begin{eqnarray}
u_i^\delta(s,x)&=&\tilde  P_{s,t}^{\bP_i,\alpha,(\tau,\xi)}h^\delta(x)+\int_s^t dv \big(\tilde  P_{s,v}^{\bP_i,\alpha,(\tau,\xi)}R_{\bP_i}^{\alpha,(\tau,\xi)}(v,\cdot)\big)(x)\notag\\
&=&\int_{\R^d }\tilde  p^{\bP_i,\alpha,(\tau,\xi)}(s, t,x,y) h^\delta(y)dy +\int_s^t dv \int_{\R^d }\tilde  p^{\bP_i,\alpha,(\tau,\xi)}(s,v,x,y) R_{\bP_i}^{\alpha,(\tau,\xi)}(v,y)dy.\label{PREAL_CIRC_U_I_DELTA}
\end{eqnarray}
Now, from \eqref{THE_TERM_PERTURB_FIXED_MEASURE} and \eqref{THE_CORRESP_DER_INT_STO}, similarly to \eqref{THE_CTR_PERTURB_1}-\eqref{PREAL_CIRCULAR_2},  choosing $(\tau,\xi)=(s,x) $ and exploiting \eqref{CORRES_LINEARIZED_FLOW_AND_FLOW} and Lemma \ref{DER_DENS_PROXY}, we derive:
\begin{eqnarray}
\Big|\int_s^t dv \int_{\R^d }\tilde  p^{\bP_i,\alpha,(\tau,\xi)}(s, v, x, y) R_{\bP_i}^{\alpha,(\tau,\xi)}(v, y)dy\Big| \Bigg|_{(\tau,\xi)=(s,x)}\notag\\
\le C\int_s^t dv (v-s)^{\frac{2\eta}{\alpha}}\Big( \|u_i^\delta\|_\infty\I_{\left\{ \alpha\le 1\right\}}+\|Du_i^\delta(v,\cdot)\|_\infty+\|Du_i^\delta(v,\cdot)\|_{\dot {\mathcal C}^{\vartheta}}\I_{\left\{\alpha\ge 1\right\}}\Big).\label{PREAL_CIRCULAR_2_U_1_DELTA}
\end{eqnarray}

Thus, recalling from the definition of $h^\delta=h\star \rho_\delta $ (for a smooth mollifying kernel $\rho_\delta $) that $\|h^\delta\|_\infty\le \|h\|_\infty $, we get: 
\begin{equation}
|u_i^\delta(s, x)|\le C\bigg(\|h\|_\infty+(t-s)^{}\|u_i^\delta\|_\infty\I_{\left\{\alpha\le1\right\}}+\int_s^t dv (v-s)^{\frac{2\eta}{\alpha}}\Big( \|Du_i^\delta(v,\cdot)\|_\infty+\|Du_i^\delta(v,\cdot)\|_{\dot {\mathcal C}^{\vartheta}}\I_{\left\{\alpha\ge 1\right\}}\Big)\bigg).%\notag\\
\label{ALL_READY_FOR_CIRCULAR_U_1_DELTA}
\end{equation}

From \eqref{PREAL_CIRC_U_I_DELTA}, we can differentiate w.r.t $x$ and also use a cancellation technique for the final condition. Namely,
\begin{equation}
D_xu_i^\delta(s,x)
=\int_{\R^d } D_x \tilde p^{\bP_i,\alpha,(\tau,\xi)}(s, t, x, y) [h^\delta(y)-h^\delta(\theta_{t,\tau}(\xi))]dy+\int_s^t dv \int_{\R^d } D_x \tilde p^{\bP_i,\alpha,(\tau,\xi)}(s, v, x, y) R_{\bP_i}^{\alpha,(\tau,\xi)}(v, y)dy.\label{DIFF_WITH_CANCEL_U_1_DELTA}
\end{equation}
Hence, from \eqref{DIFF_WITH_CANCEL_U_1_DELTA}, recalling that $h^\delta$ is $\gamma $-H\"older continuous in space uniformly in the regularization parameter $\delta $, with in particular $\|h^\delta\|_{\dot {\mathcal C}^\gamma}\le \|h\|_{\dot {\mathcal C}^\gamma}  $, similarly to \eqref{ALL_READY_FOR_CIRCULAR_U_1_DELTA}, we get:
\begin{eqnarray}
|D_xu_i^\delta(s,x)|\label{THE_CTR_PERTURB_1_GRAD_U_1_DELTA}\\
\le C\Big(\frac{\|h\|_{\dot {\mathcal C}^\gamma}}{(t-s)^{\frac 1\alpha}}\int_{\R^d } dy |y-\theta_{v,\tau}(\xi)|^{\gamma} \bar q\big(v-s, y-m_{v,s}^{(\tau,\xi)}(x)\big) \notag\\
+\int_s^t \frac{dv}{(v-s)^{\frac 1\alpha}} \int_{\R^d} dy\Big(\|b(v,\cdot,\bP_i(v))\|_{{\mathcal C}^{2\eta}}|D_yu_i^\delta(v, y)|+ \|\sigma(v,\cdot,\bP_i(v))\|_{{\mathcal C}^{2\eta}}\big[ (\|u_i^\delta\|_\infty+\|D_xu_i^\delta(v,\cdot)\|_\infty)\I_{\left\{\alpha<1\right\}}\notag\\
+\|D_xu_i^\delta(v,\cdot)\|_{{\mathcal C}^{\vartheta}}\I_{\left\{\alpha>1\right\}}+\|u_i^\delta(v,\cdot)\|_{{\mathcal C}^{1+\vartheta}}\I_{\left\{\alpha=1\right\}}\big]\Big)|y-\theta_{v,\tau}(\xi)|^{2\eta} \bar q\big(v-s, y-m_{v, s}^{(\tau,\xi)}(x)\big)\Big).\notag
\end{eqnarray}

Taking $(\tau,\xi)=(s,x) $, it follows again from \eqref{CORRES_LINEARIZED_FLOW_AND_FLOW} and Lemma \ref{DER_DENS_PROXY} that:
\begin{eqnarray}
&&|D_x u_i^\delta(s,x)|%\notag
\\
&\le &C\bigg(\|h\|_{\dot {\mathcal C}^\gamma} (t-s)^{-\frac 1\alpha+\frac \gamma \alpha}%+(t-s)^{1-(\frac1\alpha-\frac {2\eta}\alpha)}\|Dw\|_\infty\notag\\
+\int_s^t dv (v-s)^{-\frac 1\alpha+\frac{2\eta}{\alpha}}\Big(\|u_i^\delta(v,\cdot)\|_\infty\I_{\left\{\alpha\le 1\right\}}+\|D_x u_i^\delta(v,\cdot)\|_\infty+\|D_x u_i^\delta(v,\cdot)\|_{\dot {\mathcal C}^{\vartheta}}\I_{\left\{\alpha \ge 1\right\}}\Big)\bigg).\notag\label{PREAL_CTR_ALPHA_GE_1_GRAD_U_1_DELTA}
\end{eqnarray}
We now use the notations of \eqref{DEF_PSI} and \eqref{DEF_PHI} with the exponents modified, namely we consider here 
\begin{equation}\label{DEF_PHI_U_1_DELTA}
\Phi(v):=(t-v)^{\Xi}\sup_{x\in \R^d}|D_xu_i^\delta(v, x)|, \ \Xi:=\frac 1\alpha-\frac \gamma\alpha,\ v\in [s,t],
\end{equation}
 and
 \begin{equation}\label{DEF_PSI_U_1_DELTA}
 \Psi(v):=(t-v)^{\Xi+\frac\vartheta\alpha}\|D_x u_i^\delta(v,\cdot)\|_{{\mathcal C}^\vartheta}, \ v\in [s,t].
 \end{equation}
 \begin{equation}\label{EQ_XI_PLUS_THETA_U_1_DELTA}
 \Xi+\frac \vartheta\alpha=\frac 1\alpha-\frac \gamma\alpha+1-\varepsilon-(\frac 1\alpha-\frac \gamma \alpha)=1- \varepsilon<1.
 \end{equation}
In particular, with these notations at hand, since $t-s$ is \textit{small}, it follows from \eqref{ALL_READY_FOR_CIRCULAR_U_1_DELTA} that
\begin{equation}
\label{CTR_PART_NORME_SUP_U_1_DELTA}
\|u_i\|_\infty\le C\bigg(\|h\|_\infty+\int_s^t dv (v-s)^{\frac{2\eta}{\alpha}}\Big( (t-v)^{-\Xi}\Phi(v)+(t-v)^{-(\Xi+\frac\vartheta\alpha)}\Psi(v)%+\textcolor{red}{\|w\|_\infty\I_{\alpha\le1}}
\Big)\bigg).
\end{equation}

As above,  we  distinguish the \textit{diagonal} and \textit{off-diagonal} regimes w.r.t. the current considered times. In the \textit{off-diagonal} regime holds, we readily get from \eqref{PREAL_CTR_ALPHA_GE_1_GRAD_U_1_DELTA}:
\begin{eqnarray}
(t-s)^{\frac \vartheta\alpha}\frac{|D_x u_i^\delta(s, x)-D_x u_i^\delta(s,x')|}{|x-x'|^\vartheta}&\le& |D_x u_i^\delta(s, x)-D_x u_i^\delta(s,x')|\le |D_x u_i^\delta(s,x)|+|D_x u_i^\delta(s,x')|\notag\\ 
&\le &C\bigg(\|h\|_{\dot {\mathcal C}^\gamma} (t-s)^{-\frac 1\alpha+\frac \gamma \alpha}\notag\\
&&+\int_s^t dv (v-s)^{-\frac 1\alpha+\frac{2\eta}{\alpha}}\Big(\|u_i^\delta\|_\infty\I_{\left\{\alpha \le 1\right\}}+(t-v)^{- \Xi}\Phi(v)+(t-v)^{-(\Xi+\frac \vartheta\alpha)}\Psi(v)\Big)\bigg),\notag\\
(t-s)^{\Xi+\frac \vartheta\alpha}\frac{|D_x u_i^\delta(s,x)- D_x u_i^\delta(s, x')|}{|x-x'|^\vartheta}&\le &C\bigg( \|h\|_{\dot {\mathcal C}^\gamma}+(t-s)^{1-\frac 1\alpha+\frac{2\eta}\alpha}\sup_{v\in [s,t]}\Phi(v)\notag\\
&&+(t-s)^{1-(\frac 1\alpha+\frac \vartheta\alpha)+\frac{2\eta}\alpha}\sup_{v\in [s, t]}\Psi(v)\bigg),\label{PREAL_CTR_ALPHA_GE_1_HOLD_HD_U_1_DELTA}
\end{eqnarray}

\noindent using as well \eqref{CTR_PART_NORME_SUP_U_1_DELTA} for the last inequality.

Let us now turn to the global \textit{diagonal regime}: $|x-x'| \le (t-s)^{\frac 1\alpha}$. We again need to split as above the analysis in function of the considered running time introducing a change of freezing point. Namely, from Proposition 12 in \cite{chau:meno:prio:19}, we derive:
\begin{eqnarray}
&&D_x u_i^\delta(s, x)-D_x u_i^\delta(s,x')\notag\\
&=&\Big(\int_s^t dv \I_{v\le \tau_0}\int_{\R^d }D_x\tilde  p^{\bP_i,\alpha,(\tau, \xi)}(s, v, x, y)R_{\bP_i}^{\alpha,(\tau,\xi)}(v,y)dy\notag\\
&&-\int_s^t dv \I_{v\le \tau_0}\int_{\R^d }D_x\tilde  p^{\bP_i,\alpha,(\tau, \xi')}(s,v,x',y)R_{\bP_i}^{\alpha,(\tau,\xi')}(v,y)dy\Big)\Big|_{(\tau_0,\tau,\xi,\xi')=(t_0,s,x,x')}\notag\\
&&-\Big(\big(D_x\tilde P_{s,\tau_0}^{\bP_i\alpha,(\tau,\xi')} u_i^\delta(\tau_0,x')\big)_{(\tau_0,\tau,\xi')=(t_0, s, x')}-\big(D_x\tilde P_{s,\tau_0}^{\bP_i,\alpha,(\tau,\tilde \xi')} u_i^\delta(\tau_0, x')\big)_{(\tau_0,\tau,\tilde \xi')=(t_0,s,x)} \Big)\notag\\
&&+\Big(\int_{\R^d}dy h^{\delta}(y)\big(D_x\tilde  p^{\bP_i,\alpha,(\tau, \xi)}(s, t,x, y)-D_x\tilde  p^{\bP_i,\alpha,(\tau,\xi)}(s,t,x',y) \big)\notag\\
&&+\int_s^t dv \I_{v>\tau_0}\int_{\R^d }\big(D_x\tilde  p^{\bP_i,\alpha,(\tau, \xi)}(s,v,x,y)-D_x\tilde  p^{\bP_i,\alpha,(\tau,\xi)}(s,v,x',y) \big) R_{\bP_i}^{\alpha,(\tau,\xi)}(v,y)dy\Big)\Big|_{(\tau_0,\tau,\xi)=(t_0,s,x)}\notag\\
&=:&\Delta u_{i 1}^\delta(s,x,x')+\Delta u_{i 2}^\delta(s,x,x')+\Delta u_{i 3}^\delta(s,x,x').\notag\\
\label{DECOUP_HOLDER_MOD_DIAG_U_1_DELTA}
 \end{eqnarray}

 The term $\Delta u_{11}^\delta(s,x,x')$, corresponding to the local off-diagonal regime within the global diagonal one, can be analyzed as above (see the analysis of the contribution $\Delta w_{1}(s,x,x')$ introduced in  \eqref{DECOUP_HOLDER_MOD_DIAG}). This yields:
 \begin{eqnarray*}
&& |\Delta u_{i 1}^\delta(s,x,x')|\le \Big|\int_s^{\tau_0} dv \int_{\R^d }D_x\tilde  p^{\bP_i,\alpha,(\tau, \xi)}(s, v, x, y)R_{\bP_i}^{\alpha,(\tau,\xi)}(v, y)dy\Big| \bigg|_{(\tau_0,\tau,\xi)=(t_0,s,x)}\notag\\
&&
+\Big|\int_s^{\tau_0} dv \int_{\R^d }D_x\tilde  p^{\bP_i,\alpha,(\tau, \xi')}(s,v,x',y)R_{\bP_i}^{\alpha,(\tau,\xi')}(v,y)dy\Big|\bigg|_{(\tau_0,\tau,\xi')=(t_0,s,x')}\\
&\le &C\bigg(\int_s^{t_0} dv (v-s)^{-\frac 1\alpha+\frac{2\eta}{\alpha}}\big[(t-v)^{-\Xi}\Phi(v)+(t-v)^{-(\Xi+\frac \vartheta\alpha)}\Psi(v)\big]\frac{|x-x'|^\vartheta}{(v-s)^{\frac \vartheta \alpha}}\bigg).
 \end{eqnarray*}
 For $v\in [s,t_0]$, $(t-v)\ge (1-c_0)(t-s) $ and therefore:
\begin{eqnarray*}
 \frac{|\Delta u_{i1}^\delta(s,x,x')|}{|x-x'|^\vartheta}
%&\le &C\bigg((t-s)^{-1+\varepsilon }\bd_{\eta,s,t}(\bP_i,\bP_2)(t_0-s)^{1+(\frac \eta\alpha-\varepsilon)-(\frac 1\alpha+\frac \theta\alpha)}
% \\
%&&+(t-s)^{-\Xi}\int_s^{t_0} dv (v-s)^{-\frac 1\alpha+\frac{2\eta}{\alpha}}\Phi(v)\frac{1}{(v-s)^{\frac \theta \alpha}}\bigg)\\
%&&+(t-s)^{-(\Xi+\frac \theta\alpha)}\int_s^{t_0} dv (v-s)^{-\frac 1\alpha+\frac{2\eta}{\alpha}}\Psi(v)\frac{1}{(v-s)^{\frac \theta \alpha}}\bigg)\\
&\le &C\bigg((t-s)^{-\Xi+1-\frac 1\alpha+\frac{2\eta}{\alpha}-\frac \vartheta \alpha}\sup_{v\in [s,t]}\Phi(v)%\\
%&&
+(t-s)^{-(\Xi+\frac \vartheta\alpha)+1-\frac 1\alpha+\frac{2\eta}{\alpha}-\frac \vartheta \alpha}\sup_{v\in [s,t]}\Psi(v)\bigg),
 \end{eqnarray*}
 recalling for the last inequality that $-\frac 1\alpha+\frac{2\eta}{\alpha} -\frac\vartheta\alpha= 2\frac \eta\alpha -(\frac 1\alpha+\frac \vartheta\alpha)= \frac{2\eta}{\alpha} -\big(\frac 1\alpha+1-\varepsilon-(\frac 1\alpha-\frac \gamma\alpha)\big)=-1+\frac{2 \eta}{\alpha} -\frac{\gamma}{\alpha}+\varepsilon>-1 $.
  This eventually gives:
 \begin{eqnarray}
(t-s)^{\Xi+\frac \vartheta\alpha} \frac{|\Delta u_{i1}^\delta(s,x,x')|}{|x-x'|^\vartheta}&\le &C\Big((t-s)^{1-\frac 1\alpha+\frac{2\eta}{\alpha}}\sup_{v\in [s,t]}\Phi(v)+(t-s)^{1-(\frac 1\alpha+\frac \vartheta \alpha)+\frac{2\eta}{\alpha}}\sup_{v\in [s,t]}\Psi(v)\Big). 
\label{CTR_HOLDER_DIAG_HORS_DIAG_1_U_1_DELTA}
\end{eqnarray}

Turning now to $\Delta u_{13}^\delta $ in \eqref{DECOUP_HOLDER_MOD_DIAG_U_1_DELTA}, expanding the frozen densities, exploiting as well \eqref{THE_CTR_DER_DENS_PROXY_DIAG_PERTURB} which gives that a diagonal perturbation of the density does not affect the related estimates, we write:
\begin{eqnarray*}
 &&|\Delta u_{i3}^\delta(s,x,x')|\\
 &\le& \Big|\int_{0}^1 d\lambda  \int_{\R^d }D^2_x \tilde  p^{\bP_i,\alpha,(\tau, \xi)}(s, t, x+\lambda (x'-x),y)(x'-x)[h^{\delta}(y)-h^{\delta}(m_{s, t}^{\tau,\xi}(x+\lambda(x'-x))]dy\Big| \bigg|_{(\tau_0,\tau,\xi)=(t_0,s,x)}\notag\\
&&+ \Big|\int_{\tau_0}^t dv\int_{0}^1 d\lambda  \int_{\R^d }D^2_x \tilde  p^{\bP_i,\alpha,(\tau, \xi)}(s, v, x+\lambda (x'-x),y)(x'-x)\times R_{\bP_i}^{\alpha,(\tau,\xi)}(v, y)dy\Big| \bigg|_{(\tau_0,\tau,\xi)=(t_0,s,x)}\notag\\
&\le &C\bigg(\|h\|_{\dot {\mathcal C}^\gamma}|x-x'|{(t-s)^{-(\frac 2\alpha-\frac{\gamma}{\alpha})}} \\
&&+|x-x'|^{\vartheta}\int_{t_0}^t dv (v-s)^{-\frac 2\alpha+\frac{2\eta}{\alpha}}\Big((t-v)^{-\Xi}\Phi(v)+ (t-v)^{-(\Xi+\frac \vartheta\alpha)}\Psi(v)\Big)|x-x'|^{1-\vartheta}\bigg).
\end{eqnarray*}
Recalling that, for $v\in [t_0,t], |x-x'|\le \big((v-s)/c_0\big)^{\frac 1\alpha} $, we derive:
\begin{eqnarray*}
 |\Delta u_{i3}^\delta(s,x,x')|&\le& 
C\bigg(\|h\|_{\dot {\mathcal C}^\gamma}|x-x'|^\vartheta(t-s)^{\frac{\gamma}{\alpha}-(\frac 1\alpha+\frac \vartheta\alpha)} \\
&&+|x-x'|^\vartheta\int_{t_0}^t dv (v-s)^{-\frac 1\alpha-\frac{\vartheta}{\alpha}+\frac{2\eta}{\alpha}}\big[(t-v)^{-\Xi}\Phi(v)+(t-v)^{-(\Xi+\frac \vartheta\alpha)}\Psi(v)\big]\bigg). 
 \end{eqnarray*}
Therefore, recalling from \eqref{DEF_PHI_U_1_DELTA}-\eqref{DEF_PSI_U_1_DELTA} that $\Xi+\frac \vartheta \alpha=\frac{1-\gamma+\vartheta}{\alpha} $, we get: 
\begin{equation}
(t-s)^{\Xi+\frac \vartheta \alpha}\frac{ |\Delta u_{i3}^\delta(s,x,x')|}{|x-x'|^\vartheta}\le C\bigg(\|h\|_{\dot {\mathcal C}^\gamma}+(t-s)^{1-\frac 1\alpha+\frac {2\eta}\alpha}\sup_{v\in [s,t]} \Phi(v)+(t-s)^{1-(\frac1\alpha+\frac\vartheta\alpha)+\frac {2\eta}\alpha}\sup_{v\in [s,t]} \Psi(v)\bigg).\label{CTR_HOLDER_DIAG_HORS_DIAG_3_U_1_DELTA}
\end{equation}
The remaining term $\Delta u_{i2}^\delta$ in \eqref{DECOUP_HOLDER_MOD_DIAG_U_1_DELTA} can be handled exactly as the previous $\Delta w_2 $ introduced in \eqref{DECOUP_HOLDER_MOD_DIAG}, replacing $w$ by $u_{i}^\delta$. This eventually gives:
\begin{equation}
\label{CTR_HOLDER_DIAG_HORS_DIAG_2_U_1_DELTA}
(t-s)^{\Xi+\frac{\vartheta}{\alpha}}\frac{|\Delta u_{i 2}^\delta(s, x, x')|}{|x-x'|^\vartheta}\le Cc_0^{\frac\vartheta\alpha}(1-c_0)^{-(\Xi+\frac{\vartheta}{\alpha})}\Psi(t_0).
\end{equation}
Plugging \eqref{CTR_HOLDER_DIAG_HORS_DIAG_1_U_1_DELTA}, \eqref{CTR_HOLDER_DIAG_HORS_DIAG_3_U_1_DELTA} and \eqref{CTR_HOLDER_DIAG_HORS_DIAG_2_U_1_DELTA} into \eqref{DECOUP_HOLDER_MOD_DIAG_U_1_DELTA}, we get that in the diagonal case $|x-x'|\le (t-s)^{\frac 1\alpha} $:
\begin{eqnarray}
(t-s)^{\Xi+\frac{\vartheta}{\alpha}}\frac{|D_x u_i^\delta(s, x)-Du_i^\delta(s, x')|}{|x-x'|^\vartheta}&\le& C\Big(\|h\|_{\dot {\mathcal C}^\gamma}+(t-s)^{1-\frac 1\alpha+\frac {2\eta}\alpha}\sup_{v\in [s,t]} \Phi(v)\notag\\
&&+(t-s)^{1+\frac {2\eta}\alpha-(\frac1\alpha+\frac\vartheta\alpha)}\sup_{v\in [s,t]} \Psi(v)
+c_0^{\frac\vartheta\alpha}(1-c_0)^{- (\Xi+\frac{\vartheta}{\alpha})}\Psi(t_0)\Big)\label{BD_HOLDER_DIAG_FINAL_ALPHA_GE_1_U_1_DELTA}.
\end{eqnarray}
Putting together \eqref{BD_HOLDER_DIAG_FINAL_ALPHA_GE_1_U_1_DELTA} and \eqref{PREAL_CTR_ALPHA_GE_1_HOLD_HD_U_1_DELTA}, we can conclude the proof of Lemma \ref{CHARACTERIZATION_OF_GRADIENT_EXPLOSION_AND_HOLDER_MODULUS} following exactly the procedure for $w$ described in Section \ref{THE_SEC_FOR_W} (using as well \eqref{ALL_READY_FOR_CIRCULAR_U_1_DELTA} and \eqref{PREAL_CTR_ALPHA_GE_1_GRAD_U_1_DELTA}). This yields the controls in equations \eqref{BD_CTR_U_I_DELTA_GRADIENT} and \eqref{BD_CTR_U_I_DELTA_MOD_HOLDER_GRADIENT}. The proof is now complete.\hfill $\square $

\subsection{Proof of Lemma \ref{LEM_DIFF_H}}

%There are two steps to prove this Lemma. The first one consists in somehow reproducing the arguments of Section \ref{THE_SEC_FOR_W} to analyze the behavior of $w$ introduced in \eqref{PDE_W} in order to establish the bounds \eqref{BD_CTR_U_I_DELTA_GRADIENT} and \eqref{BD_CTR_U_I_DELTA_MOD_HOLDER_GRADIENT}. 

We first analyze the sensitivities of the generators w.r.t. the measure arguments  or w.r.t to both the spatial points and the measure arguments. The  results of Lemma \ref{LEM_DIFF_H}, i.e. equations  \eqref{CTR_DIST_OP} and  \eqref{CTR_DIST_OP_HOLDER} then readily follow from Lemma \ref{LEMME_SENSI_MES_ET_SPACE} and the  a priori bounds of Lemma \ref{CHARACTERIZATION_OF_GRADIENT_EXPLOSION_AND_HOLDER_MODULUS}. 

%\subsubsection{Sensitivies of the generators w.r.t. to the measure and spatial arguments}
%We will actually only prove the sensitivity w.r.t both the measure and spatial arguments, which is needed to obtain  \eqref{CTR_DIST_OP_HOLDER}. This is the most delicate term to analyze. Equation \eqref{CTR_DIST_OP} can indeed be derived similarly.
%The following lemma is crucial for our analysis.
\begin{lem}[Sensitivity analysis of the generators w.r.t. the measure and the spatial parameters]\label{LEMME_SENSI_MES_ET_SPACE}
Let $\varphi $ be a function in ${\mathcal C}^{1+\vartheta}(\R^d,\R) $ with $\vartheta>\alpha-1 $ if $\alpha\ge 1 $.
 Under \A{A${}_S $} it then holds that, there exists a positive constant $C:=C($\A{A${}_S$}$)$ s.t. for all $v\in [s,t] $:
 \begin{eqnarray}
&&|(b(v,x,\bP_1(v))-b(v,x,\bP_2(v)))\cdot D\varphi(x)|\le C \bd_{2 \eta,s,t}(\bP_1,\bP_2)\|D\varphi\|_\infty,\label{DIFF_DRIFT_SIMPLE}\\
&& \Big| (L_v^{\bP_1,\alpha}-L_v^{\bP_2,\alpha}) \varphi(x)\Big|\le C \bd_{2 \eta,s,t}(\bP_1,\bP_2)\big((\|D\varphi\|_\infty+\|\varphi\|_\infty)\I_{\left\{\alpha <1\right\}}+\|D\varphi\|_{\dot {\mathcal C}^\vartheta}\I_{\left\{\alpha>1\right\}}+\|\varphi\|_{{\mathcal C}^{1+\vartheta}}\I_{\left\{ \alpha=1\right\}}\big).\label{DIFF_MEAS_SIMPLE}
\end{eqnarray}
 
 Furthermore, the following estimates also hold for the differences: there exists a positive constant $C:=C($\A{A${}_S$}$)$ such that for any $\lambda \in [0,1]$
 \begin{eqnarray}
&&\Big|(b(v,x,\bP_1(v))-b(v,x,\bP_2(v)))\cdot D\varphi(x)-(b(v,x',\bP_1(v))-b(s,x',\bP_2(v)))\cdot D\varphi(x')\Big|\notag\\
&\le& C \big( \bd_{ (1-\lambda)2 \eta,s,t}(\bP_1,\bP_2)|x-x'|^{\lambda 2\eta}\|D\varphi\|_{\infty}+\bd_{2  \eta,s,t}(\bP_1,\bP_2) |x-x'|^\vartheta\|D\varphi\|_{\dot {\mathcal C}^\vartheta}\big),\label{DIFF_B_MES}\\
&& \Big| (L_v^{\bP_1,\alpha}-L_v^{\bP_2,\alpha}) \varphi(x)-(L_v^{\bP_1,\alpha}-L_v^{\bP_2,\alpha}) \varphi(x'))\Big|\notag\\
&\le&C  \Big(\bd_{(1-\lambda) 2\eta,s,t}(\bP_1,\bP_2) |x-x'|^{\lambda 2\eta}\big((\|D\varphi\|_\infty+\|\varphi\|_\infty)\I_{\left\{\alpha <1\right\}}+\|D\varphi\|_{\dot {\mathcal C}^\vartheta}\I_{\left\{\alpha>1\right\}}+\|\varphi\|_{{\mathcal C}^{1+\vartheta}}\I_{\left\{\alpha=1\right\}}\big)\notag\\
&&+\bd_{2 \eta,s,t}(\bP_1,\bP_2) |x-x'|^{1+\vartheta-\alpha}\|D\varphi\|_{\dot {\mathcal C}^\vartheta} \Big).\label{DIFF_L_MES} 
 \end{eqnarray}
 \end{lem}
  
 \begin{proof}
 Write first using \A{B${}_H $} and \eqref{uniform:bound:diff:mes}:
 \begin{eqnarray*}
 \Big|(b(v,x,\bP_1(v))-b(v,x,\bP_2(v)))\cdot D\varphi(x)\Big|
 \le  C\bd_{2\eta,s,t}(\bP_1,\bP_2)\|D\varphi\|_{\infty}.
 \end{eqnarray*}
 This gives \eqref{DIFF_DRIFT_SIMPLE}.\\

 Write now
\begin{eqnarray}
 &&(L_v^{\bP_1,\alpha}-L_v^{\bP_2,\alpha}) \varphi(x)\notag\\
 &=&{\rm p.v.} \int_{\R^d} \big[\varphi(x+\sigma(v,x,\bP_1(v))z)-\varphi(x)\big] \nu(dz)
 -{\rm p.v.} \int_{\R^d} \big[\varphi(x+\sigma(v,x,\bP_2(v))z)-\varphi(x)\big] \nu(dz)
 \notag \\
 &=&{\rm p.v.} \int_{0}^{+\infty}\frac{dr}{r^{1+\alpha}}\int_{{\mathbb S}^{d-1}}\Lambda_{{\mathbb S}^{d-1}}(d\theta)
\big[ \varphi(x+r\theta)
-\varphi(x)\big]%\notag\\
%&&\times 
 \Big( \frac{g(\frac{(\sigma_x^{v,\bP_1})^{-1} \theta}{|(\sigma_x^{v,\bP_1})^{-1} \theta|})}{|(\sigma_x^{v,\bP_1})^{-1} \theta|^{d+\alpha} \d\big( \sigma_x^{v,\bP_1}\big)}
-\frac{g(\frac{(\sigma_x^{v,\bP_2})^{-1} \theta}{|(\sigma_x^{v,\bP_2})^{-1} \theta|})}{|(\sigma_x^{v,\bP_2})^{-1} \theta|^{d+\alpha} \d\big( \sigma_x^{v,\bP_2}\big)} \Big),\notag\\
\label{LA_DIFF_DES_OP_AVEC_TRANSFERT_DENSITE}
\end{eqnarray}
where we used the absolute continuity condition on the L\'evy measure, assumption \A{AC}, and the specific structure \eqref{EXPR_FOR_F} of the density w.r.t. the Lebesgue measure. For simplicity we have also denoted in the above equation for $i\in \{1,2\} $: 
\begin{equation} 
\label{DIFF_SHORT_TO_COMPARE}
\sigma_x^{v,\bP_{i}} := \sigma(v,x,\bP_i(v)).
\end{equation}
From now on and for notational convenience, we introduce the following map:
\begin{equation}\label{DEF_D}
\pp \ni m \mapsto D(v, x, \theta, m) :=  \frac{g(\frac{(\sigma(v,x, m))^{-1} \theta}{|(\sigma(v, x ,m))^{-1} \theta|})}{|(\sigma(v, x ,m))^{-1} \theta|^{d+\alpha} \d\big( \sigma(v, x ,m)\big)}
\end{equation}
\noindent so that, with our notations
\begin{align}
\delta D(v, x,\theta,\bP_1(v),\bP_2(v))& =D(v, x, \theta, \bP_1(v)) - D(v, x, \theta, \bP_2(v)) \notag \\
	& = \Big( \frac{g(\frac{(\sigma_x^{v,\bP_1})^{-1} \theta}{|(\sigma_x^{v,\bP_1})^{-1} \theta|})}{|(\sigma_x^{v,\bP_1})^{-1} \theta|^{d+\alpha} \d\big( \sigma_x^{v,\bP_1}\big)}
-\frac{g(\frac{(\sigma_x^{v,\bP_2})^{-1} \theta}{|(\sigma_x^{v,\bP_2})^{-1} \theta|})}{|(\sigma_x^{v,\bP_2})^{-1} \theta|^{d+\alpha} \d\big( \sigma_x^{v,\bP_2}\big)} \Big).\label{DEF_delta_D}
\end{align}

\noindent Under \A{D${}_H $} and \A{AC}, using the Lipschitz regularity of $g$ and \eqref{uniform:bound:diff:mes} with $\beta=0 $ applied to the diffusion coefficient $\sigma$, it is readily seen that:
\begin{equation}\label{CTR_D}
|\delta D(v,x,\theta,\bP_1(v),\bP_2(v))|\le C \bd_{2\eta,s,t}(\bP_1,\bP_2).
\end{equation}
Equation \eqref{DIFF_MEAS_SIMPLE} then readily follows from \eqref{LA_DIFF_DES_OP_AVEC_TRANSFERT_DENSITE}-\eqref{CTR_D} by usual Taylor expansions and the symmetry condition which allows to use cancellation techniques for the small jumps when $\alpha\ge 1 $.\\

Let us now turn to the differences. Write first,
 \begin{eqnarray*}
 &&\Big|(b(v,x,\bP_1(v))-b(v,x,\bP_2(v)))\cdot D\varphi(x)-(b(v,x',\bP_1(v))-b(s,x',\bP_2(v)))\cdot D\varphi(x')\Big|\\
 &\le& |b(v,x,\bP_1(v))-b(v,x,\bP_2(v))| |D\varphi(x)-D\varphi(x')|\\
 &&\quad + \big| (b(v,x,\bP_1(v))-b(v,x,\bP_2(v)))-(b(v,x',\bP_1(v))-b(v,x',\bP_2(v))) \big| |D\varphi (x')|\\
 &\le & C(\bd_{2\eta,s,t}(\bP_1,\bP_2) |x-x'|^\vartheta\|D\varphi\|_{\dot {\mathcal C}^\vartheta}+\bd_{(1-\lambda)2\eta,s,t}(\bP_1,\bP_2)|x-x'|^{\lambda 2\eta}\|D\varphi\|_{\infty}),
 \end{eqnarray*}
 
 \noindent using \A{B${}_H $}, \eqref{uniform:bound:diff:mes} and \eqref{holder:reg:diff:mes:b:or:sigma_stable} applied to the drift coefficient $b$, with $\beta=0 $ and $\beta=\lambda $ respectively, for the last inequality. This gives \eqref{DIFF_B_MES}.\\

Let us now turn to the difference of the non-local operators. From \eqref{LA_DIFF_DES_OP_AVEC_TRANSFERT_DENSITE} we get with the notation of \eqref{DEF_D}:
\begin{eqnarray}
&&\Big| (L_v^{\bP_1,\alpha}-L_v^{\bP_2,\alpha}) \varphi(x)-(L_v^{\bP_1,\alpha}-L_v^{\bP_2,\alpha}) \varphi(x'))\Big|\notag\\
&\le&\Big| {\rm p.v.} \int_{0}^{+\infty}\frac{dr}{r^{1+\alpha}}\int_{{\mathbb S}^{d-1}}\Lambda_{{\mathbb S}^{d-1}}(d\theta)
\big[ \varphi(x+r\theta)
-\varphi(x)\big] \delta D(v,x,\theta,\bP_1(v),\bP_2(v))\notag\\
&&-{\rm p.v.} \int_{0}^{+\infty}\frac{dr}{r^{1+\alpha}}\int_{{\mathbb S}^{d-1}}\Lambda_{{\mathbb S}^{d-1}}(d\theta)
\big[ \varphi(x'+r\theta)
-\varphi(x')\big] \delta D(v, x', \theta,\bP_1(v),\bP_2(v))\Big|\notag\\
&\le &\Big| {\rm p.v.} \int_{0}^{+\infty}\frac{dr}{r^{1+\alpha}}\int_{{\mathbb S}^{d-1}}\Lambda_{{\mathbb S}^{d-1}}(d\theta)
\big[ \varphi(x+r\theta)
-\varphi(x)\big] [\delta D(v,x,\theta,\bP_1(v),\bP_2(v))- \delta D(v,x',\theta,\bP_1(v),\bP_2(v))]\Big|\notag\\
&&+\Big|{\rm p.v.} \int_{0}^{+\infty}\frac{dr}{r^{1+\alpha}}\int_{{\mathbb S}^{d-1}}\Lambda_{{\mathbb S}^{d-1}}(d\theta)
\big[ \varphi(x'+r\theta)
-\varphi(x')-(\varphi(x+r\theta)
-\varphi(x))\big] \delta D(v,x',\theta,\bP_1{(v)},\bP_2{(v)})\Big|\notag\\
&=:&({\mathscr D}_1+{\mathscr D}_2)(v,x,x',\bP_1,\bP_2).\notag\\
\label{SPLIT_D1_D2}
\end{eqnarray}
  Under \A{AC} and \A{D${}_H $}, using the mean-value theorem and the fact that $g$ has a bounded a Lipschitz gradient, we deduce that the continuous map $D(v, x, \theta, .)$ defined by \eqref{DEF_D} admits a bounded and continuous functional derivative such that $(\rr^d)^2 \ni (x, y) \mapsto [\delta D(v, x, \theta, m) / \delta m](y)$ is $2\eta$-H\"older uniformly w.r.t. the variables $v$, $\theta$ and $m$, so that, similarly to \eqref{holder:reg:diff:mes:b:or:sigma_stable}, it is easily seen that there exists $C$ s.t. for all $v\in [s,t], \theta \in {\mathbb S}^{d-1} $ and for any $\lambda \in [0,1]$
  $$
  |\delta D(v, x,\theta,\bP_1(v),\bP_2(v))-\delta D(v, x',\theta,\bP_1(v),\bP_2(v))|\le C |x-x'|^{\lambda 2\eta}\bd_{(1-\lambda) 2\eta,s,t}(\bP_1,\bP_2),
  $$
  which readily gives:
  \begin{equation}\label{CTR_D1}
|  {\mathscr D}_1(v,x,x',\bP_1,\bP_2)|\le C|x-x'|^{\lambda 2\eta}\bd_{(1-\lambda) 2\eta,s,t}(\bP_1,\bP_2)\big((\|D\varphi\|_\infty+\|\varphi\|_\infty)\I_{\left\{\alpha <1\right\}}+\|D\varphi\|_{\dot {\mathcal C}^\vartheta}\I_{\left\{\alpha>1\right\}}+\|\varphi\|_{{\mathcal C}^{1+\vartheta}}\I_{\left\{\alpha=1\right\}}\big).
  \end{equation}
Let us turn now to ${\mathscr D}_2$. We will establish:
 \begin{eqnarray}\label{CTR_D2}
|  {\mathscr D}_2(v,x,x',\bP_1,\bP_2)|%&\le& |D(v,x',\bP_1,\bP_2)||\Delta^{\frac \alpha2}\varphi(x)-\Delta^{\frac \alpha2}\varphi(x')|\nonumber\\
&\le& C\bd_{2\eta,s,t}(\bP_1,\bP_2) \|D\varphi\|_{\dot{\mathcal C}^\vartheta}|x-x'|^{1+\vartheta-\alpha} .
  \end{eqnarray}
Write indeed assuming first that $\alpha<1 $:%\\
%\textcolor{red}{Attention: %il y a ambiguite entre $\theta $ indice de regularite et $\theta$ variable spherique. Changer plutot la variable spherique c'est plus local comment l'appeler, $\delta $?
 %Pareil pour le $\lambda $ des Taylor et $\lambda $ mesure de Lebesgue de la sphere. Utiliser $\ell $ pour la mesure de Lebsegue?}
\begin{eqnarray*}
&&|{\mathscr D}_2(v,x,x',\bP_1,\bP_2)|\\
&\le& \Big|\int_0^1d\lambda  \int_{r\le |x-x'|}\frac{dr}{r^{1+\alpha}}\int_{{\mathbb S}^{d-1}}\Lambda_{{\mathbb S}^{d-1}}(d\theta) \Big(D\varphi(x+\lambda r\theta) -D\varphi(x'+\lambda r\theta) \Big) \cdot r \theta \delta D(v,x',\theta,\bP_1{(v)},\bP_2{(v)}) \Big|\\
&&+\Big|\int_{0}^1 d\lambda \int_{r\ge |x-x'|}\frac{dr}{r^{1+\alpha}}\int_{{\mathbb S}^{d-1}}\Lambda_{{\mathbb S}^{d-1}}(d\theta)\\
&&\times \Big( D \varphi(x'+r\theta+\lambda(x-x'))-D\varphi(x'+\lambda (x-x'))\Big)\cdot (x-x') \delta D(v,x',\theta,\bP_1{(v)},\bP_2{(v)}) \Big|\\
&\le & C\bd_{2\eta,s,t}(\bP_1,\bP_2)\Big( \int_{r\in (0,|x-x'|]} \frac{dr}{r^{1+\alpha}} \|D\varphi\|_{\dot {\mathcal C}^\vartheta} |x-x'|^{\vartheta}r
+\int_{r\ge |x-x'| }\frac{dr}{r^{1+\alpha}}  \|D\varphi\|_{\dot {\mathcal C}^\vartheta}  r^{\vartheta}|x-x'|\Big) \\
&\le & C_{\vartheta,\alpha} \bd_{2\eta,s,t}(\bP_1,\bP_2)\|D\varphi\|_{\dot {\mathcal C}^\vartheta} |x-x'|^{1+\vartheta-\alpha},
\end{eqnarray*}
using \eqref{CTR_D} for the last but one inequality and $\alpha>\vartheta $ for the last one (see equation \eqref{BD_CTR_U_I_DELTA_MOD_HOLDER_GRADIENT}).

The only modifications needed for $\alpha\ge 1 $ concern the \textit{small jumps}. Indeed, we can introduce the compensator only up to the threshold $|x-x'|$. We are simply led to analyze:
\begin{eqnarray*}
\Big|\int_0^1d\lambda  \int_{r\in (0,|x-x'|]}\frac{dr}{r^{1+\alpha}} \int_{{\mathbb S}^{d-1}} \Lambda_{{\mathbb S}^{d-1}}(d\theta)\Big([D\varphi(x+\lambda r\theta) -D\varphi(x) ]-[D\varphi(x'+\lambda r\theta)-D\varphi(x')] \Big) \cdot r\theta\\
\times \delta D(v,x',\theta,\bP_1{(v)},\bP_2{(v)}) \Big|\\
\le C\bd_{2\eta,s,t}(\bP_1,\bP_2)\int_{r\in (0,|x-x'|]}\frac {dr}{r^{1+\alpha}}\|D\varphi\|_{\dot {\mathcal C}^\vartheta}  r^{1+\vartheta}\le C_{\alpha,\vartheta} \bd_{2\eta,s,t}(\bP_1,\bP_2)\|D\varphi\|_{\dot {\mathcal C}^\vartheta} |x-x'|^{1+\vartheta-\alpha}.
\end{eqnarray*}
This indeed gives \eqref{CTR_D2}. Plugging \eqref{CTR_D1} and \eqref{CTR_D2} into \eqref{SPLIT_D1_D2} we derive the statement \eqref{DIFF_L_MES}.
 \end{proof}

The bound in \eqref{CTR_DIST_OP} (resp. \eqref{CTR_DIST_OP_HOLDER}) then readily follows from \eqref{DIFF_DRIFT_SIMPLE}, \eqref{DIFF_MEAS_SIMPLE} (resp. \eqref{DIFF_B_MES}, \eqref{DIFF_L_MES}) and the a priori bounds \eqref{BD_CTR_U_I_DELTA_GRADIENT}, \eqref{BD_CTR_U_I_DELTA_MOD_HOLDER_GRADIENT} which {have been} established {in Lemma \ref{CHARACTERIZATION_OF_GRADIENT_EXPLOSION_AND_HOLDER_MODULUS} proven above}.

{We conclude this section by discussing a bit why in the current McKean-Vlasov setting the absolute continuity condition on the spectral measure (Assumption \A{AC}) seems essential.

\begin{remark}[About the absolute continuity of the spectral measure]
The key point is that, when considering the difference of the non-local operators associated with two different measure arguments it seems difficult to make explicitly the distance between the considered measures appear. Let us illustrate this fact considering the elementary following case.  Consider $\nu(dz) =\frac 12\sum_{i=1}^d (\delta_{e_i}+\delta_{-e_i})\frac{dr}{r^{1+\alpha}} $, where the $(e_i)_{i\in \leftB 1,d\rightB}$ correspond to the canonical vectors, corresponding, up to a suitable normalizing constant to the example of the cylindrical Laplacian. In this case we cannot do the change of variable of \eqref{LA_DIFF_DES_OP_AVEC_TRANSFERT_DENSITE} and directly write for the difference:
\begin{eqnarray}
&&\Big| (L_v^{\bP_1,\alpha}-L_v^{\bP_2,\alpha}) \varphi(x)\Big|\notag\\
&\le&\Big|\sum_{i=1}^d \frac 12\sum_{j\in \{0,1\}}{\rm p.v.}  \int_{0}^{+\infty}\frac{dr}{r^{1+\alpha}}
\Big(\big[ \varphi(x+(-1)^j \sigma_ x^{v,\bP_1}  re_i)
-\varphi(x)\big] -\big[ \varphi(x+(-1)^j \sigma_ x^{v,\bP_2}  re_i)
-\varphi(x)\big]\Big)\Big|.\notag\\
\label{EQ_DISC_AC}
\end{eqnarray}
As usual, the strategy consists in separating the small and large jumps. Let us first consider the case $\alpha<1 $. In that case we write, for a threshold ${\mathscr T} $ (possibly depending on $(\sigma_x^{v,\bP_i})_{i\in \{1,2\}} $):
\begin{eqnarray*}
&&\Big|\sum_{i=1}^d \frac 12\sum_{j\in \{0,1\}}{\rm p.v.}  \int_{0}^{\mathscr T}\frac{dr}{r^{1+\alpha}}
\Big(\big[ \varphi(x+(-1)^j \sigma_ x^{v,\bP_1}  re_i)
-\varphi(x)\big] -\big[ \varphi(x+(-1)^j \sigma_ x^{v,\bP_2}  re_i)
-\varphi(x)\big]\Big)\Big|\\
&=&\Big|\sum_{i=1}^d \frac 12\sum_{j\in \{0,1\}}{\rm p.v.}  \int_{0}^{\mathscr T}\frac{dr}{r^{1+\alpha}}
\int_0^1 d\lambda  D \varphi(x+(-1)^j ((1-\lambda)\sigma_ x^{v,\bP_1}  +\lambda \sigma_ x^{v,\bP_2} ) e_i r)\cdot (-1)^j (\sigma_ x^{v, \bP_1}-\sigma_ x^{v, \bP_2} )  re_i
\big]\Big)\Big|\\
&\le& C_{\alpha,{\mathscr T}}\|D\varphi\|_\infty \bd_{2\eta}(\bP_1(v),\bP_2(v)).
\end{eqnarray*}
Hence, since the required term $\bd_{2\eta}(\bP_1(v),\bP_2(v)) $ appears this way, it seems natural to consider a threshold ${\mathscr T} $ at a \textit{macro} scale, and in particular independent of $\sigma_ x^{v,\bP_2}-\sigma_ x^{v,\bP_1} $. 
This choice anyhow yields problems for the \textit{large jumps}, i.e. those above the threshold ${\mathscr T}$. Namely,
\begin{eqnarray*}
&&\Big|\sum_{i=1}^d \frac 12\sum_{j\in \{0,1\}}{\rm p.v.}  \int_{\mathscr T}^{+\infty}\frac{dr}{r^{1+\alpha}}
\Big(\big[ \varphi(x+(-1)^j \sigma_ x^{v,\bP_1}  re_i)
-\varphi(x)\big] -\big[ \varphi(x+(-1)^j \sigma_ x^{v,\bP_2}  re_i)
-\varphi(x)\big]\Big)\Big|\\
& \le &\sum_{i=1}^d \frac 12\sum_{j\in \{0,1\}} %{\rm p.v.}  
\int_{\mathscr T}^{+\infty}\frac{dr}{r^{1+\alpha}}
\|\varphi\|_{\dot {\mathcal C}^\beta} |\sigma_ x^{v,\bP_2}-\sigma_ x^{v,\bP_1} |^\beta  r^\beta,
\end{eqnarray*}
for any $\beta \in (0,1) $ s.t.  $\beta<\alpha $ in order to preserve some integrability and where $\|\varphi\|_{\dot {\mathcal C}^\beta} $ again stands for the $\beta $-H\"older modulus of $\varphi $ ({homogeneous H\"older norm}) which can be estimated from $\|\varphi\|_\infty, \|D\varphi\|_\infty $ through elementary interpolation. Hence, we cannot recover that way the required control with exactly the distance, we only end up with:
\begin{eqnarray*}
&&\Big|\sum_{i=1}^d \frac 12\sum_{j\in \{0,1\}}{\rm p.v.}  \int_{\mathscr T}^{+\infty}\frac{dr}{r^{1+\alpha}}
\Big(\big[ \varphi(x+(-1)^j \sigma_ x^{v,\bP_1}  re_i)
-\varphi(x)\big] -\big[ \varphi(x+(-1)^j \sigma_ x^{v,\bP_2}  re_i)
-\varphi(x)\big]\Big)\Big|\\
&\le& C_{\alpha,\beta,{\mathscr T}} \|\varphi\|_{\dot {\mathcal C}^\beta} \bd_{\eta}^\beta(\bP_1(v),\bP_2(v)).
\end{eqnarray*}
One could object that we do not have exploited the \textit{full} regularity of $\varphi \in {\mathcal C}^{1+\vartheta}(\R^d,\R) $. But for the small jumps this will not a priori improve the dependence in the distance (see also the small jump part for $\alpha>1 $ below).

If now $\alpha> 1 $ the problems are reversed. Indeed, the large jumps can be readily controlled with the expected bound. Indeed,
\begin{eqnarray*}
&&\Big|\sum_{i=1}^d \frac 12\sum_{j\in \{0,1\}}{\rm p.v.}  \int_{\mathscr T}^{+\infty}\frac{dr}{r^{1+\alpha}}
\Big(\big[ \varphi(x+(-1)^j \sigma_ x^{v,\bP_1}  re_i)
-\varphi(x)\big] -\big[ \varphi(x+(-1)^j \sigma_ x^{v,\bP_2}  re_i)
-\varphi(x)\big]\Big)\Big|\\
&=&\Big|\sum_{i=1}^d \frac 12\sum_{j\in \{0,1\}}{\rm p.v.}  \int_{\mathscr T}^{+\infty}\frac{dr}{r^{1+\alpha}}
\|D\varphi\|_\infty |\sigma_ x^{v,\bP_2}-\sigma_ x^{v,\bP_1} |  r\Big|\le C_{\alpha,{\mathscr T}}\|D\varphi\|_\infty \bd_{2\eta}(\bP_1(v),\bP_2(v)).
\end{eqnarray*}
On the other hand, for the small jumps, we are led to use explicitly the $\vartheta $-H\"older regularity of the gradient $D\varphi $ (recall from \eqref{BD_CTR_U_I_DELTA_MOD_HOLDER_GRADIENT} that $\vartheta=\vartheta(\alpha,\eta,\varepsilon)=\alpha\Big[1-\varepsilon-\big(\frac 1\alpha-\frac \eta\alpha\big)\Big] $). Write precisely, 
\begin{eqnarray*}
&&\Big|\sum_{i=1}^d \frac 12\sum_{j\in \{0,1\}}{\rm p.v.}  \int_{0}^{\mathscr T}\frac{dr}{r^{1+\alpha}}
\Big(\big[ \varphi(x+(-1)^j \sigma_ x^{v,\bP_1}  re_i)
-\varphi(x)\big] -\big[ \varphi(x+(-1)^j \sigma_ x^{v,\bP_2}  re_i)
-\varphi(x)\big]\Big)\Big|\\
&=&\Big|\sum_{i=1}^d \frac 12\sum_{j\in \{0,1\}} (-1)^j  {\rm p.v.}  \int_{0}^{\mathscr T}\frac{dr}{r^{1+\alpha}}
\Big(\int_0^1 d\lambda \big[ D\varphi(x+\lambda (-1)^j \sigma_ x^{v,\bP_1}  re_i)\cdot  \sigma_ x^{v,\bP_1}  \underline{+} D\varphi(x+\lambda (-1)^j \cdot\sigma_ x^{v,\bP_1}  re_i) \sigma_ x^{v,\bP_2}
\\
&&- D\varphi(x+ \lambda(-1)^j \sigma_ x^{v,\bP_2}  re_i)\cdot \sigma_ x^{v,\bP_2}  \big]\Big)re_i\Big|\\
&= &\Big|\sum_{i=1}^d \frac 12\sum_{j\in \{0,1\}} (-1)^j  {\rm p.v.}  \int_{0}^{\mathscr T}\frac{dr}{r^{1+\alpha}}
\Big(\int_0^1 d\lambda \big[ (D\varphi(x+\lambda (-1)^j \sigma_ x^{v,\bP_1}  re_i)-D\varphi(x)) \cdot(\sigma_ x^{v,\bP_1}-\sigma_ x^{v,\bP_2})\\
&&  + \big(D\varphi(x+\lambda (-1)^j \sigma_ x^{v,\bP_1}  re_i) - D\varphi(x+ \lambda (-1)^j \sigma_ x^{v,\bP_2}  re_i)\big)\cdot \sigma_ x^{v,\bP_2}  \big]\Big)re_i\Big|\\
&\le & C \int_0^{\mathscr T} \frac{dr}{r^\alpha} \|D\varphi\|_{\dot {\mathcal C}^\vartheta}r^\vartheta \Big(\bd_{\eta}(\bP_1(v),\bP_2(v))+\bd_{\eta}^\vartheta(\bP_1(v),\bP_2(v)) \Big)\\
&\le&  C_{\alpha,\vartheta}\|D\varphi\|_{\dot {\mathcal C}^\vartheta} \Big(\bd_{\eta}(\bP_1(v),\bP_2(v))+\bd_{\eta}^\vartheta(\bP_1(v),\bP_2(v)) \Big),
\end{eqnarray*}
where we used a cancellation argument for the first  inequality, and exploit as well that $\alpha-\vartheta=1-\eta +\varepsilon \alpha<1 $ for $\varepsilon$ small enough, for the last one. In any case we see that this is not enough to conclude whereas assumption \A{AC} allows to transfer all the sensitivity analysis to the spectral measure (see again the proof of the previous lemma).
\end{remark}

\subsection{Proof of Lemma \ref{NEW_LEMMA_R}}
\label{SEC_PROOF_NEW_LEMMA_R}
The stated controls follow reproducing the \textit{transfer} arguments of the diffusion coefficient to the density of the L\'evy measure in equation \eqref{LA_DIFF_DES_OP_AVEC_TRANSFERT_DENSITE} for the integro-differential part. The gradient term is controlled directly. Namely, from \eqref{THE_TERM_PERTURB_FIXED_MEASURE},
\begin{eqnarray*}
|R_{\bP_1}^{\alpha,(\tau,\xi)}(v,y)|&\le& |\big(b(v,y,\bP_1(v))-b(v,\theta_{v,\tau}(\xi),\bP_1(v)) \big) \cdot D w(v, y)|+|\Big(L_s^{\bP_1,\alpha}-\tilde L_s^{\bP_1,\alpha,(\tau,\xi)}\Big) w(v,y)|\\
 &\le& \|b(v,\cdot,\bP_1(v))\|_{\dot {\mathcal C}^{2\eta}}|y-\theta_{v,\tau}(\xi)|^{2\eta}|Dw(v,y)|+ \bigg|{\rm p.v.} \int_{\R^d} \big[\varphi(y+\sigma(v,y,\bP_1(v))z)-\varphi(y)\big] \nu(dz)\\
&& -{\rm p.v.} \int_{\R^d} \big[\varphi(y+\sigma(v,\theta_{v,\tau}(\xi),\bP_1(v))z)-\varphi(y)\big] \nu(dz)\bigg|
 \notag \\
 &\le&\|b(v,\cdot,\bP_1(v))\|_{\dot {\mathcal C}^{2\eta}}|y-\theta_{v,\tau}(\xi)|^{2\eta}|Dw(v,y)|+ \bigg|{\rm p.v.} \int_{0}^{+\infty}\frac{dr}{r^{1+\alpha}}\int_{{\mathbb S}^{d-1}}\Lambda_{{\mathbb S}^{d-1}}(d\theta)
\big[ \varphi(y+r\theta)
-\varphi(y)\big]\notag\\
&&\times 
 \Big( \frac{g(\frac{(\sigma(v,y,\bP_1(v)))^{-1} \theta}{|(\sigma(v,y,\bP_1(v)))^{-1} \theta|})}{|(\sigma(v,y,\bP_1(v)))^{-1} \theta|^{d+\alpha} \d\big( \sigma(v,y,\bP_1(v))\big)}\\
&&-
\frac{g(\frac{(\sigma(v,\theta_{v,\tau}(\xi),\bP_1(v)))^{-1} \theta}{|(\sigma(v,\theta_{v,\tau}(\xi),\bP_1(v)))^{-1} \theta|})}{|(\sigma(v,\theta_{v,\tau}(\xi),\bP_1(v)))^{-1} \theta|^{d+\alpha} \d\big( \sigma(v,\theta_{v,\tau}(\xi),\bP_1(v))\big)}
%\frac{g(\frac{(\sigma_x^{v,\bP_2})^{-1} \theta}{|(\sigma_x^{v,\bP_2})^{-1} \theta|})}{|(\sigma_x^{v,\bP_2})^{-1} \theta|^{d+\alpha} \d\big( \sigma_x^{v,\bP_2}\big)} \Big)
\bigg|.
\end{eqnarray*}
Again the Lipschitz regularity of $g$, the non-degeneracy and $2\eta $-H\"older spatial continuity of $\sigma$ and the symmetry of the L\'evy measure yield:
\begin{eqnarray}\label{CTR_PT_R_LEM}
 \big|R_{\bP_1}^{\alpha,(\tau,\xi)}(v,y)\big|
\le C\Big(\|b(v,\cdot,\bP_1(v))\|_{\dot {\mathcal C}^{2\eta}}|D_y w(v, y)|+ \|\sigma(v,\cdot,\bP_1(v))\|_{{\mathcal C}^{2\eta}}\big[ (\|w\|_\infty+\|D_x w(v,\cdot)\|_\infty)\I_{\left\{\alpha<1\right\}}\notag\\
+\|D_xw(v,\cdot)\|_{{\mathcal C}^{\vartheta}}\I_{\left\{\alpha>1\right\}}+\|w(v,\cdot)\|_{{\mathcal C}^{1+\vartheta}}\I_{\left\{\alpha=1\right\}}\big]\Big)|y-\theta_{v,\tau}(\xi)|^{2\eta},
\end{eqnarray}
which the is required statement.

%There exists a constant $C$ s.t. for all $(\tau,\xi),\ (v, y)\in [0,T]\times \R^d $:
%\begin{eqnarray}\label{CTR_PT_R_LEM}
% \big|R_{\bP_1}^{\alpha,(\tau,\xi)}(v,y)\big|
%\le C\Big(\|b(v,\cdot,\bP_1(v))\|_{\dot {\mathcal C}^{2\eta}}|D_y w(v, y)|+ \|\sigma(v,\cdot,\bP_1(v))\|_{{\mathcal C}^{2\eta}}\big[ (\|w\|_\infty+\|D_x w(v,\cdot)\|_\infty)\I_{\left\{\alpha<1\right\}}\notag\\
%+\|D_xw(v,\cdot)\|_{{\mathcal C}^{\vartheta}}\I_{\left\{\alpha>1\right\}}+\|w(v,\cdot)\|_{{\mathcal C}^{1+\vartheta}}\I_{\left\{\alpha=1\right\}}\big]\Big)|y-\theta_{v,\tau}(\xi)|^{2\eta}.
%\end{eqnarray}

\section{Theorem \ref{THM_WP_NL} in the supercritical case under the sole condition $\alpha > 2\eta \vee (1-\eta)$}
\label{sec:modification:proof}

In order to establish Theorem \ref{THM_WP_NL} under {\bf(A${}_S $)} and the sole condition $\alpha > 2\eta \vee (1-\eta)$ in the supercritical case $\alpha <1$, we proceed as follows. Restarting from \eqref{PREAL_CIRC} with $h \in \mathcal{C}^{\gamma_1}$, with $\alpha + \gamma_1 >1$ and $\gamma_1 \in (0,2\eta]$ to be specified later on and employing \eqref{CTR_DIST_OP} instead of \eqref{CTR_H_BSUP}, similarly to \eqref{ALL_READY_FOR_CIRCULAR}, we derive:
\begin{eqnarray}
|w(s,x)|&\le& C\bigg(\bd_{2\eta,s,t}(\bP_1,\bP_2)(t-s)^{ 1-(\frac{1}{\alpha}-\frac{\gamma_1}{\alpha})}+(t-s)\|w\|_\infty%\notag\\
%&&
+\int_s^t dv (v-s)^{\frac{2\eta}{\alpha}} \|Dw(v,\cdot)\|_\infty\bigg)%\notag\\
\label{PREAL_CIR_SUPER_CRITICAL_APP}
\end{eqnarray}

%\noindent which in turn by taking $0\leq t-s \leq T$ sufficiently small yields
%\begin{eqnarray}
%|w(s,x)|&\le& C\bigg(\bd_{2\eta,s,t}(\bP_1,\bP_2)(t-s)^{ 1-(\frac{1}{\alpha}-\frac{\gamma}{\alpha})}%\notag\\
%%&&
%+\int_s^t dv (v-s)^{\frac{2\eta}{\alpha}} \|Dw(v,\cdot)\|_\infty\bigg).%\notag\\
%\label{PREAL_CIR_SUPER_CRITICAL_suite}
%\end{eqnarray}
Observe that, for the first term of the above r.h.s., we have assumed that $(\frac{1}{\alpha}-\frac{\gamma_1}{\alpha})<1\iff \alpha+\gamma_1>1  $. This seems to be somehow a necessary condition to derive a smoothing effect for the contribution $H_{\bP_1,\bP_2}^\delta $. 
%Since $\eta\in (0,1/2] $ this implies that we restrict to $\alpha \in (1/2,1) $ in the super-critical case.

Let us now proceed  from the cancellation argument as in \eqref{DIFF_WITH_CANCEL}. We get from \eqref{CTR_DIST_OP_HOLDER} that in the current case:
 \begin{eqnarray}\label{CTR_H_BSUP_HOLDER_alpha_LE_1_WITH_ALL_APP}
|H_{\bP_1,\bP_2}u_2^\delta(v, y)-H_{\bP_1,\bP_2}u_2^\delta(v,\theta_{v,\tau}(\xi))|&\le& C\bd_{(1-\lambda)2\eta}(\bP_1(v),\bP_2(v)) |y-\theta_{v,\tau}(\xi)|^{2\lambda \eta} (t-v)^{-\big(\frac 1\alpha-\frac{\gamma_1}{ \alpha}\big)}\notag\\
&&+ C\bd_{2\eta}(\bP_1(v),\bP_2(v)) |y-\theta_{v,\tau}(\xi)|^\vartheta(t-v)^{-1+\varepsilon}
%&\le &C\bd_\eta(\bP_1(v),\bP_2(v))|y-\theta_{v,\tau}(\xi)|^{\alpha-1+\eta-\alpha\varepsilon}(t-v)^{-1+\varepsilon},
\end{eqnarray}
recalling from \eqref{BD_CTR_U_I_DELTA_MOD_HOLDER_GRADIENT} that $0< \vartheta:=\alpha\Big[1-\varepsilon-\big(\frac 1\alpha-\frac{\gamma_1}{\alpha}\big)\Big]< \alpha-1+\gamma_1$.

Now, similarly to \eqref{PREAL_CTR_ALPHA_GE_1_GRAD}, using \eqref{CTR_H_BSUP_HOLDER_alpha_LE_1_WITH_ALL_APP} on the time interval $[s, \frac{t+s}{2}]$ and \eqref{CTR_DIST_OP} on the time interval $[\frac{t+s}{2}, t]$, taking also $\lambda$ large enough so that $2\lambda \eta+\alpha >1$, we get:
\begin{eqnarray}
|Dw(s,x)|&\le& C\Big(\int_s^{\frac{t+s}{2}}  (v-s)^{-\frac{1}{\alpha}+ \frac{2\lambda \eta}{\alpha}}\bd_{(1-\lambda)2\eta}(\bP_1(v),\bP_2(v)) (t-v)^{-\frac{1-\gamma_1}{\alpha}} \, dv\notag\\
&& 
+ \int_s^{\frac{t+s}{2}}  (v-s)^{-\frac{1}{\alpha}+ \frac{\vartheta}{\alpha}}\bd_{2\eta}(\bP_1(v),\bP_2(v)) (t-v)^{-1+\varepsilon} \, dv\notag\\
&&
+ \int_{\frac{t+s}{2}}^{t}  (v-s)^{-\frac{1}{\alpha}}   \bd_{2\eta}(\bP_1(v),\bP_2(v)) (t-v)^{-\frac{1-\gamma_1}{\alpha}} \, dv \notag\\
&&
+\int_s^t  (v-s)^{-\frac 1\alpha+\frac{2\eta}{\alpha}}\big(\|w(v, .)\|_\infty+\|Dw(v,.)\|_\infty\big) \, dv\Big)\notag\\
&\le &C\Big(\bd_{(1-\lambda) 2\eta,s,t}(\bP_1,\bP_2) (t-s)^{1-(\frac 2\alpha-\frac{2 \lambda\eta}{\alpha} - \frac{\gamma_1}{\alpha}) } + \bd_{2\eta,s,t}(\bP_1,\bP_2) (t-s)^{1-(\frac 2\alpha- \frac{\gamma_1}{\alpha})} \notag\\
& & \quad +\int_s^t  (v-s)^{-\frac 1\alpha+\frac{2\eta}{\alpha}}\big(\|w\|_{\infty}+\|Dw(v,\cdot)\|_\infty\big) \, dv\Big),\label{CTR_GRAD_ALPHA_SUPER_CRITICAL_APP}
\end{eqnarray}
provided that $\vartheta$ is sufficiently large so that
\begin{equation}\label{COND_SUPER_CRITIQUE_APP}\vartheta+ \alpha >1.
\end{equation}
With the notation of \eqref{DEF_PHI} for $\big(\Phi(v)\big)_{v\in [s,t]} $,  taking in our current supercritical case  $\Xi= \frac{1}{\alpha}-\frac{\gamma_1}{\alpha} + \frac{\varepsilon}{2} <1$, we derive from 
\eqref{PREAL_CIR_SUPER_CRITICAL_APP} and \eqref{CTR_GRAD_ALPHA_SUPER_CRITICAL_APP}:
\begin{eqnarray}
|w(s,x)|&\le& C\bigg(\bd_{2 \eta,s,t}(\bP_1,\bP_2)(t-s)^{ 1-(\frac{1}{\alpha}-\frac{\gamma_1}{\alpha})}
+(t-s)\|w\|_\infty+ \sup_{v\in [s,t]}\Phi(v)\int_s^t  (v-s)^{\frac{2\eta}{\alpha}}(t-v)^{-\Xi} \, dv \bigg)\notag\\
&\le& C\bigg(\bd_{2 \eta,s,t}(\bP_1,\bP_2)(t-s)^{ 1-(\frac{1}{\alpha}-\frac{\gamma_1}{\alpha})}+ (t-s)\|w\|_\infty
+\sup_{v\in [s, t]}\Phi(v) (t-s)^{1+\frac{2\eta}{\alpha}-\Xi} \bigg),%\notag\\
\label{PREAL_CIR_SUPER_CRITICAL_2_APP}
\end{eqnarray}
and 
\begin{eqnarray*}
|Dw(s,x)|
&\le &C\Big(\bd_{(1-\lambda) 2\eta,s,t}(\bP_1,\bP_2) (t-s)^{1-(\frac 2\alpha-\frac{2 \lambda\eta}{\alpha} - \frac{\gamma_1}{\alpha}) } 
+ \bd_{2\eta,s,t}(\bP_1,\bP_2) (t-s)^{1-(\frac 2\alpha- \frac{\gamma_1}{\alpha})}  \\
& & \quad +\sup_{v\in [s,t]}\Phi(v)\int_s^t  (v-s)^{-\frac 1\alpha+\frac{2\eta}{\alpha}}(t-v)^{-\Xi} \, dv\Big)\\
&\le &C\Big(\bd_{(1-\lambda) 2\eta,s,t}(\bP_1,\bP_2) (t-s)^{1-(\frac 2\alpha-\frac{2 \lambda\eta}{\alpha} - \frac{\gamma_1}{\alpha}) } 
+ \bd_{2\eta,s,t}(\bP_1,\bP_2) (t-s)^{1-(\frac 2\alpha- \frac{\gamma_1}{\alpha})}  +\sup_{v\in [s,t]}\Phi(v)(t-s)^{1-(\frac 1\alpha-\frac{2\eta}\alpha)-\Xi}\Big)
\end{eqnarray*}
which in turn gives
\begin{align*}
(t-s)^{\Xi}|Dw(s,x)|& \le C\Big(\bd_{(1-\lambda) 2\eta,s,t}(\bP_1,\bP_2) (t-s)^{1-(\frac 2\alpha-\frac{2 \lambda\eta}{\alpha} - \frac{\gamma_1}{\alpha} ) + \Xi } 
+ \bd_{2\eta,s,t}(\bP_1,\bP_2) (t-s)^{1-(\frac 2\alpha- \frac{\gamma_1}{\alpha})+ \Xi} \\
& \quad   +\sup_{v\in [s,t]}\Phi(v)(t-s)^{1-(\frac 1\alpha-\frac{2\eta}\alpha)}\Big).
\end{align*}

Taking into account that $0< t-s \leq T$ is sufficiently small, from the preceding inequality, we obtain
$$
\sup_{v\in [s,t]}\Phi(v) \leq C\Big(\bd_{(1-\lambda) 2\eta,s,t}(\bP_1,\bP_2) (t-s)^{1-(\frac 2\alpha-\frac{2 \lambda\eta}{\alpha} - \frac{\gamma_1}{\alpha} ) + \Xi } +\bd_{2\eta,s,t}(\bP_1,\bP_2) (t-s)^{1-(\frac 2\alpha- \frac{\gamma_1}{\alpha}) + \Xi} \Big)
$$

\noindent and plugging the previous estimate into \eqref{PREAL_CIR_SUPER_CRITICAL_2_APP} yields
\begin{align*}
|w(s, x)|&\le C\bigg(\bd_{2 \eta, s, t}(\bP_1,\bP_2) [ (t-s)^{ 1-(\frac{1}{\alpha}-\frac{\gamma_1}{\alpha})} + (t-s)^{2-(\frac{2}{\alpha}-\frac{2\eta}{\alpha}-\frac{\gamma_1}{\alpha})}]
+\bd_{(1-\lambda) 2\eta,s,t}(\bP_1,\bP_2) (t-s)^{2-(\frac 2\alpha-\frac{2 \lambda\eta}{\alpha} - \frac{2\eta}{\alpha}- \frac{\gamma_1}{\alpha} )  } \notag \\ 
& \quad + (t-s)\|w\|_\infty\bigg)\notag\\
&\le C\bigg(\bd_{2 \eta, s, t}(\bP_1,\bP_2) (t-s)^{ 1-(\frac{1}{\alpha}-\frac{\gamma_1}{\alpha})}
+\bd_{(1-\lambda) 2\eta,s,t}(\bP_1,\bP_2) (t-s)^{2-(\frac 2\alpha-\frac{2 \lambda\eta}{\alpha} - \frac{2\eta}{\alpha}- \frac{\gamma_1}{\alpha} )  } +(t-s)\|w\|_\infty \bigg)\notag
%&\le C\bigg(\bd_{2 \eta,s,t}(\bP_1,\bP_2)(t-s)^{ 1-(\frac{1}{\alpha}-\frac{\gamma}{\alpha})}
%+\sup_{v\in [s, t]}\Phi(v) (t-s)^{1+\frac{2\eta}{\alpha}-\Xi} \bigg).
\end{align*}

\noindent where for the last inequality we used the fact that $2\eta+\alpha >1$. 

We now select $\gamma_1 = 2\eta$. We importantly recall the condition on the regularization parameter $\vartheta$: $1-\alpha < \vartheta < \alpha-1+\gamma = \alpha -1 + 2\eta$ so that a necessary condition is $\alpha+\eta>1$. Taking into account this condition, we next optimize the previous inequality on $w$ w.r.t $h \in \mathcal{C}^{\gamma_1}(\R^d)$. We thus deduce
\begin{equation}
\bd_{2\eta, s, t}(\mathscr{\mathbf T}(\bP_1),\mathscr{\mathbf T}(\bP_2)) \le C\bigg(\bd_{2 \eta, s, t}(\bP_1,\bP_2) (t-s)^{ 1-(\frac{1}{\alpha}-\frac{2\eta}{\alpha})}
+\bd_{(1-\lambda) 2\eta,s,t}(\bP_1,\bP_2) (t-s)^{2-(\frac 2\alpha-\frac{2 \lambda\eta}{\alpha} - \frac{4\eta}{\alpha})  } \bigg) \label{inequality:first:iteration:T:operator}
\end{equation}

\noindent for any $\lambda \in [0,1]$ s.t. $2\lambda \eta+ \alpha >1$ under the assumption that $\alpha + \eta>1$.

Our final step consists in iterating the previous analysis. We aim at establishing an estimate similar to \eqref{inequality:first:iteration:T:operator} but for the map $\mathscr{\mathbf T}^2$ and a distance $\bd_{(1-\lambda)2\eta, s, t}$ for a well-chosen $\lambda \in [0,1]$. Keeping in mind the inequality \eqref{inequality:first:iteration:T:operator}, we proceed in a similar manner but replacing $\bP_i$ by $\mathbf{T}(\bP_i)$. For $h \in \mathcal{C}^{\gamma_2}$, we consider $u_i^\delta(s, x)$, $i=1,2$ satisfying
\begin{equation}\label{EDP_U_I_DELTA_APP}
\begin{cases}
\Big(\partial_s +\mathscr{A}_s^{\mathscr{\mathbf T}(\bP_i)}\Big)u_i^\delta(s, x) =0,\ (s,x)\in [0,t)\times \R^d,\\
u_i^\delta(t,x)=h^\delta(x),\ x\in \R^d,
\end{cases}
\end{equation}

\noindent and $u_i^\delta\in L^\infty([0,t],C^{\alpha+\gamma_2}(\R^d,\R)) $. 

To compare both semigroups, we again write the PDE satisfied by $w(s, x):=w^\delta(s, x)=(u_1^\delta-u_2^\delta)(s, x)$. Namely,
\begin{equation}
\label{PDE_W_APP}
\left\{
\begin{split}
\Big(\partial_s +\mathscr{A}_s^{\mathscr{\mathbf T}(\bP_1)})w(s,x)&=-\Big( \big[b(s,x,\mathscr{\mathbf T}(\bP_1))-b(s,x,\mathscr{\mathbf T}(\bP_2)) \big]\cdot D_xu_2^\delta(s,x)\\
&\hspace*{.5cm}+(L_s^{\mathscr{\mathbf T}(\bP_1),\alpha}-L_s^{{{\mathbf T}}(\bP_2),\alpha})u_2^\delta(s,x) \Big)=: -H_{\mathscr{\mathbf T}(\bP_1),\mathscr{\mathbf T}(\bP_2)} u^{\delta}_2(s,x),\ (s,x)\in [0,t)\times \R^d,\\
w(t,x)&=0,\ x\in \R^d.
\end{split}
\right.
\end{equation}

The inequality \eqref{CTR_H_BSUP} now becomes
 \begin{align}
|H_{\mathbf{T}(\bP_1),\mathbf{T}(\bP_2)}u_2^\delta(v, y)|&\le C\bd_{2\eta}(\mathbf{T}(\bP_1)(v),\mathbf{T}(\bP_2)(v))  (t-v)^{-\frac{1-\gamma_2}{\alpha}} .\label{new:estimate:H:operator}
\end{align}

This new estimate will be combined with \eqref{inequality:first:iteration:T:operator} in the sequel. Similarly to \eqref{PREAL_CIR_SUPER_CRITICAL_APP}, from \eqref{inequality:first:iteration:T:operator}, we get
\begin{eqnarray}
|w(s,x)|&\le& C\bigg(\bd_{2\eta,s,t}(\mathbf{T}(\bP_1),\mathbf{T}(\bP_2))(t-s)^{ 1-(\frac{1}{\alpha}-\frac{\gamma_2}{\alpha})}
+(t-s)\|w\|_\infty +\int_s^t dv (v-s)^{\frac{2\eta}{\alpha}} \|Dw(v,\cdot)\|_\infty\bigg) \notag\\
& \leq & C \bigg(\bd_{(1-\lambda)2 \eta, s, t}(\bP_1,\bP_2) (t-s)^{ 2-(\frac{2}{\alpha}-\frac{2\eta}{\alpha} - \frac{\gamma_2}{\alpha})}+(t-s)\|w\|_\infty +\int_s^t dv (v-s)^{\frac{2\eta}{\alpha}} \|Dw(v,\cdot)\|_\infty \bigg)
\label{PREAL_CIR_SUPER_CRITICAL_NEXT_STEP}
\end{eqnarray}

\noindent where we used the direct inequality $\bd_{2 \eta, s, t}(\bP_1, \bP_2) \leq \bd_{(1-\lambda) 2\eta, s, t}(\bP_1, \bP_2)$. Remark also that similarly to \eqref{CTR_GRAD_ALPHA_SUPER_CRITICAL_APP}, using directly \eqref{new:estimate:H:operator}, as soon as $\eta+\alpha>1$ and $\gamma_2+\alpha>1$, one has
\begin{align}
|Dw(s, x)|&\le C\Big(\int_s^{t}  (v-s)^{-\frac{1}{\alpha}}\bd_{2\eta}(\mathbf{T}(\bP_1)(v),\mathbf{T}(\bP_2)(v)) (t-v)^{-\frac{1-\gamma_2}{\alpha}} \, dv\notag\\
&\quad \quad +\int_s^t  (v-s)^{-\frac 1\alpha+\frac{2\eta}{\alpha}}\big(\|w(v, .)\|_\infty+\|Dw(v,.)\|_\infty\big) \, dv\Big)\notag\\
&\le C\Big(\int_s^{t}  (v-s)^{-\frac{1}{\alpha}}\bd_{2\eta, s, v}(\mathbf{T}(\bP_1),\mathbf{T}(\bP_2)) (t-v)^{-\frac{1-\gamma_2}{\alpha}} \, dv\notag\\
&\quad \quad +\int_s^t  (v-s)^{-\frac 1\alpha+\frac{2\eta}{\alpha}}\big(\|w(v, .)\|_\infty+\|Dw(v,.)\|_\infty\big) \, dv\Big)\notag\\
&\le C  \bd_{(1-\lambda)2\eta, s, t}(\bP_1,\bP_2) \int_s^{t} (v-s)^{ 1-(\frac{2}{\alpha}-\frac{2\eta}{\alpha})} (t-v)^{-\frac{1-\gamma_2}{\alpha}} \, dv\notag\\
&\quad \quad + C \int_s^t  (v-s)^{-\frac 1\alpha+\frac{2\eta}{\alpha}}\big(\|w(v, .)\|_\infty+\|Dw(v,.)\|_\infty\big) \, dv\notag\\
&\le C\Big(\bd_{(1-\lambda) 2\eta,s,t}(\bP_1,\bP_2) (t-s)^{2-(\frac 3\alpha-\frac{2\eta}{\alpha} - \frac{\gamma_2}{\alpha}) }   +\int_s^t  (v-s)^{-\frac 1\alpha+\frac{2\eta}{\alpha}}\big(\|w(v,.)\|_{\infty}+\|Dw(v,\cdot)\|_\infty\big) \, dv\Big),\label{CTR_GRAD_ALPHA_SUPER_CRITICAL_NEXT_STEP}
\end{align}

\noindent for any $\lambda $ s.t. $2\lambda \eta + \alpha>1$ and where we used \eqref{inequality:first:iteration:T:operator} for the last but one inequality. We now proceed in a completely analogous manner as the first iteration. We thus conclude
$$
{\|w\|_\infty} \leq C\bd_{(1-\lambda)2 \eta, s, t}(\bP_1,\bP_2) (t-s)^{ 2-(\frac{2}{\alpha}-\frac{2\eta}{\alpha} - \frac{\gamma_2}{\alpha})}.
$$

We now select $\gamma_2=(1-\lambda) 2\eta$ with $\lambda $ large enough so that $2\lambda \eta + \alpha>1$ and small enough so that $\alpha+\gamma_2>1$. Note that this choice is licit under our current assumption $\alpha+\eta>1$ and that it also yields $2-(\frac{2}{\alpha}-\frac{2\eta}{\alpha} - \frac{\gamma_2}{\alpha})>0$. We next optimize the previous inequality on $w$ w.r.t $h \in \mathcal{C}^{\gamma_2}(\R^d)$
$$
\bd_{(1-\lambda)2\eta, s, t}(\mathscr{\mathbf T}^{(2)}(\bP_1),\mathscr{\mathbf T}^{(2)}(\bP_2))  \leq C\bd_{(1-\lambda)2 \eta, s, t}(\bP_1,\bP_2) (t-s)^{ 2-(\frac{2}{\alpha}-\frac{2\eta}{\alpha} - \frac{\gamma_2}{\alpha})}.
$$

This shows that $\mathscr{\mathbf T}^{(2)}$ is a contraction on the complete metric space $\mathcal{\boldsymbol A}_{s, s+T, \mu}$ w.r.t the distance $\bd_{(1-\lambda)2\eta, s, s+T}$ under the condition $\alpha+ \eta>1$ provided $T$ is sufficiently small. According to the Banach fixed point theorem, $\mathscr{\mathbf T}$ admits a unique fixed point which is the unique solution to the martingale problem on the time interval $[s, s+T]$. We eventually conclude following the same lines of reasonings as those employed at the end of the proof of Theorem \ref{THM_WP_NL}. We omit the remaining technical details.

\begin{remark}\label{remark:number:iterate:picard}
We point out that in the previous argument only two iterations of the map $\bold{T}$ were needed to derive our result. One may naturally ask if one can do better by performing more iteration of $\bold{T}$. This is not clear since the necessary condition $\alpha + \eta>1$ (coming from the regularization parameter $\vartheta$) already appears at the first iterate and seems to be the natural one if one wants to be able to compare the distance $\db_{\gamma, s, t}(\bold{T}(\bP_1, \bP_2))$ with the distance $\db_{2\eta, s, t}(\bP_1, \bP_2)$ in order to prove a contraction property.    
\end{remark}

\section*{Acknowledgment}
For the second and the third authors, the study has been funded by the Russian Science Foundation (project No 17-11-01098).

\bibliographystyle{alpha}
\bibliography{bibli}

\end{document}